\documentclass[12pt,twoside]{amsart}
\usepackage{amssymb}
\usepackage{amscd}

\title[MINIMAL MODELS AND ABUNDANCE]
{Minimal models and abundance for positive characteristic log surfaces} 
\author{Hiromu Tanaka} 
\address{Department of Mathematics, 
Graduate School of Science, 
Kobe University, 
Kobe, 657-8501, Japan} 
\email{tanakahi@math.kobe-u.ac.jp}

\newcommand{\Supp}[0]{{\operatorname{Supp}}}
\newcommand{\Pic}[0]{{\operatorname{Pic}}}
\newcommand{\mult}[0]{{\operatorname{mult}}}
\newtheorem{thm}{Theorem}[section]
\newtheorem{lem}[thm]{Lemma}
\newtheorem{cor}[thm]{Corollary}
\newtheorem{prop}[thm]{Proposition}
\newtheorem{ans}[thm]{Answer}
\newtheorem{ex}[thm]{Example}

\newtheorem{fact}[thm]{Fact}

\theoremstyle{definition}
\newtheorem{ques}[thm]{Question}
\newtheorem{dfn}[thm]{Definition}

\newtheorem{rem}[thm]{Remark}
\newtheorem*{ack}{Acknowledgments}      
\newtheorem{nota}[thm]{Notation}         
\newtheorem{step}{Step}

\newcommand{\pr}[1]{$\mathbb{P}^#1$}
\newcommand{\Q}{$\mathbb{Q}$}
\newcommand{\R}{$\mathbb{R}$}
\newcommand{\neb}{\overline{NE}}

\newcommand{\ff}{$\overline{\mathbb{F}}_p$}
\newcommand{\fff}{\overline{\mathbb{F}}_p}
\begin{document}

\maketitle

\begin{abstract}
We discuss the birational geometry of singular surfaces 
in positive characteristic. 
More precisely, we establish 
the minimal model program and the abundance theorem 
for \Q-factorial surfaces and 
for log canonical surfaces. 
Moreover, 
in the case where the base field is the algebraic closure of a finite field, 
we obtain the same results under much weaker assumptions. 
\end{abstract}

\tableofcontents


\section{Introduction}

In this paper, 
we consider the minimal model theory 
for surfaces with some singularities in positive characteristic. 
If the singularities are \Q-factorial or log canonical, then 
we establish the minimal model program and the abundance theorem. 
In the case where the base field is the algebraic closure of a finite field, 
we obtain the same result under much weaker assumptions. 
More precisely, we prove the following two theorems in this paper.

\begin{thm}[Minimal model program]\label{0mmp}
Let $X$ be a projective normal surface $X$, 
which is defined over an algebraically closed field $k$ of positive characteristic. 
Let $\Delta$ be an $\mathbb{R}$-divisor on $X$ and 
let $\Delta=\sum_{j\in J}\delta_j\Delta_j$ be its prime decomposition. 
Assume that one of the following conditions holds{\em{:}}
\begin{enumerate}
\item[(QF)]{$X$ is \Q-factorial and $0\leq \delta_j\leq 1$ for all $j\in J$.}
\item[(FP)]{$k=\overline{\mathbb{F}}_p$ and $0\leq\delta_j$ for all $j\in J$.}
\item[(LC)]{$(X, \Delta)$ is a log canonical surface.}
\end{enumerate}
Then, there exists a sequence of birational morphisms 
\begin{eqnarray*}
&(X, \Delta)=:(X_0, \Delta_0) \overset{\phi_0}\to 
(X_1, \Delta_1) \overset{\phi_1}\to \cdots 
\overset{\phi_{s-1}}\to (X_s, \Delta_s)=:(X^{\dagger}, \Delta^{\dagger})\\
&\,\,\,where\,\,\,(\phi_{i-1})_*(\Delta_{i-1})=:\Delta_i
\end{eqnarray*}
with the following properties. 
\begin{enumerate}
\item{Each $X_i$ is a projective normal surface.}
\item{Each $(X_i, \Delta_i)$ satisfies {\em (QF)}, {\em (FP)} 
or {\em (LC)} according as the above assumption.}
\item{For each $i$, ${\rm Ex}(\phi_i)=:C_i$ is a proper irreducible curve such that 
$$(K_{X_i}+\Delta_i)\cdot C_i<0$$
and that $C_i$ generates an extremal ray.}
\item{$(X^{\dagger}, \Delta^{\dagger})$ satisfies one of the following conditions.
\begin{enumerate}
\item{$K_{X^{\dagger}}+\Delta^{\dagger}$ is nef.}
\item{There is a surjective morphism $\mu:X^{\dagger} \to Z$ to 
a smooth projective curve $Z$ such that $\mu_*\mathcal O_{X^{\dagger}}=\mathcal O_Z$, 
$-(K_{X^{\dagger}}+\Delta^{\dagger})$ is $\mu$-ample and $\rho(X^{\dagger}/Z)=1$.}
\item{$-(K_{X^{\dagger}}+\Delta^{\dagger})$ is ample and 
$\rho(X^{\dagger})=1$.}
\end{enumerate}}
\end{enumerate}
In case $({\rm a})$, we say $(X^{\dagger}, \Delta^{\dagger})$ is a minimal model of $(X, \Delta)$.\\
In case $({\rm b})$ and $({\rm c})$, we say $(X^{\dagger}, \Delta^{\dagger})$ is a Mori fiber space. 
\end{thm}

\begin{thm}[Abundance theorem]\label{0abundance}
Let $X$ be a projective normal surface $X$, 
which is defined over an algebraically closed field $k$ of positive characteristic. 
Let $\Delta$ be an $\mathbb{R}$-divisor on $X$ and 
let $\Delta=\sum_{j\in J}\delta_j\Delta_j$ be its prime decomposition. 
Assume that one of the following conditions holds{\em{:}}
\begin{enumerate}
\item[(QF)]{$X$ is \Q-factorial and $0\leq \delta_j\leq 1$ for all $j\in J$.}
\item[(FP)]{$k=\overline{\mathbb{F}}_p$ and $0\leq\delta_j$ for all $j\in J$.}
\item[(LC)]{$(X, \Delta)$ is a log canonical surface.}
\end{enumerate}
If $K_X+\Delta$ is nef, then $K_X+\Delta$ is semi-ample. 
\end{thm}

Note that, if $X$ is a normal surface over $\overline{\mathbb{F}}_p$,
then $X$ is \Q-factorial (cf. Theorem~\ref{Q-factorial}). 
In particular, $K_X+\Delta$ is an \R-Cartier \R-divisor. 

In the case where the characteristic of the base field is zero, 
the above two theorems are proven by \cite{Fujino}. 
His proofs heavily depend on the Kodaira vanishing theorem and its generalizations. 
Unfortunately, in positive characteristic, 
there exist counter-examples to the Kodaira vanishing theorem (cf. \cite{Raynaud}). 
To prove the above two theorems, we use a result established in \cite{Keel}. 

In characteristic zero, 
the basepoint free theorem follows from the Kawamata--Viehweg vanishing theorem, 
which is a generalization of the Kodaira vanishing theorem (cf. \cite[Theorem~3.3]{KM}). 
Although we cannot use the Kodaira vanishing theorem, 
we can show the following basepoint free theorem. 

\begin{thm}[Basepoint free theorem]\label{0bpf}
Let $X$ be a projective normal \Q-factorial surface $X$, 
which is defined over an algebraically closed field $k$ of positive characteristic. 
Let $\Delta$ be a $\mathbb{Q}$-divisor. 
Let $\Delta=\sum_{j\in J}\delta_j\Delta_j$ be 
its prime decomposition and assume $0\leq \delta_j<1$ for all $j\in J$. 
Let $D$ be a nef Cartier divisor satisfying one of the following properties. 
\begin{enumerate}
\item{$D-(K_X+\Delta)$ is nef and big.} 
\item{$D-(K_X+\Delta)$ is semi-ample.}
\end{enumerate}
Then $D$ is semi-ample. 
\end{thm}

Although there exist counter-examples the Kodaira vanishing theorem, 
we can use the relative Kawamata--Viehweg vanishing theorem for birational morphisms of surfaces 
by \cite{KK}. 
Then, we obtain the following result on rational singularities. 

\begin{thm}\label{0rational}
Let $X$ be a projective normal surface $X$, 
which is defined over an algebraically closed field $k$ of positive characteristic. 
Let $\Delta$ be an $\mathbb{R}$-divisor on $X$. 
Let $\Delta=\sum_{j\in J}\delta_j\Delta_j$ be its prime decomposition and 
assume  $0\leq \delta_j\leq 1$ for all $j\in J.$ 
Assume that  $X$ has at worst rational singularities. 
Then, the following assertions hold.
\begin{enumerate}
\item{$X$ is \Q-factorial. In particular, by Theorem~\ref{0mmp}, 
we can run a $(K_X+\Delta)$-minimal model program$:$
\begin{eqnarray*}
&(X, \Delta)=:(X_0, \Delta_0) \overset{\phi_0}\to 
(X_1, \Delta_1) \overset{\phi_1}\to \cdots 
\overset{\phi_{s-1}}\to (X_s, \Delta_s)\\
&\,\,\,where\,\,\,(\phi_{i-1})_*(\Delta_{i-1})=:\Delta_i.
\end{eqnarray*}}
\item{Each $X_i$ has at worst rational singularities.}
\end{enumerate}
\end{thm}

\subsection{Overview of related literature}
We summarize some literature related to this paper with respect to 
the surface theory, the minimal model theory and Keel's result (Theorem~\ref{keel}). 

\subsubsection{Surface}
The Italian school established the classification theory for smooth algebraic surfaces, 
which was generalized by Kodaira, Shafarevich's seminar and Bombieri--Mumford 
(\cite{Mum}, \cite{BMII} and \cite{BMIII}). 
Theories of log surfaces and normal surfaces have been developed by 
Iitaka, Kawamata, Miyanishi, Sakai and many others. 
See, for example, 
\cite{Sakai} and \cite{Miyanishi}. 
\cite{Fujita} established the abundance theorem 
for pairs $(X, \Delta)$ where $X$ is a smooth projective surface 
and $\Delta$ is a \Q-boundary, that is, 
$\Delta$ is a \Q-divisor such that, for the prime decomposition $\Delta=\sum_{j\in J}\delta_j\Delta_j$, 
all the coefficients $\delta_j$ satisfy $0\leq\delta_j\leq 1$. 
In characteristic zero, \cite{Fujino} generalized this result. 
More precisely, \cite{Fujino} shows that 
the abundance theorem holds 
for pairs $(X, \Delta)$ where 
$X$ is a projective normal \Q-factorial surface and $\Delta$ is an \R-boundary. 
In this paper, we generalize this result to positive characteristic. 

\subsubsection{Minimal model theory}
In characteristic zero, 
the minimal model theory has been developed 
by Kawamata, Koll\'ar, Mori, Shokurov and many others. 
See, for example, \cite{KM} and \cite{KMM}. 
To establish fundamental theorems in minimal model theory, 
we use the Kodaira vanishing theorem and its generalization. 

But, in positive characteristic, 
there exist counter-examples to the Kodaira vanishing theorem 
even in the case where the dimension is two (cf. \cite{Raynaud}). 
In \cite{T}, the author established a weak Kodaira vanishing theorem 
for positive characteristic surfaces and 
established a basepoint free theorem for klt surfaces. 
\cite{Kollar} established the contraction theorem for smooth threefolds in positive characteristic. 
\cite{Kawamata} established the minimal model program for 
semi-stable threefolds in positive characteristic. 

\cite{Fujino} established the minimal model theory 
for \Q-factorial surfaces and log canonical surfaces in characteristic zero. 
In this paper, we generalize this result to positive characteristic. 

\subsubsection{Keel's result} 
Keel's result (Theorem~\ref{keel}) is a key in this paper. 
Thus, we want to summarize some literature, related to Keel's result. 
Theorem~\ref{keel} is the surface version of \cite[Theorem~0.2]{Keel}. 
Keel's proof depends on the Frobenius maps and the theory of the algebraic spaces. 
Note that Keel's result (Theorem~\ref{keel}) holds only in positive characteristic 
(cf. \cite[Section~3]{Keel}). 
For alternative proofs of \cite[Theorem~0.2]{Keel}, 
see \cite{CMM} and \cite{FT}. 
The proofs of \cite{CMM} and \cite{FT} do not depend on the theory of algebraic spaces. 
\cite{FT} only considers the case of surfaces. 

\cite{Keel} also shows the basepoint free theorem for \Q-factorial threefolds 
over $\overline{\mathbb F}_p$ with non-negative Kodaira dimension. 
Over $\overline{\mathbb F}_p$, we can often obtain some strong results. 
The reason is due to Corollary~\ref{pic}. 
See also \cite{Artin}, \cite{Keel2}, \cite{Masek} and \cite{Totaro}.

\subsection{Overview of contents}
In Section~2, we summarize the notations and two known results: 
Keel's result (Theorem~\ref{keel}) and Fact~\ref{principle}, 
which play crucial roles in this paper. 

In Section~3, we prove the case (QF) of Theorem~\ref{0mmp} and Theorem~\ref{0abundance}. 
To show the case (QF) of Theorem~\ref{0mmp}, 
we establish the cone theorem and the contraction theorem. 
The cone theorem follows from Mori's Bend and Break and the minimal resolution. 
We consider the Bend and Break for proper normal surfaces in Section~3.1. 
The contraction theorem (Theorem~\ref{cont}) is obtained by Keel's result (Theorem~\ref{keel}). 

To show the case (QF) of Theorem~\ref{0abundance}, 
we divide the argument into the two cases: 
$k\neq \overline{\mathbb F}_p$ and $k=\overline{\mathbb F}_p$. 
Since we treat the case $k=\overline{\mathbb F}_p$ in Section~4, 
we prove Theorem~\ref{0abundance} only for the case $k\neq\overline{\mathbb F}_p$. 
By a standard argument, we may assume that $\Delta$ is a \Q-divisor (Section~9). 
First, we prove $\kappa:=\kappa(X, K_X+\Delta)\geq 0$ (Theorem~\ref{nonv}). 
This follows from the same argument as \cite{Fujino}. 
Second, we consider the three cases: $\kappa=0$, $\kappa=1$ and $\kappa=2$. 
If $\kappa=1$, then the assertion follows from a more general known result (Proposition~\ref{kappa1}). 
By using Keel's result and the contraction theorem, 
we can prove the case of $\kappa=2$ (Proposition~\ref{nefbig}). 
In the case where $\kappa=0$  (Theorem~\ref{qabundancekappazero}), 
we use the arguments in \cite{Fujino} and \cite{Fujita}, which depends on the classification of smooth surfaces. 

In Section~4, 
we prove the case (FP) of Theorem~\ref{0mmp} and Theorem~\ref{0abundance}. 
The case (FP) of Theorem~\ref{0mmp} follows from Keel's result (Theorem~\ref{keel}) and Corollary~\ref{pic}. 
The proof of the case (FP) of Theorem~\ref{0abundance} is almost all the same as \cite{Masek}. 
In Section~4, we also show that normal surfaces over $\overline{\mathbb F}_p$ are \Q-factorial (Section~4.3). 

In Section~5, 
we consider the case (LC) of 
Theorem~\ref{0mmp} and Theorem~\ref{0abundance}. 
To show the case (LC) of Theorem~\ref{0mmp}, 
we describe the log canonical surface singularities by using results obtained in Section~3. 
This is discussed in Section~5.1. 
The case (LC) of Theorem~\ref{0abundance} follows from a known result (cf. \cite{Fujita}).

In Section~6, 
we generalize the results in Section~3, Section~4 and Section~5 to the relative situations. 
The relative version of Theorem~\ref{0mmp} and Theorem~\ref{0abundance} are 
Theorem~\ref{relmmp} and Corollary~\ref{relabundance}, respectively. 

In Appendix A and Appendix B, 
we prove Theorem~\ref{0bpf} and Theorem~\ref{0rational}. 
Note that, in characteristic zero, 
the basepoint free theorem holds for log canonical varieties (cf. \cite[Theorem~13.1]{Fujino1}). 
Its proof heavily depends on the Kodaira vanishing theorem and its generalizations. 
Although, in positive characteristic, there exist counter-examples to the Kodaira vanishing theorem (cf. \cite{Raynaud}), 
we can establish the basepoint free theorem for surfaces (Theorem~\ref{0bpf}). 
Our proof depends on the Keel's result (Theorem~\ref{keel}), the classification of smooth surfaces and 
the Riemann--Roch theorem. 
Theorem~\ref{0rational} follows from the relative Kawamata--Viehweg vanishing theorem 
for birational morphisms of surfaces.

\begin{ack}
The author would like to 
thank Professor Osamu Fujino and 
Professor Kenji Matsuki for many comments and discussions. 
He wishes to thank Professor Lucian B\u{a}descu 
for many valuable comments. 
He thanks Professor Atsushi Moriwaki 
for warm encouragement. 
He also thanks the referee for many valuable comments and suggestions. 
The author is partially supported by JSPS Fellowships for Young Scientists. 
\end{ack}

\section{Notation and known results}

\subsection{Notation}
We will freely use the notation and terminology in \cite{KM}. 

We will not distinguish the notations line bundles, invertible sheaves and Cartier divisors. 
For example, we will write $L+M$ for line bundles $L$ and $M$. 

Throughout this paper, 
we work over an algebraically closed field $k$, 
whose characteristic ${\rm char}\,k=:p$\,\, 
supposed to be positive 
unless otherwise mentioned.

In this paper, a {\em variety} means 
an integral scheme which is separated and of finite type over $k$. 
A {\em curve} or a {\em surface} means a variety 
whose dimension is one or two, respectively. 

Let $D$ be an \R-divisor and 
let $D=\sum_{j\in J} d_jD_j$ be its prime decomposition. 
For a real number $a$, we define 
$D\geq a$ by $d_j\geq a$ for all $j\in J$. 
We define $D\leq a$, $D> a$ and $D<a$ in the same way. 

We say $D$ is an {\em $\mathbb{R}$-boundary} 
(resp. a {\em \Q-boundary}) if 
$D$ is an $\mathbb{R}$-divisor (resp. a \Q-divisor) and 
$0\leq D\leq 1$. 

\begin{dfn}[Semi-ample \R-divisors]
Let $\pi:X\to S$ be a proper morphism between varieties. 
Let $D$ be an \R-Cartier \R-divisor. 
We say $D$ is {\em $\pi$-semi-ample} if 
$$D=\sum_{1\leq i\leq N} d_iD_i$$
where $d_i\in\mathbb R_{\geq 0}$ and 
$D_i$ is a $\pi$-semi-ample Cartier divisor for every $i$. 
For more details, see \cite[Section~4]{Fujino1}.
\end{dfn}

\subsection{Known results}

\subsubsection{Keel's result}

In positive characteristic, 
we can not use the Kodaira vanishing theorem. 
But we can use the following theorem by Keel. 
The following assertion is 
the surface version of the original theorem by Keel. 
For an alternative proof, see \cite{CMM} and \cite[Section~2]{FT}.  

\begin{thm}[Keel's result]\label{keel}
Let $X$ be a projective normal surface over 
an algebraically closed field $k$ of positive characteristic. 
Let $L$ be a nef and big line bundle. 
Let $E(L)$ be the reduced subscheme whose support is 
the union of all the curves $C$ with $L\cdot C=0$. 
Then, $L$ is semi-ample iff $L|_{E(L)}$ is semi-ample. 
\end{thm}

\begin{proof}
See \cite[Theorem 0.2]{Keel}.
\end{proof}

\subsubsection{Difference between $k\neq\overline{\mathbb{F}}_p$ and 
$k=\overline{\mathbb{F}}_p$}
In this paper, we often divide the argument 
into the two cases: (1)$k \neq \fff$ and (2)$k=\fff$. 
The reason for this comes from the following fact. 

\begin{fact}\label{principle}
Let $k$ be an algebraically closed field of arbitrary characteristic.\\
$(1)$If $k \neq \fff$, all abelian varieties over $k$ have infinite rank.\\
$(2)$If $k=\fff$, all group schemes 
of finite type over $\overline{\mathbb{F}}_p$ are torsion groups. 
\end{fact}

\begin{proof}
(1)See \cite[Theorem 10.1]{FJ}.\\
(2)Let $X$ be a group scheme of finite type 
over $\overline{\mathbb{F}}_p$. 
Let $P$ be a closed point of $X$. 
Then we see that 
$P$ and $X$ are defined over a finite field. 
Thus we can consider $P$ as a rational point of 
a group scheme of finite type over a finite field. 
Since this group is finite, $P$ is a torsion. 
\end{proof}

As a corollary, 
we obtain the following information on the line bundles of 
varieties over $\overline{\mathbb{F}}_p$. 

\begin{cor}\label{pic}
Let $X$ be a projective variety over $\overline{\mathbb{F}}_p$ 
and let $D$ be a Cartier divisor. 
If $D \equiv 0$, then $D$ is a torsion in $\Pic\,X$. 
\end{cor}

\begin{proof}
Consider the Picard space of $X$ and apply Fact~\ref{principle}. 
For more details, see \cite[Lemma~2.16]{Keel}.
\end{proof}

\section{\Q-factorial surfaces}

\subsection{Bend and Break}

In this section, we consider Mori's Bend-and-Break 
for proper normal surfaces. 
We use the following intersection theory for normal surfaces by \cite{Mumford1}.

\begin{dfn}[Intersection theory by Mumford]
Let $X$ be a normal surface and 
let $f: X' \to X$ be a resolution of singularities. 
Let $E_1, \cdots, E_n$ be the exceptional curves of $f$. 
Let $C$ be a proper curve in $X$ and 
let $D$ be an \R-divisor on $X$. 
Let $C'$ and $D'$ be their proper transforms respectively. 
\begin{enumerate}

\item{We define $f^*D:=D'+\sum e_iE_i$ where all $e_i$ are real numbers 
uniquely determined by the linear equations $(D'+\sum e_iE_i)\cdot E_j=0$ for $j=1, \cdots, n$. 
Note that the intersection matrix $(E_i\cdot E_j)$ is negative definite (cf. \cite[Lemma~3.40]{KM}).}

\item{We define the intersection pairing by $C\cdot D:=f^*C\cdot f^*D =C'\cdot f^*D$.}

\item{If $X$ is proper, then we can 
naturally extend this intersection theory to Weil divisors 
with $\mathbb{Q}$ or $\mathbb{R}$ coefficients by linearity.}
\end{enumerate}
\end{dfn}

\begin{dfn}
Let $X$ be a proper normal surface and 
let $D$ and $D'$ be $\mathbb{R}$-divisors. 
We denote $D\equiv_{\rm Mum} D'$ if $D\cdot C=D'\cdot C$ for every curve $C$ in $X$.
\end{dfn}

Let us define the Bend-and-Break. 
This is the key for the proof of the cone theorem. 

\begin{dfn}[Bend-and-Break]\label{defbb}
Let $X$ be a proper normal surface. 
We say $X$ satisfies {\em Bend-and-Break} 
if $X$ satisfies the following two conditions (BBI) and (BBII):
\begin{enumerate}
\item[(BBI)]{
If $Z$ is a rational curve in $X$, 
then $Z \equiv_{\rm Mum} Z_1+ \dots +Z_r$, 
where each $Z_i$ is a rational curve and $-Z_i\cdot K_X \leq 3$.}

\item[(BBII)]{
Let $C$ be a curve in $X$ with $C\cdot K_X<0$. 
Then for an arbitrary point $c_0 \in C \setminus {\rm Sing}~X$, 
there exists a positive integer $n(X, C, c_0)$ with the following conditions. 
For an arbitrary positive integer $n$ with $n \geq n(X, C, c_0)$, 
there exist a non-negative integer $\alpha_n$, 
a curve $C_n$ and an effective 1-cycle $Z_n$ 
with the following four conditions:
\begin{enumerate}
\item{$p^n C \equiv_{\rm Mum} \alpha_nC_n + Z_n$.}
\item{$Z_n = Z_{n,1} + \dots + Z_{n,r_n}$, 
where each $Z_{n,i}$ is a rational curve.} 
\item{$-\alpha_nC_n\cdot K_X \leq 2g(C_{\rm normal})$ 
where $C_{\rm normal}$ is the normalization of $C$.} 
\item{$c_0 \in \Supp Z_n$.} 
\end{enumerate}
}
\end{enumerate}
\end{dfn}

The smooth case is the original Bend-and-Break proven by Mori.

\begin{prop}
If $X$ is a projective smooth surface. 
then $X$ satisfies Bend-and-Break. 
\end{prop}

\begin{proof}
See \cite[Theorem~4 and Theorem~5]{Mori1} and their proofs.
\end{proof}

Using this result, 
we extend the Bend-and-Break to the proper normal surfaces.

\begin{prop}\label{propbb}
If $X$ is a proper normal surface, 
then $X$ satisfies Bend-and-Break. 
\end{prop}

\begin{proof}
Let $f:X' \to X$ be the minimal resolution and 
$K_{X'}+\sum{e_i E_i}=f^*(K_X)$ where 
all $E_i$ are exceptional curves and $e_i \geq 0$.

(BBI): Let $Z$ be a rational curve in $X$ and 
let $Z'$ be its proper transform. 
Here, $Z'$ is rational. 
Since $X'$ is smooth, $X'$ satisfies (BBI). 
Therefore, $Z' \equiv Z'_1+ \dots +Z'_r$ and 
all $Z_i'$ are rational curves and $-Z_i'\cdot K_{X'} \leq 3$. 
Apply $f_*$ to this equation. 
We obtain that $Z \equiv_{\rm Mum} Z_1+ \dots +Z_r$, where $f_*Z'_i=Z_i$. 
Note that $Z_i$ may be zero. 
But if all of $Z_i$ are zero, then we have $Z \equiv_{\rm Mum} 0$. 
This is a contradiction. 
Moreover, the above relation between $K_X$ and $K_{X'}$ shows that 
$-Z_i\cdot K_X \leq 3$.

(BBII): Let $C$ be a curve in $X$ with $C\cdot K_X<0$ and 
let $C'$ be its proper transform. 
We see $C'\cdot K_{X'}<0$ from the above relation between canonical divisors. 
Let $c_0$ be an arbitrary element of $C \setminus {\rm Sing} X$ and 
let $c_0'$ be a point of $C'$ such that $f(c_0')=c_0$. 
Since $X'$ is smooth, $X'$ satisfies (BBII). 
Thus we obtain $n(X', C', c_0')$, 
$\alpha_n'$, $C_n'$ and $Z_n'$. 
Let $n(X, C, c_0):=n(X', C', c_0')$, 
$\alpha_nC_n:=f_*(\alpha_n'C_n')$ and $Z_n:=f_*(Z_n')$. 
It is easy to see that these satisfy (BBII). 
\end{proof}

From now on, let us generalize this result for pairs $(X, \Delta)$.

\begin{dfn}[($K_X+\Delta$)-Bend-and-Break]\label{deflogbb}
Let $X$ be a proper normal surface and 
let $\Delta$ be an effective $\mathbb{R}$-divisor. 
Let $\Delta=\sum b_iB_i$ be its prime decomposition. 
We say $(X, \Delta)$ satisfies $(K_X+\Delta)$-Bend-and-Break if 
$X$ and $\Delta$ satisfy the following two conditions (BBI) and (BBII):
\begin{enumerate}
\item[(BBI)]{
There exists a positive integer $L(X, \Delta)$ satisfying (1) and (2).
\begin{enumerate}
\item[(1)]{If $Z$ is a rational curve in $X$, then 
$Z \equiv_{\rm Mum} Z_1+ \dots +Z_r$, where 
all $Z_i$ are rational curves 
and $-Z_i\cdot (K_X+\Delta) \leq L(X, \Delta)$.}

\item[(2)]{If $B_i^2<0$, then $-B_i\cdot (K_X+\Delta) \leq L(X, \Delta)$.}
\end{enumerate}
}

\item[(BBII)]{
Let $C$ be a curve in $X$ with 
$C\cdot (K_X+\Delta)<0$ and $C\neq B_i$ for all $i$ such that $B_i^2<0$. 
Then, for an arbitrary point $c_0 \in C \setminus {\rm Sing} X$, 
there exists a positive integer $n(X, \Delta, C, c_0)$ 
with the following conditions. 
For an arbitrary integer $n$ with $n \geq n(X, \Delta, C, c_0)$, 
there exist a non-negative integer $\alpha_n$, a curve $C_n$ 
and an effective 1-cycle $Z_n$ with the following four conditions:
\begin{enumerate}
\item{$p^n C \equiv_{\rm Mum} \alpha_n C_n + Z_n$.}
\item{$Z_n = Z_{n,1} + \dots + Z_{n,r_n}$, where 
all $Z_{n,i}$ are rational curves.}
\item{$-\alpha_nC_n\cdot (K_X+\Delta) \leq 2g(C_{\rm normal})$ where 
$C_{\rm normal}$ is the normalization of $C$ or 
$C_n=B_i$ for some $i$ such that $B_i^2<0$.}
\item{$c_0 \in \Supp Z_n$.}
\end{enumerate}
}
\end{enumerate}
\end{dfn}

We obtain the following main result in this section.

\begin{thm}\label{theoremlogbb}
If $X$ is a proper normal surface 
and $\Delta$ is an effective $\mathbb{R}$-divisor, 
then $(X, \Delta)$ satisfies $(K_X+\Delta)$-Bend-and-Break.
\end{thm}

\begin{proof}
We write the prime decomposition $\Delta=\sum{b_iB_i}$. 

(BBI):
Let 
$$L(X, \Delta):=\max(\{3\} \cup \{-(K_X+\Delta)\cdot B_{\mu}\})$$
where $B_{\mu}$ ranges over the prime components of $\Delta$ with $B_{\mu}^2<0.$ 
We check the conditions (1) and (2).
But, $(2)$ is obvious. 
Thus, let us prove (1). 
Let $Z$ be a rational curve in $X$. 
By Proposition~\ref{propbb}, 
we have $Z \equiv_{\rm Mum} Z_1+ \dots +Z_r$ 
where any $Z_j$ is rational and 
satisfies $-Z_j\cdot K_X \leq 3$. 
If $Z_j=B_{\mu}$ with $B_{\mu}^2<0$, 
then we obtain $-Z_j\cdot (K_X+\Delta) \leq L(X, \Delta)$. 
If $Z_j \neq B_{\mu}$, then we have 
$$-Z_j\cdot (K_X+\Delta) \leq -Z_j\cdot K_X \leq 3 \leq L(X, \Delta).$$

(BBII):
Let $C$ be a curve in $X$ with $C\cdot (K_X+\Delta)<0$ and 
with $C\neq B_i$ for all $B_i$ such that $B_i^2<0$. 
Then we obtain the following inequalities: 
$$C\cdot K_X \leq C\cdot (K_X+\Delta)<0.$$ 
By Proposition~\ref{propbb}, we can use the Bend-and-Break in the sense of Definition~\ref{defbb}. 
Let 
\begin{eqnarray*}
n(X, \Delta, C, c_0)&:=&n(X, C, c_0). 
\end{eqnarray*}
These satisfies the four conditions of (BBII) 
of Definition~\ref{deflogbb}. 
Indeed, (a)(b)(d) are obvious. 
We consider (c). 
If $C_n \neq B_i$ for all $B_i$ such that $B_i^2<0$, 
then we have 
$$-\alpha_nC_n\cdot (K_X+\Delta) \leq -\alpha_nC_n\cdot K_X 
\leq 2g(C_{\rm normal}).$$ 
This completes the proof. 
\end{proof}

Let us calculate $L(X, \Delta)$ in the case where $\Delta$ is an $\mathbb R$-boundary. 

\begin{prop}\label{bblength3}
Let $X$ be a proper normal surface 
and let $\Delta$ be an $\mathbb{R}$-boundary. 
Then, $(X, \Delta)$ satisfies $(K_X+\Delta)$-Bend-and-Break 
for $L(X, \Delta)=3$.
\end{prop}

\begin{proof}
By the proof of Theorem~\ref{theoremlogbb}, 
$(X, \Delta)$ satisfies $(K_X+\Delta)$-Bend and Break for 
$$L(X, \Delta)=\max(\{3\} \cup \{-(K_X+\Delta)\cdot B_{\mu}\})$$
where $B_{\mu}$ ranges over the prime components of $\Delta$ with $B_{\mu}^2<0.$ 
Thus, the assertion follows from the following lemma. 
\end{proof}

\begin{lem}\label{negative-length2}
Let $X$ be a normal surface and 
let $\Delta$ be an $\mathbb R$-boundary. 
If $C$ be a proper curve in $X$ such that $C^2\leq 0$, then $-(K_X+\Delta)\cdot C\leq 2.$
\end{lem}

\begin{proof}
Let $f:Y\to X$ be the minimal resolution and 
let $C_Y$ be the proper transform of $C$. 
We define $\Delta_Y$ by 
$$K_Y+C_Y+\Delta_Y=f^*(K_X+C).$$
Note that $\Delta_Y\geq 0$ and $C_Y\not\subset \Supp \Delta_Y.$ 
Then, we see $(K_Y+C_Y)\cdot C_Y\geq -2.$ 
We obtain 
\begin{eqnarray*}
(K_X+\Delta)\cdot C&\geq&(K_X+C)\cdot C\\
&=&f^*(K_X+C)\cdot C_Y\\
&=&(K_Y+C_Y+\Delta_Y)\cdot C_Y\\
&\geq& (K_Y+C_Y)\cdot C_Y\\
&\geq& -2.
\end{eqnarray*}
\end{proof}

\subsection{Cone theorem}
In this section we prove the cone theorem. 
We use the Bend-and-Break in the sense of Definition~\ref{deflogbb}. 
Thus, in this section, 
we use the notations in Definition~\ref{deflogbb}. 

Here, let us recall the definition of the Kleiman--Mori cone. 

\begin{dfn}
Let $X$ be a projective variety. 
Then we define  
\begin{eqnarray*}
N(X)&:=&\{r_1Z_1+\cdots+r_sZ_s\,|\,r_i\in \mathbb{R}\,\,
{\rm and}\,\,Z_i\,\,{\rm is\,\,a\,\, curve\,\,in\,\,}X\}/{\equiv}\\
NE(X)&:=&\{[r_1Z_1+\cdots+r_sZ_s]\,|\,r_i\geq 0,Z_i\,\,{\rm is\,\, a\,\,curve\,\,in\,\,}X\}\subset N(X)
\end{eqnarray*}
where $[r_1Z_1+\cdots+r_sZ_s]$ means the equivalence class.
\end{dfn}

Note that $N(X)$ is the quotient space by "$\equiv$". 
Although we often use the intersection theory by Mumford, 
we do NOT take the quotient by "$\equiv_{\rm Mum}$". 
The numerical equivalence "$\equiv$" is induced 
by the intersections only with $\mathbb{R}$-Cartier divisors.

In this section, we use the following lemma repeatedly. 

\begin{lem}\label{easy-lemma}
Let $a, b\in \mathbb R$ and $c, d\in \mathbb R_{>0}.$
Then, 
$$\frac{a+b}{c+d}\leq\max\{\frac{a}{c},\frac{b}{d}\}.$$
\end{lem}

\begin{proof}
The proof is easy. Thus, we omit the proof. 
\end{proof}

The following lemma is a key in this section. 

\begin{lem}\label{keylemma}
Let $X$ be a projective normal surface and 
let $\Delta$ be an effective $\mathbb{R}$-divisor 
such that $K_X+\Delta$ is $\mathbb{R}$-Cartier. 
Let $\Delta=\sum b_iB_i$ be the prime decomposition. 
Let $H$ be an $\mathbb{R}$-Cartier ample $\mathbb{R}$-divisor. 
If $C$ is a curve in $X$ such that $C\cdot (K_X+\Delta)<0$, then 
there exists a curve $E$ in $X$ with the following properties. 
\begin{enumerate}
\item{$E$ is rational or $E=B_j$ for some $j$ such that $B_j^2<0$.}
\item{$0<-E\cdot (K_X+\Delta) \leq L(X, \Delta)$}.
\item{
$$\frac{-C\cdot (K_X+\Delta)}{C\cdot H} \leq \frac{-E\cdot (K_X+\Delta)}{E\cdot H}.$$
}
\end{enumerate}
\end{lem}

This proof is very similar to \cite[Theorem~1.13]{KM}. 

\begin{proof}
In this proof, we use the notation (BBI) and (BBII) in the sense of Definition~\ref{deflogbb}. 
First, if $C=B_j$ with $B_j^2<0$, 
then the assertion is obvious. 
We may assume that $C\neq B_j$ for all $B_j$ with $B_j^2<0$. 
Then, we can use (BBII). 
But, since we do not use $c_0$, we fix $c_0\in C\setminus {\rm Sing} X$. 
Set $C_n':=\alpha_nC_n$. 
We consider the following number 
\begin{eqnarray*}
M&:=&\frac{-C\cdot (K_X+\Delta)}{C\cdot H}=\frac{-p^nC\cdot (K_X+\Delta)}{p^nC\cdot H}\\
&=&\frac{-C_n'\cdot (K_X+\Delta)-Z_n\cdot (K_X+\Delta)}{C_n'\cdot H+Z_n\cdot H}
=\frac{a_n+b_n}{c_n+d_n}
\end{eqnarray*}
where $a_n$, $b_n$, $c_n$ and $d_n$ are defined by 
\begin{eqnarray*}
a_n&:=&-C_n'\cdot (K_X+\Delta)\\
b_n&:=&-Z_n\cdot (K_X+\Delta)\\
c_n&:=&C_n'\cdot H\\
d_n&:=&Z_n\cdot H.
\end{eqnarray*}

\setcounter{step}{0}

\begin{step}
In this step, we reduce the proof to the case where $\alpha_n>0$ for all $n \gg 0$. 

Assume that there is a positive integer $n$ such that 
$n\geq n(X, \Delta, C, c_0)$ and $\alpha_n=0$. 
Then we have 
\begin{eqnarray*}
\frac{-C\cdot (K_X+\Delta)}{C\cdot H}=\frac{-Z_n\cdot (K_X+\Delta)}{Z_n\cdot H}
\leq \frac{-Z_{n,i}\cdot (K_X+\Delta)}{Z_{n,i}\cdot H}
\end{eqnarray*}
for some $i$ by Lemma~\ref{easy-lemma}. 
Moreover, by (BBI) and Lemma~\ref{easy-lemma}, we obtain the desired result. 
\end{step}

\begin{step}
In this step, we reduce the proof to the case where 
$$a_n=-\alpha_nC_n\cdot (K_X+\Delta)\leq 2g(C_{\rm normal})$$ 
for all $n \gg 0$. 

Suppose the contrary. Then, by (c) of (BBII), we obtain $C_n=B_j$ for some $j$ such that $B_j^2<0$. 
By Lemma~\ref{easy-lemma}, we have the following equality 
\begin{eqnarray*}
\frac{-C\cdot (K_X+\Delta)}{C\cdot H}&=&
\frac{-\alpha_nB_j\cdot (K_X+\Delta)-Z_n\cdot (K_X+\Delta)}
{\alpha_nB_j\cdot H+Z_n\cdot H}\\
&\leq&\max\left\{ \frac{-B_j\cdot (K_X+\Delta)}{B_j\cdot H}, 
\frac{-Z_n\cdot (K_X+\Delta)}{Z_n\cdot H}\right\}.\end{eqnarray*}
If 
$$\frac{-C\cdot (K_X+\Delta)}{C\cdot H} \leq
\frac{-B_j\cdot (K_X+\Delta)}{B_j\cdot H},$$
then this is the desired result. 
If 
$$\frac{-C\cdot (K_X+\Delta)}{C\cdot H} 
\leq \frac{-Z_n\cdot (K_X+\Delta)}{Z_n\cdot H}, $$
then, by (BBI) and Lemma~\ref{easy-lemma}, we obtain the desired result. 
\end{step}

From now on, 
we consider the asymptotic behaviors of $a_n$, $b_n$, $c_n$ and $d_n$.

\begin{step}
The sequence $a_n$ is bounded and $b_n$ is not bounded. 

Indeed, the boundedness of $a_n$ follows from Step~2. 
Since $a_n+b_n=-p^nC\cdot (K_X+\Delta)$ is not bounded, 
$b_n$ is not bounded. 
\end{step}

\begin{step}
In this step, we prove that for 
an arbitrary positive real number $\epsilon$, 
there exists a curve $E$ in $X$ with the following properties. 
\begin{enumerate}
\item[$(1)'$]{$E$ is rational.}
\item[$(2)'$]{$0<-E\cdot (K_X+\Delta)\leq L(X, \Delta)$.}
\item[$(3)'$]{$$M-\epsilon<\frac{-E\cdot (K_X+\Delta)}{E\cdot H}.$$}
\end{enumerate}

If ${a_n}/{c_n}<M$ for some $n \gg 0$, then we have ${b_n}/{d_n}\geq M$, 
which gives us the desired result by (BBI) and Lemma~\ref{easy-lemma}. 
Thus, we may assume that ${a_n}/{c_n}\geq M$ 
for all $n \gg 0$.
Then, since $a_n$ is bounded, 
so is $c_n$ because $M$ is a positive number. 
Because $c_n+d_n=p^nC\cdot H$, 
$d_n$ is not bounded. 
Therefore, for sufficiently large $n$, we obtain 
\begin{eqnarray*}
\frac{b_n}{d_n}+\epsilon>\frac{a_n+b_n}{d_n}>\frac{a_n+b_n}{c_n+d_n}=M. 
\end{eqnarray*}
By (BBI) and Lemma~\ref{easy-lemma}, there exists a rational curve $E$ with the desired properties. 
\end{step}

\begin{step}
We take an arbitrary positive real number $\epsilon$ 
with $0<\epsilon \leq M/2$. 
Then, by Step~4, we obtain 
\begin{eqnarray*}
E\cdot H<\frac{-E\cdot (K_X+\Delta)}{M-\epsilon} 
\leq \frac{L(X, \Delta)}{M/2}=\frac{2L(X, \Delta)}{M}
.\end{eqnarray*}
Since $H$ is ample, 
the following subset in numerical classes of 
effective 1-cycles in $X$ with integral coefficients
\begin{eqnarray*}
\{[E]|E\cdot H<\frac{2L(X, \Delta)}{M}\}
\end{eqnarray*}
has only finitely many members. 
Therefore, so is the following set 
\begin{eqnarray*}
\left\{\frac{-E\cdot (K_X+\Delta)}{E\cdot H}\middle|
E\cdot H<\frac{2L(X, \Delta)}{M}\,\,{\rm and}\,\,E\,\,{\rm satisfies }\,\,(1)'(2)'\right\}\end{eqnarray*}
because $K_X+\Delta$ is $\mathbb{R}$-Cartier. 
Take a sufficiently small $\epsilon>0$. 
Then, by Step~4, 
we obtain a rational curve $E$ in $X$ such that 
\begin{eqnarray*}
E\,\,{\rm satisfies}\,\,(1)'(2)'
\,\,{\rm and}\,\, 
\frac{-E\cdot (K_X+\Delta)}{E\cdot H}\geq M.\end{eqnarray*}
\end{step}
This completes the proof. 
\end{proof}

Let us prove the cone theorem. 

\begin{thm}[Cone theorem]\label{ct}
Let $X$ be a projective normal surface and 
let $\Delta$ be an effective $\mathbb{R}$ divisor 
such that $K_X+\Delta$ is $\mathbb{R}$-Cartier. 
Let $\Delta=\sum b_iB_i$ be the prime decomposition. 
Let $H$ be an $\mathbb{R}$-Cartier ample $\mathbb{R}$-divisor. 
Then the following assertions hold:

\begin{enumerate}

\item{
$
\neb(X)=
\neb(X)_{K_X+\Delta \geq 0}+
\sum{\mathbb{R}_{\geq 0}[C_i]}
$.}
\item{
$
\neb(X)=
\neb(X)_{K_X+\Delta+H \geq 0}+
\sum\limits_{\rm finite}{\mathbb{R}_{\geq 0}[C_i]}
$.}
\item{Each $C_i$ in $(1)$ and $(2)$ is rational or 
$C_i=B_j$ for some $B_j$ with $B_j^2<0$.}
\item{Each $C_i$ in $(1)$ and $(2)$ 
satisfies $0<-C_i\cdot (K_X+\Delta) \leq L(X, \Delta)$.}
\end{enumerate}
\end{thm}

This proof is essentially the same as \cite[Theorem~1.24]{KM}.

\begin{proof}
(1)Let $W$ be the right hand side in (1), i.e.
$$W:=\overline{NE}(X)_{(K_X+\Delta) \geq 0}+
\sum_{C_i\,\,{\rm satisfies}\,\,(3)(4)}\mathbb{R}_{\geq 0}[C_i].$$ 
Note that $W$ is a closed set 
by the same proof as in \cite[Ch III, Theorem 1.2]{Kollar2}. 
We would like to prove $\overline{NE}(X)=W$. 
The inclusion $\overline{NE}(X) \supset W$ is clear. 
Let us assume $\overline{NE}(X) \supsetneq W$ and 
derive a contradiction. 
Then we can find a Cartier divisor $D$ 
which is positive on $W\setminus{0}$ and 
which is negative on some element of $\overline{NE}(X)$. 
Let $\mu$ be a positive real number such that 
$H+\mu D$ is nef and $H+\mu'D$ is ample 
for all positive real numbers $\mu'$ with $\mu'<\mu$. 
Then we can take a 1-cycle $Z$ 
with $Z \in \overline{NE}(X)\setminus\{0\}$ and $(H+\mu D)\cdot Z=0$. 
Since $Z\cdot H>0$ means $Z\cdot D<0$, $Z$ is not in $W$. 
By the definition of $W$, we obtain $Z\cdot (K_X+\Delta)<0$. 
Because $Z$ is an element of $\overline{NE}(X)$, 
there exist effective 1-cycles $Z_k=\sum a_{k,j}Z_{k,j}$ such that 
the limit of $Z_k$ is $Z$. 
Take an arbitrary positive real number $\mu'$ with $\mu'<\mu$. 
By the ampleness of $H+\mu'D$, we have 
\begin{eqnarray*}
\max_j\frac{-Z_{k,j}\cdot (K_X+\Delta)}{Z_{k,j}\cdot (H+\mu'D)} \geq 
\frac{-Z_k\cdot (K_X+\Delta)}{Z_k\cdot (H+\mu'D)}.
\end{eqnarray*}
We may assume that the max on the left hand side occurs when $j$ is zero. 
By Lemma~\ref{keylemma}, we obtain 
\begin{eqnarray*}
\frac{-E_{k}\cdot (K_X+\Delta)}{E_{k}\cdot (H+\mu'D)} \geq 
\frac{-Z_{k,0}\cdot (K_X+\Delta)}{Z_{k,0}\cdot (H+\mu'D)} \geq 
\frac{-Z_k\cdot (K_X+\Delta)}{Z_k\cdot (H+\mu'D)}.
\end{eqnarray*}
Here, $E_k$ satisfies (3) and (4). 
Thus, we have $E_k \in W$ and 
this means $E_{k}\cdot D \geq 0$. 
Therefore we have 
\begin{eqnarray*}
\frac{-E_{k}\cdot (K_X+\Delta)}{E_{k}\cdot H} \geq 
\frac{-E_{k}\cdot (K_X+\Delta)}{E_{k}\cdot (H+\mu'D)} \geq 
\frac{-Z_k\cdot (K_X+\Delta)}{Z_k\cdot (H+\mu'D)}.
\end{eqnarray*}
Take a large positive number $r$ such that $rH+(K_X+\Delta)$ is ample. 
This shows 
\begin{eqnarray*}
r>\frac{-E_{k}\cdot (K_X+\Delta)}{E_{k}\cdot H}.
\end{eqnarray*}
Combining the inequalities, we obtain 
\begin{eqnarray*}
r>\frac{-Z_k\cdot (K_X+\Delta)}{Z_k\cdot (H+\mu'D)}.
\end{eqnarray*}
Recall that we choose $\mu'$ as an arbitrary positive real number with $\mu'<\mu$. 
By taking the limit $\mu'$ to $\mu$, we obtain 
\begin{eqnarray*}
r \geq \frac{-Z_k\cdot (K_X+\Delta)}{Z_k\cdot (H+\mu D)}.
\end{eqnarray*}
Moreover, by taking the limit $k$ to $\infty$, 
we obtain 
\begin{eqnarray*}
r \geq 
\lim_{k \to \infty}\frac{-Z_k\cdot (K_X+\Delta)}{(Z_k\cdot H+\mu D)}=
\frac{(\rm positive)}{+0}=+\infty.
\end{eqnarray*}
This is a contradiction. 
This completes the proof of (1).

(2)
If $C_i\cdot (K_X+\Delta+H)<0$, then we have 
\begin{eqnarray*}
C_i\cdot H<-C_i\cdot (K_X+\Delta) \leq L(X, \Delta).
\end{eqnarray*}
There are only finitely many numerical classes of curves like this. 
This shows (2). 
The remaining assertions (3) and (4) have already proven in the above arguments. 
\end{proof}

\begin{rem}\label{ray3}
In Theorem~\ref{ct}, $L(X, \Delta)$ gives a upper bound of length of extremal rays. 
By the proof of Theorem~\ref{theoremlogbb}, 
$$L(X, \Delta):=\max(\{3\} \cup \{-(K_X+\Delta)\cdot B_{\mu}\})$$
where $B_{\mu}$ ranges over the prime components of $\Delta$ with $B_{\mu}^2<0.$ 
In the case where $\Delta$ is an $\mathbb{R}$-boundary, 
we can set $L(X, \Delta)=3$ by Proposition~\ref{bblength3}.
\end{rem}

Moreover, in the case where $\Delta$ is an $\mathbb{R}$-boundary, 
every $(K_X+\Delta)$-negative extremal ray is generated by a rational curve. 

\begin{prop}\label{ext-rat-length1}
Let $X$ be a projective normal surface and 
let $\Delta$ be an $\mathbb{R}$-boundary such that $K_X+\Delta$ is $\mathbb{R}$-Cartier. 
If $R$ is a $(K_X+\Delta)$-negative extremal ray of $\overline{NE}(X)$, 
then $R=\mathbb R_{\geq 0}[C]$ where $C$ is a rational curve such that 
$-(K_X+\Delta)\cdot C\leq 3$.
\end{prop}

\begin{proof}
By Theorem~\ref{ct} and Remark~\ref{ray3}, 
we can write $R=\mathbb R_{\geq 0}[C]$ 
where $C$ is a curve such that $-(K_X+\Delta)\cdot C\leq 3$ and that 
$C$ is rational or $C^2<0.$ 
Assume $C^2<0$. 
Then, we obtain 
$$(K_X+C)\cdot C\leq (K_X+\Delta)\cdot C<0.$$
Thus, the assertion follows from the following lemma. 
\end{proof}

\begin{lem}\label{non-qcar-rational}
Let $X$ be a normal surface and 
let $C$ be a proper curve in $X$. 
If $(K_X+C)\cdot C<0$, then $C$ is a rational curve. 
\end{lem}

\begin{proof}
Let $f:Y\to X$ be the minimal resolution and 
let $C_Y$ be the proper transform of $C$. 
We define $\Delta_Y$ by $K_Y+C_Y+\Delta_Y=f^*(K_X+C).$ 
Then, we see 
$$(K_Y+C_Y)\cdot C_Y\leq (K_Y+C_Y+\Delta_Y)\cdot C_Y=(K_X+C)\cdot C<0.$$
Then, since $C_Y$ is a rational curve, so is $C$. 
\end{proof}

\subsection{Results on adjunction formula}
In this section, we summarize results on adjunction formula. 

\begin{prop}
Let $X$ be a projective normal surface and 
let $C$ be a curve in $X$. 
Then, there exists a exact sequence  
$$0\to \mathcal T \to \omega_X(C)|_C \to \omega_C \to 0$$
where $\mathcal T$ is the torsion subsheaf of $\omega_X(C)|_C$.
\end{prop}

\begin{proof}
See \cite[Lemma~4.4]{Fujino}.
\end{proof}

Using this adjunction formula, 
we obtain the following result on global sections. 

\begin{lem}\label{adjunctionlemma}
Let $X$ be a projective normal surface and 
let $C$ be a curve in $X$. 
Fix a positive integer $r\in\mathbb Z_{>0}$. 
If $H^1(C, \mathcal{O}_C)\neq 0$, then 
$H^0(C, \omega_X(C)^{[r]}|_C)\neq 0$ where 
$\omega_X(C)^{[r]}$ is the double dual of 
$\omega_X(C)^{\otimes r}$. 
\end{lem}

\begin{proof}
We consider the exact sequence 
$$0\to \mathcal T \to \omega_X(C)|_C \to \omega_C \to 0.$$
Since $\mathcal T$ is a skyscraper sheaf, we have 
$H^1(C, \mathcal T)=0$. 
By $H^0(C, \omega_C)\neq 0$, 
we obtain $H^0(C, \omega_X(C)|_C/\mathcal T)\neq 0$. 
Thus there exists a map
$$\mathcal O_{C} \to \omega_X(C)|_C$$
such that this is injective on some non-empty open set. 
Therefore we obtain a map
$$\mathcal O_{C} \to \omega_X(C)^{\otimes r}|_C,$$
which is injective on some non-empty open set. 
On the other hand, 
there is a natural map 
$$\omega_X(C)^{\otimes r}|_C \to \omega_X(C)^{[r]}|_C,$$
which is bijective on some non-empty open set. 
Combining these maps, we have the map 
$$\mathcal O_{C} \to \omega_X(C)^{[r]}|_C,$$
which is injective on some non-empty open set. 
Thus, the kernel $K$ of this map is a torsion subsheaf of $\mathcal O_C$. 
Then, we have $K=0$. 
Therefore, we obtain an injection $\mathcal O_{C} \hookrightarrow \omega_X(C)^{[r]}|_C.$ 
This means $H^0(C, \omega_X(C)^{[r]}|_C) \neq 0.$
\end{proof}

Using this lemma, 
we obtain the following theorem, 
which plays a crucial role in this paper.

\begin{thm}\label{adjunction}
Let $X$ be a projective normal surface and 
let $C$ be a curve in $X$ 
such that $r(K_X+C)$ is Cartier for some positive integer $r$. 
\begin{enumerate}
\item{If $C\cdot (K_X+C)<0$, then $C\simeq \mathbb{P}^1$.}
\item{If $C\cdot (K_X+C)=0$, then $C\simeq \mathbb{P}^1$ or 
$\mathcal{O}_C((K_X+C)^{[r]}) \simeq \mathcal O_C$.}
\end{enumerate}
\end{thm}

\begin{proof}
(1)Since $C\cdot (K_X+C)<0$ means $H^0(C, \omega_X(C)^{[r]}|_C)=0$, 
the curve $C$ must be $\mathbb{P}^1$ by Lemma~\ref{adjunctionlemma}. \\
(2)Assume $C\not\simeq \mathbb{P}^1$. 
Then we can apply Lemma~\ref{adjunctionlemma} and 
we obtain $H^0(C, \omega_X(C)^{[r]}|_C) \neq 0.$ 
By $C\cdot (K_X+C)=0$, we have $\omega_X(C)^{[r]}|_C\simeq \mathcal O_C.$
\end{proof}

\subsection{Contraction theorem}
In this section, we show that 
extremal rays are contractable for \Q-factorial surfaces 
with $\mathbb{R}$-boundaries. 
First, we consider the following theorem, 
which we will use later.

\begin{thm}\label{irratqfac}
In this theorem, let $k$ be an algebraically closed field of arbitrary characteristic. 
Let $\pi:Y \to B$ be a surjective morphism over $k$ 
from a smooth projective surface $Y$ 
to a smooth projective irrational curve $B$. 
Let $f:Y \to X$ be a birational morphism 
to a projective normal \Q-factorial surface. 
If $k$ is not the algebraic closure of a finite field, then 
all $f$-exceptional curves are $\pi$-vertical. 
\end{thm}

This proof is essentially due to \cite[Lemma~5.2]{Fujino}. 

\begin{proof}
We assume that $C$ is an $f$-exceptional curve with $\pi(C)=B$ and 
want to derive a contradiction. 
We may assume that $C$ is smooth by taking 
a sequence of blow-ups of singular points of $C$. 
We have 
\begin{eqnarray*}
\pi|_C:C&\longrightarrow& B\\
\Pic^0C &\overset{(\pi|_C)^*}\longleftarrow& \Pic^0B.
\end{eqnarray*}
We prove that the image $(\pi|_C)^*(\Pic^0B)$ is an abelian group whose rank is infinite. 
By considering $(\pi|_C)^*$ as a morphism between Jacobian varieties, 
we see that $(\pi|_C)^*(\Pic^0B)$ is an abelian variety. 
Note that the dimension of $(\pi|_C)^*(\Pic^0B)$ as a scheme is not zero 
by $(\pi|_C)_*\circ(\pi|_C)^*={\rm deg}(\pi|_C)$ and by the irrationality of $B$. 
Thus, by Fact~\ref{principle}, 
the rank of $(\pi|_C)^*(\Pic^0B)$ is infinite. 
Then, we have 
$$(\pi|_C)^*(\Pic^0B)\otimes_{\mathbb{Z}}\mathbb{Q}\setminus
\sum_{i=1}^{r}\mathbb{Q}(E_i|_C)\neq\emptyset$$
where $E_1,\cdots ,E_r$ are the $f$-exceptional curves. 
Therefore we can take a \Q-divisor $D$ on $B$ such that 
$$(\pi^*D)|_C \not\in\sum_{i=1}^{r}\mathbb{Q}(E_i|_C)$$
On the other hand, since $X$ is \Q-factorial, we obtain 
$$\pi^*D-f^*f_*\pi^*D\in\sum_{i=1}^{r}\mathbb{Q}E_i. $$
Restricting this relation to $C$, we have the following contradiction
$$(\pi^*D)|_C\in\sum_{i=1}^{r}\mathbb{Q}(E_i|_C) $$
because $C$ is $f$-exceptional. 
\end{proof}

Originally, \cite{Fujino} uses this theorem to prove the non-vanishing theorem. 
We use this theorem not only for the non-vanishing theorem 
but also for the following contraction theorem. 

\begin{thm}[Contraction theorem]\label{cont}
Let $X$ be a projective normal \Q-factorial surface and 
let $\Delta$ be an $\mathbb{R}$-boundary. 
Let $R=\mathbb{R}_{\geq 0}[C]$ be a $(K_X+\Delta)$-negative extremal ray. 
Then there exists a surjective morphism $\phi_R: X \to Y$ 
to a projective variety $Y$ with the following properties{\em{:}} 
{\em{(1)-(5)}}.
\begin{enumerate}
\item{Let $C'$ be a curve on $X$. 
Then $\phi_R(C')$ is one point iff $[C'] \in R$.}
\item{$(\phi_R)_*(\mathcal{O}_X)=\mathcal{O}_Y$.}
\item{If $L$ is an invertible sheaf with $L\cdot C=0$, 
then $nL=(\phi_R)^*L_Y$ for some invertible sheaf $L_Y$ on $Y$ and 
for some positive integer $n$.}
\item{$\rho(Y)=\rho(X)-1$.}
\item{$Y$ is \Q-factorial if $\dim Y=2$.}
\end{enumerate}
\end{thm}

We divide the proof into the three cases: 
$C^2>0$, $C^2=0$ and $C^2<0$.

\begin{proof}[Proof of the case where $C^2>0$.]
$C^2>0$ shows that $C$ is a nef and big divisor. 
Therefore for an arbitrary curve $C'$, 
there exists an effective Cartier divisor $E$ and 
positive integers $n$ and $m$ such that 
$nC \sim mC'+E$ by Kodaira's lemma. 
Since $C$ generates an extremal ray, 
we have $C' \equiv qC$ for some rational number $q$. 
Recall that we choose $C'$ as an arbitrary curve. 
Thus we obtain $\rho(X)=1$ and $-K_X$ is ample. 
Then let $Y$ be one point and 
(1)(2)(4) are satisfied. 
We want to prove (3), that is, 
we must show that for a \Q-divisor $D$ if $D \equiv 0$, then $D$ is a torsion. 
It is sufficient to prove $\kappa(X, D) \geq 0$. 
Thus, we assume $\kappa(X, D) =-\infty$ and derive a contradiction. 
Let $f:X' \to X$ be the minimal resolution, 
$D'=f^*D$ and 
$K_{X'}+E'=f^*K_X$ where $E'$ is an effective $f$-exceptional \Q-divisor. 
Then we obtain
$$\kappa(X', K_{X'})\leq \kappa(X', K_{X'}+E')=\kappa(X, K_X)=-\infty.$$
First we prove that $X'$ is an irrational ruled surface. 
By Serre duality, we obtain 
$h^2(X', D')=
h^0(X', K_{X'}-D')$. 
Moreover we get 
\begin{align*}
\kappa(X', K_{X'}-D') &\leq 
\kappa(X', K_{X'}+E'-D')\\ &=
\kappa(X', f^*(K_{X}-D))\\ & =
\kappa(X, K_X-D)=-\infty
\end{align*}
by the anti-ampleness of $K_X-D$. 
Hence $h^2(X', D')=0$. 
Then, by the Riemann--Roch theorem, we obtain 
$$-h^1(X', D')=\chi(\mathcal{O}_{X'})+\frac{1}{2}D'\cdot(D'-K_{X'})=\chi(\mathcal{O}_{X'})$$ 
because $D'=f^*D\equiv 0$. 
This shows that 
$$0 \geq -h^1(X', D') = 1-h^1(X', \mathcal{O}_{X'}).$$ 
Thus we obtain  $h^1(X', \mathcal{O}_X') \geq 1$ 
and this means that $X'$ is an irrational ruled surface. 
Let $\pi:X' \to B$ be its ruling. 
Here, if $k=\fff$, then $D$ is a torsion by Corollary \ref{pic}. 
Hence we consider the case $k \neq \fff$. 
Then we can apply Theorem~\ref{irratqfac} and 
all $f$-exceptional curves are $\pi$-vertical. 
This shows that $\pi$ factors through 
$X' \to X \to B$. 
A curve in a fiber has non-positive self-intersection number. 
But this is a contradiction because 
each curve in $X$ is ample. 
This completes the proof of the case $C^2>0$. 
\end{proof}

\begin{proof}[Proof of the case where $C^2=0$.]
First let us prove $\rho(X)=2$. 
It is sufficient to show that for an arbitrary divisor $F$, 
if $F\cdot C=F\cdot (K_X+\Delta)=0$, then  $F \equiv 0$. 
We need the following lemma. 

\begin{lem}\label{quad}
If $D_1, D_2 \in C^{\perp}=\{D \, |\, {\text{$D$ is a divisor and }}\, D\cdot C=0\}$, 
then $D_1\cdot D_2=0$
\end{lem}

This proof is essentially due to \cite[Lemma~3.29]{Mori2}. 

\begin{proof}[Proof of {\em{Lemma \ref{quad}}}]
We consider the quadratic form 
$Q:C^{\perp}_{\mathbb{R}} \to \mathbb{R}$. 
Here we consider $C^{\perp}_{\mathbb{R}}$ as a subvector-space of 
the numerical equivalence classes of $\mathbb{R}$-divisors and 
$Q$ is defined by the self-intersection. 
We want to prove that $Q$ is identically zero. 
Take a nef divisor $G$ such that 
$\overline{NE}(X) \cap G^{\perp}=\mathbb{R}_{\geq 0}[C]$. 
By the nefness of $G$, we obtain $G^2 \geq 0$. 
But $G^2$ must be $0$ because $G^2>0$ shows that $G$ is nef and big. 
Then, by $G\cdot C=0$, we obtain $C^2<0$ and this is a contradiction. 
This shows that $Q$ is zero in a non-empty dense subset 
of an open subset in $C^{\perp}_{\mathbb{R}}$ by the cone theorem. 
Therefore $Q$ must be identically zero. 
\end{proof}

Since $F \in C^{\perp}$, we obtain $D\cdot F=0$ 
for any divisor $D \in C^{\perp}$. 
$\mathbb{R}$-subvector-space $C^{\perp}_{\mathbb{R}}$ 
in numerical classes of divisors has codimension one. 
Take its basis $D_1, \cdots , D_{\rho-1}$. 
Then we get the basis $D_1, \cdots , D_{\rho-1},$ $ (K_X+\Delta)$ of the whole space. 
Indeed, by $C\cdot (K_X+\Delta) \neq 0$, these vectors are linearly independent. 
Since $F\cdot D_1= \cdots =F\cdot D_{\rho-1}=F\cdot (K_X+\Delta)=0$, 
we get $F \equiv 0$. 
Thus, we obtain $\rho(X)=2$. 

Next, let us prove that the divisor $C$ is semi-ample. 
By $C^2=0$ and $(K_X+\Delta)\cdot C<0$, 
we obtain $K_X\cdot C<0$. 
Let $f: X' \to X$ be a resolution. 
By 
\begin{eqnarray*}
(f^*C)^2&=&f_*(f^*C)\cdot C=C\cdot C=0\,\,\,\,\,{\rm and}\\
K_{X'}\cdot f^*C&=&f_*(K_{X'})\cdot C=K_X\cdot C<0,
\end{eqnarray*}
The Riemann--Roch theorem shows that $\kappa(X', f^*C) \geq 1$. 
Note that $h^2(X', nf^*C)=h^0(X', K_{X'}-nf^*C)=0$ 
for all $n \gg 0$.  
Therefore we get $\kappa(X, C)=\kappa(X', f^*C) \geq 1$. 
$C^2=0$ implies $\kappa(X, C)=1$. 
Then, by the following proposition, $C$ is a semi-ample divisor. 

\begin{prop}\label{kappa1}
Let $X$ be a projective normal surface and $L$ be a nef line bundle. 
If $\kappa(X, L)=1$, then $L$ is semi-ample.
\end{prop}

\begin{proof}
See \cite[Theorem~4.1]{Fujita}.
\end{proof}

Hence the complete linear system $|mC|$ induces a morphism 
$\phi_R:X \to Y$ to a smooth projective curve $Y$. 
This morphism satisfies (1), (2) and (4). 
We would like to show (3). 
Take a line bundle $L$ with $L\cdot C=0$. 
Since $\rho(X)=2$, 
we have $L\equiv qC$ for some rational number $q$. 
We take a large positive integer $s$ such that $q+s$ is positive. 
Then we have 
$$L+sC\equiv (q+s)C.$$
By the same argument as above, we see that $L+sC$ is semi-ample. 
Then sufficiently large multiple of $L+sC$ induces a morphism 
$\psi:X \to Z$ to a smooth projective curve $Z$. 
Moreover since this morphism satisfies the condition (1), 
we obtain the factorization 
$$\psi:X \overset{\phi_R}\to Y \overset{\sigma}\to Z$$
with $\sigma_* \mathcal O_Y=\mathcal O_Z$. 
Since $Y$ and $Z$ are smooth projective curves, 
$\sigma$ must be isomorphism. 
Then $n(L+sC)$ is a pull-back of line bundle on $Y$ 
for some positive integer $n$. 
This means (3). 
\end{proof}

Before the proof of the case where $C^2<0$, 
we state a proposition on the contraction of \pr1.

\begin{prop}\label{p1cont}
Let $X$ be a projective normal surface and 
let $C$ be a curve in $X$ isomorphic to \pr1. 
Assume $G$ is a nef and big line bundle on $X$ 
such that, for every curve $C'$ in $X$, 
$G\cdot C'=0$ iff $C'=C$. 
Then, $G$ is semi-ample. 
\end{prop}

\begin{proof}
By Keel's result (Theorem~\ref{keel}), 
if $G|_C$ is semi-ample, then $G$ is semi-ample.
But, by $C=\mathbb P^1$, this is obvious. 
\end{proof}

\begin{proof}[Proof of the case where $C^2<0$]

By $C\cdot (K_X+\Delta)<0$, 
we have $C\cdot (K_X+C)<0$. 
Therefore, by Theorem~\ref{adjunction}, we see $C \simeq \mathbb{P}^1$. 
Let $G$ be a nef and big divisor such that 
for any curve $C'$, $G\cdot C'=0$ iff $C'=C$. 
(The way to construct such a divisor $G$ is the following: 
let $H$ be an ample divisor and $G$ be the divisor such that 
$G=H+qC$ and $G\cdot C=0$ for rational number $q$.) 
Then, by Proposition~\ref{p1cont}, there exists $\phi_R$ 
satisfying (1) and (2). 
The remaining assertions (3), (4) and (5) hold 
from the following propositions. 
\end{proof}

First we prove (3). 
We generalize the setting a little for a later use. 

\begin{prop}[Proof of (3)]\label{(3)}
Let $f:X \to Y$ be a proper birational morphism 
from a normal \Q-factorial surface $X$ 
to a normal surface $Y$. 
Assume $C:={\rm Ex}(f)$ is a proper irreducible curve and $f(C)$ is one point. 
Let $L$ be a Cartier divisor on $X$ with $L\cdot C=0$. 
If $L|_C$ is a torsion, then $nL=f^*(L_Y)$ 
for some Cartier divisor $L_Y$ on $Y$ and for some positive integer $n$.
\end{prop}

\begin{proof}
\setcounter{step}{0}
\begin{step}
In this step, we assume $X$ and $Y$ are projective and 
we prove the assertion. 

Let $G$ be the pull-back of an ample divisor. 
By Kodaira's lemma, $G=A+E$ where $A$ is an ample \Q-divisor and 
$E$ is an effective \Q-divisor. 
By replacing $G$ by its suitable multiple, 
it is easy to see that we may assume $E=qC$ for some $q\in \mathbb Q_{>0}.$ 
Consider the divisor 
$$G'=mG+L=(mA+L)+mqC$$ 
for $m \gg 0.$ 
Since $mA+L$ is ample for $m\gg 0$, 
we see $G'\cdot C'>0$ for every curve $C'\neq C.$ 
On the other hand, we have 
$$G'\cdot C=(mG+L)\cdot C=0.$$
Thus, for sufficiently large integer $m\gg 0$, 
the Cartier divisor $G'=mG+L$ is nef and big such that 
$G'\cdot C'=0$ iff $C'=C$ for every curve $C'$. 
Since $L|_C$ is a torsion, 
$G'=mG+L$ is semi-ample by Keel's result (Theorem~\ref{keel}). 
By Zariski's main theorem, $|nG'|$ induces the 
same morphism as $f$ for some $n\in\mathbb Z_{>0}$. 
Thus, $nG'=nmG+nL$ is a pull-back of some line bundle on $Y$. 
So is the difference $nL=nG'-nmG$. 
\end{step}

\begin{step}
In this step, we assume $Y$ is quasi-projective and 
we prove the assertion. 

Take a compactification $Y\subset \overline Y$ such that 
$\overline Y$ is projective and $\overline Y$ is smooth on $\overline Y\setminus Y$. 
We define $\overline X$ by patching $X$ and $\overline Y$ along 
$X\setminus C\simeq Y\setminus \{f(C)\}.$ 
Then, $\overline X$ is projective because $\overline X$ is proper and \Q-factorial 
(cf. \cite[Lemma~2.2]{Fujino}). 
Thus, by Step~1, we obtain the required assertion. 
\end{step}

\begin{step}
In this step, we prove the assertion. 

Let $f(C)\in Y_0\subset Y$ be an affine open subset and let 
$X_0:=f^{-1}(Y_0).$ 
Let $f|_{X_0}=:f_0.$ 
Then, by Step~2, we obtain $nL|_{X_0}=(f_0)^*L_{Y_0}$. 
Let $L_Y$ be the $\mathbb Z$-divisor on $Y$ 
such that $L_Y|_{Y_0}=L_{Y_0}$ and that $L_Y$ has no prime component contained in $Y\setminus Y_0$. 
Then, $L_Y$ is \Q-Cartier. 
Consider the following prime decomposition 
$$L=\sum l_iC_i=\sum_{C_i\subset X\setminus X_0}l_iC_i+\sum_{C_j\not\subset X\setminus X_0}l_jC_j.$$
We see $nf^*L_Y=\sum_{C_j\not\subset X\setminus X_0}l_jC_j.$ 
Since $\sum_{C_j\subset X\setminus X_0}l_jC_j$ is the pull-back of some \Q-Cartier divisor, 
we obtain the assertion. 
\end{step}
\end{proof}

The condition (4) is an immediate corollary from (3). 
Thus we would like to prove (5). 

\begin{prop}[Proof of (5)]\label{(5)}
Let $f:X \to Y$ be a proper birational morphism 
from a normal \Q-factorial surface $X$ 
to a normal surface $Y$. 
Assume $C:={\rm Ex}(f)$ is a proper irreducible curve and $f(C)$ is one point. 
Suppose the following condition. 
\begin{enumerate}
\item[(3)]{If $L$ is a Cartier divisor with $L\cdot C=0$, 
then $nL=(\phi_R)^*L_Y$ for some Cartier divisor $L_Y$ on $Y$ and 
for some positive integer $n$.}
\end{enumerate}
Then, $Y$ is \Q-factorial.
\end{prop}

\begin{proof}
Let $E$ be a prime divisor on $Y$ and let $D$ be its proper transform. 
Since $C^2<0$, there exists a rational number $q$ such that $(D+qC)\cdot C=0$. 
By (3), we have $n(D+qC)=f^*(L_Y)$ 
for some Cartier divisor $L_Y$. 
By operating $f_*$, we obtain the equality $nE=L_Y$ as Weil divisors. 
Therefore $E$ is \Q-Cartier.
\end{proof}

Since we have the cone theorem and the contraction theorem, 
we obtain the minimal model program for \Q-factorial surfaces with boundaries.

\begin{thm}[Minimal model program]\label{qfacmmp}
Let $X$ be a projective normal \Q-factorial surface and 
let $\Delta$ be an $\mathbb{R}$-boundary. 
Then, there exists a sequence of projective birational morphisms 
\begin{eqnarray*}
&(X, \Delta)=:(X_0, \Delta_0) \overset{\phi_0}\to 
(X_1, \Delta_1) \overset{\phi_1}\to \cdots 
\overset{\phi_{s-1}}\to (X_s, \Delta_s)=:(X^{\dagger}, \Delta^{\dagger})\\
&\,\,\,where\,\,\,(\phi_{i-1})_*(\Delta_{i-1})=:\Delta_i
\end{eqnarray*}
with the following properties. 
\begin{enumerate}
\item{Each $X_i$ is a projective normal \Q-factorial surface.}
\item{Each $\Delta_i$ is an \R-boundary.} 
\item{For each $i$, ${\rm Ex}(\phi_i)=:C_i$ is an irreducible curve such that 
$$(K_{X_i}+\Delta_i)\cdot C_i<0$$
and that $C_i$ generates an extremal ray.}
\item{$(X^{\dagger}, \Delta^{\dagger})$ satisfies one of the following conditions.
\begin{enumerate}
\item{$K_{X^{\dagger}}+\Delta^{\dagger}$ is nef.}
\item{There is a projective surjective morphism $\mu:X^{\dagger} \to Z$ to 
a smooth projective curve $Z$ such that $\mu_*\mathcal O_{X^{\dagger}}=\mathcal O_Z$, 
$-(K_X^{\dagger}+\Delta^{\dagger})$ is $\mu$-ample and $\rho(X^{\dagger})=2$.}
\item{$-(K_{X^{\dagger}}+\Delta^{\dagger})$ is ample and 
$\rho(X^{\dagger})=1$.}
\end{enumerate}}
\end{enumerate}
\end{thm}

\subsection{Finite generation of canonical rings}

It is important to consider the finite generation of canonical rings, 
which is closely related to the minimal model program. 
In this section, we prove the following theorem. 

\begin{thm}[Finite generation theorem]
Let $X$ be a projective normal \Q-factorial surface over $k$ and 
let $\Delta$ be a \Q-boundary. 
Then $R(X, K_X+\Delta):=\bigoplus\limits_{m \geq 0} H^0(X, \llcorner m(K_X+\Delta)\lrcorner)$
is a finitely generated $k$-algebra. 
\end{thm}

\begin{proof}
Let us consider the Kodaira dimension $\kappa:=\kappa(X, K_X+\Delta)$. 
It is obvious for the case $\kappa=-\infty$ and the case $\kappa=0$ . 
In particular we may assume that $K_X+\Delta$ is effective. 
Then, by Theorem~\ref{qfacmmp}, 
we may assume that $K_X+\Delta$ is nef. 
The case $\kappa=1$ follows from Proposition~\ref{kappa1}. 
Therefore we may assume $\kappa=2$, that is, 
$K_X+\Delta$ is nef and big. 
This case follows from Proposition~\ref{nefbig}.
\end{proof}

\begin{prop}\label{nefbig}
Let $X$ be a projective normal \Q-factorial surface and 
let $\Delta$ be a \Q-boundary. 
If $K_X+\Delta$ is nef and big, then $K_X+\Delta$ is semi-ample.
\end{prop}

\begin{proof}
By Keel's result (Theorem~\ref{keel}), 
it is sufficient to prove that 
if 
$$E:=\bigcup\limits_{C\cdot (K_X+\Delta)=0}C=C_1 \cup \dots \cup C_r,$$ 
then $(K_X+\Delta)|_{E}$ is semi-ample. 
Let $C \subset E$. 
Then we have 
\begin{eqnarray*}
(K_X+C)\cdot C \leq (K_X+\Delta)\cdot C=0.
\end{eqnarray*}

\setcounter{step}{0}
\begin{step}
In this step, we reduce the proof to the case 
where if $C \subset E$, then $(K_X+C)\cdot C=0$. 

Assume $C \subset E$ and $(K_X+C)\cdot C<0$. 
Then $C$ is a $(K_X+C)$-negative extremal curve. 
Thus, by Theorem~\ref{cont}, we can contract $C$. 
Let $f:X \to Y$ be its contraction and $\Delta_Y:=f_*(\Delta)$. 
Then since $K_X+\Delta=f^*(K_Y+\Delta_Y)$ and $Y$ is \Q-factorial, 
if we can prove that $K_Y+\Delta_Y$ is semi-ample, 
then $K_X+\Delta$ is semi-ample. 
We can repeat this procedure and we obtain the desired reduction. 
\end{step}

\begin{step}
In this step, 
we prove that $E$ is a disjoint union of irreducible curves and 
if $C \subset E$, 
then $(K_X+\Delta)|_C=(K_X+C)|_C$. 

Let $C \subset E$. 
By Step 1, we have $(K_X+C)\cdot C=0$. 
Then, the inequality over Step 1 is an equality. 
Thus $C \subset \Supp \Delta$ and 
$C$ is disjoint from any other component of $\Delta$. 
\end{step}

By Step 2, it is sufficient to prove that, 
if $(K_X+C)\cdot C=0$, then $(K_X+C)|_C$ is semi-ample. 
This is satisfied by Theorem~\ref{adjunction}. 
\end{proof}

\subsection{Abundance theorem ($k\neq \overline{\mathbb F}_p$)}

In this section, we prove the abundance theorem for 
\Q-factorial surfaces with \Q-boundary over $k\neq\overline{\mathbb{F}}_p$. 
The case where $k=\overline{\mathbb{F}}_p$ will be treated in Section~4. 
In the case where $\kappa(X, K_X+\Delta)=0$, 
we give the proof of abundance theorem which does not depend 
on the characteristic of the base field $k$ (Theorem~\ref{qabundancekappazero}). 

First we see the following non-vanishing theorem. 

\begin{thm}[Non-vanishing theorem]\label{nonv}
In this theorem, let $k$ be an algebraically closed field of arbitrary characteristic. 
Let $X$ be a projective normal \Q-factorial surface over $k$ and 
let $\Delta$ be a \Q-boundary. 
If $k$ is not the algebraic closure of a finite field and 
$K_X+\Delta$ is pseudo-effective, then $\kappa(X, K_X+\Delta) \geq 0$.
\end{thm}

\begin{proof}
See \cite[Theorem~5.1, Lemma~5.2]{Fujino}. 
Note that, instead of Lemma~5.2 of \cite{Fujino}, 
we use Theorem~\ref{irratqfac} of this paper.
\end{proof}

Before the proof of the abundance theorem, 
we see the definition of indecomposable curves of canonical type 
in the sense of \cite[P.330]{Mum}.

\begin{dfn}
In this definition, let $k$ be an algebraically closed field of arbitrary characteristic. 
Let $X$ be a smooth projective surface over $k$ and 
let $Y=\sum n_iE_i$ be an effective divisor with $n_i\in \mathbb Z_{>0}$. 
We say $Y$ is an {\em indecomposable curve of canonical type} if 
$Y\neq 0$, $K_X\cdot E_i=Y\cdot E_i=0$ for all $i$, $\Supp Y$ is connected and $\gcd(n_i)=1$.
\end{dfn}

We see criteria for the movability of 
indecomposable curves of canonical type.

\begin{prop}\label{icct1}
In this proposition, 
let $k$ be an algebraically closed field 
of arbitrary characteristic. 
Let $X$ be a smooth projective surface over $k$ and 
let $Y$ be an indecomposable curve of canonical type in $X$. 
Assume that one of the following assertions holds: 
\begin{enumerate}
\item{${\rm char}\,k=p>0$.}
\item{$H^1(X, \mathcal O_X)=0$.}
\end{enumerate}
If $\mathcal{O}_Y(Y)$ is a torsion, then $\kappa(X, Y)=1$. 
\end{prop}

The proof of (2) is very similar to \cite[Theorem~2.1]{Totaro}.

\begin{proof}
(1)See \cite[Lemma in P.682]{Masek}. \\
(2)Assume $H^1(X, \mathcal O_X)=0$. 
Let $m$ be the order of $\mathcal{O}_Y(Y)$ and 
$a$ be an integer with $1 \leq a \leq m-1$. 
The case where $\mathcal{O}_Y(Y)=\mathcal{O}_Y$ is easy. 
So we exclude this case and we can assume $m \geq 2$. 
Let us consider the following exact sequence: 
\[H^0(Y, \mathcal{O}_Y(aY)) \to H^1(X, (a-1)Y) \to H^1(X, aY) 
\to H^1(Y, \mathcal{O}_Y(aY)).\] 
By Serre duality and \cite[Corollary 1 in P.333]{Mum}, we obtain 
$$h^1(Y, \mathcal{O}_Y(aY))=h^0(Y, \mathcal{O}_Y(K_Y-aY))=
h^0(Y, \mathcal{O}_Y(-aY)). $$ 
The choice of $a$ shows that 
$$h^0(Y, \mathcal{O}_Y(aY))=h^0(Y, \mathcal{O}_Y(-aY))=0. $$
Indeed, suppose the contrary, that is, 
for example suppose $h^0(Y, aY|_Y) \neq 0$. 
Then we have $\mathcal{O}_Y(aY)=\mathcal{O}_Y$ 
by \cite[Lemma on P.332]{Mum}. This is a contradiction. 
Therefore we get 
\[0=h^1(X, \mathcal{O}_X)=h^1(X, Y)= \dots =h^1(X, (m-1)Y).\] 
This leads the following exact sequence: 
\[0 \to H^0(X, (m-1)Y) \to H^0(X, mY) \to H^0(Y, \mathcal{O}_Y) \to 0.\]
Thus $Y$ is an effective semi-ample divisor on $X$ and $Y^2=0$. 
This shows that $\kappa(X, Y)=1$.
\end{proof}

Now, we prove the abundance theorem.

\begin{thm}[Abundance theorem]\label{qabundance}
In this theorem, let $k$ be an algebraically closed field of positive characteristic. 
Let $X$ be a projective normal \Q-factorial surface over $k$ and 
let $\Delta$ be a \Q-boundary. 
If $k$ is not the algebraic closure of a finite field and 
$K_X+\Delta$ is nef, then $K_X+\Delta$ is semi-ample. 
\end{thm}

\begin{proof}
By Theorem~\ref{nonv}, we may assume $\kappa(K_X+\Delta)\geq 0$. 
Moreover, we may assume $\kappa(K_X+\Delta)=0$ 
by Proposition~\ref{kappa1} and Proposition~\ref{nefbig}. 
Thus it is sufficient to prove the following theorem. 
\end{proof}

\begin{thm}\label{qabundancekappazero}
In this theorem, 
let $k$ be an algebraically closed field of arbitrary characteristic. 
Let $X$ be a projective normal \Q-factorial surface over $k$ and 
let $\Delta$ be a \Q-boundary. 
If $k$ is not the algebraic closure of a finite field, 
$K_X+\Delta$ is nef and 
$\kappa(X, K_X+\Delta)=0$, 
then $K_X+\Delta\sim_{\mathbb{Q}}0$. 
\end{thm}

This proof is very similar to \cite[Theorem~6.1]{Fujino} and 
uses many techniques in \cite[section~5]{Fujita}.

\begin{proof}
\setcounter{step}{0}
Let $f:V\to X$ be the minimal resolution. 
We set $K_V+\Delta_V=f^*(K_X+\Delta)$. 
We note that $\Delta_V$ is effective. 
It is sufficient to see that $K_V+\Delta_V\sim _{\mathbb Q}0$. 
Let 
$$
\varphi:V=:V_0\overset{\varphi_0}\to V_1\overset{\varphi_1}\to \cdots 
\overset{\varphi_{k-1}}\to V_k=:S 
$$ 
be a sequence of blow-downs such that 
\begin{itemize}
\item[(1)] $\varphi_i$ is a blow-down of a $(-1)$-curve $C_i$ on $V_i$, 
\item[(2)] $\Delta_{V_{i+1}}=\varphi_{i*}\Delta_{V_i}$, and 
\item[(3)] $(K_{V_i}+\Delta_{V_i})\cdot C_i=0$, 
\end{itemize}
for every $i$. 
We can assume that there are no $(-1)$-curves $C$ on $S$ with 
$(K_S+\Delta_S)\cdot C=0$. We note that 
$K_V+\Delta_V=\varphi^*(K_S+\Delta_S)$. 
It is sufficient to 
see that $K_S+\Delta_S\sim _{\mathbb Q}0$. 
Since $\kappa(S, K_S+\Delta_S)=0$, there is a member $Z$ of 
$|m(K_S+\Delta_S)|$ for some positive integer $m$. 
Then, for every positive integer $t$, 
$tZ$ is the unique member of $|tm(K_S+\Delta_S)|$. 
We will derive a contradiction assuming $Z\ne 0$. 

\begin{step}
In this step, we prove that for each prime component $Z_i$ of $Z$, 
we have 
$$K_S\cdot Z_i=\Delta_S \cdot Z_i=Z\cdot Z_i=0. $$
Since $(K_S+\Delta_S)\cdot Z=m(K_S+\Delta_S)^2=0$ 
and $(K_S+\Delta_S)$ is nef, 
$(K_S+\Delta_S)\cdot Z_i=0$ for all $i$. 
This means $$Z\cdot Z_i=0$$ and $Z_i^2 \leq 0$. 
Now, we prove $K_S\cdot Z_i \geq 0$ for every $i$. 
If $K_S\cdot Z_i < 0$, then we obtain 
$(K_S+Z_i)\cdot Z_i<0$ and $Z_i \cong \mathbb{P}^1$. 
If $Z_i^2 \geq 0$, then we obtain
$\kappa(S, Z) \geq \kappa (S, Z_i)>0.$ 
This contradicts 
$\kappa(S, K_S+\Delta_S)=\kappa(S, Z)=0.$ 
If $Z_i^2 < 0$, then $Z_i$ is a $(-1)$-curve 
with $(K_S+\Delta_S)\cdot Z_i=0.$ 
This contradicts the definition of $S$. 
Anyway, we have $K_S\cdot Z_i \geq 0$ for every $i$.
This implies $K_S\cdot Z=K_S\cdot m(K_S+\Delta_S) \geq 0.$ 
The nefness of $K_S+\Delta_S$ shows $(K_S+\Delta_S)\cdot \Delta_S \geq 0.$ 
By $(K_S+\Delta_S)^2=0$, we see 
$(K_S+\Delta_S)\cdot K_S=(K_S+\Delta_S)\cdot \Delta_S=0.$ 
This is equivalent to $Z\cdot K_S=Z\cdot \Delta_S=0.$ 
Since $K_S\cdot Z_i \geq 0$, we see 
$$K_S\cdot Z_i=\Delta_S \cdot Z_i=0. $$
\end{step}

\begin{step}
We can decompose $Z$ into the connected components as follows: 
$$
Z=\sum _{i=1}^{r} \mu_i Y_i, 
$$ 
where $\mu_i Y_i$ is a connected component of 
$Z$ such that $\mu_i$ is the greatest common divisor of 
the coefficients of prime components of $Y_i$ in $Z$ for every $i$. 
Then we see that, for every $i$, 
each $Y_i$ is an indecomposable curve of 
canonical type by Step~1. 
We obtain $\omega _{Y_i}\simeq \mathcal O_{Y_i}$ 
by \cite[Corollary~1 in P.333]{Mum}. 
\end{step}

\begin{step}
In this step, 
we assume $\kappa (S, K_S)\geq 0$ and prove the assertion. 
Since 
$$0\leq \kappa (S, K_S)\leq \kappa (S, K_S+\Delta_S)=0, $$
we obtain $\kappa (S, K_S)=0$. 
Let us prove $S$ is a minimal surface. 

Suppose the contrary, that is, 
suppose that there exists a $(-1)$-curve $E$. 
Then we have the contraction $g:S\to S'$ of $E$ and 
we obtain a morphism 
$$S\overset{g}\to S'\overset{h}\to S_{{\rm min}}$$
to a minimal surface $S_{{\rm min}}$. 
Since $0=\kappa(S, K_S)=\kappa(S_{{\rm min}}, K_{S_{{\rm min}}})$, 
we see $K_{S_{{\rm min}}}\sim_{\mathbb{Q}}0$. 
Because 
$$K_S=g^*K_{S'}+E=g^*(h^*(K_{S_{{\rm min}}})+({\rm effective\,\,divisor}))+E,$$ 
we see that 
$K_S\sim_{\mathbb{Q}} ({\rm effective\,\,divisor})+E.$
This means that 
$$nZ\sim nm(K_S+\Delta_S)\sim ({\rm effective\,\,divisor})+nmE+nm\Delta_S$$
for some $n\in\mathbb Z_{>0}.$ 
Since $\kappa(S, Z)=0$, 
$$nZ=({\rm effective\,\,divisor})+nmE+nm\Delta_S$$
as Weil divisors. 
In particular, we have $E\subset \Supp~Z$ and 
$E=Z_i$ for some $i$. 
This implies that 
$Z_i\cdot (K_S+\Delta_S)=0$ and $Z_i$ is a $(-1)$-curve and 
this is a contradiction to the construction of $S$. 
Therefore, $S$ is minimal. 

Then we obtain the following contradiction 
$$\kappa (S, K_S+\Delta_S)=\kappa (S, Z)\geq \kappa (S, Y_i) \geq 1$$ 
from the known result $\kappa (S, Y_i) \geq 1$. 
(See, for example, \cite[Theorem 7.11]{Badescu}. ) 
\end{step}

\begin{step}
By Step~3, we may assume that $\kappa (S, K_S)=-\infty$. 
In Step~5 and Step~6, 
we assume that $S$ is rational and prove the assertion. 
In Step~7--Step~12 we assume that $S$ is irrational and prove the assertion. 
Note that 
since $\kappa(X, Z)=0$, 
in order to derive a contradiction, 
we want to prove that $\kappa(X, Y_i)\geq 1$ for some $i$. 
\end{step}

We assume that $S$ is rational.

\begin{step}
In this step, we prove 
$$\Delta_S=\sum y_iY_i \,\,{\rm and}\,\, y_i>1. $$
We fix $i$. 
By $H^1(S, \mathcal O_S(K_S))=0$ and the following exact sequence 
$$
0 \to \mathcal O_S(K_S) \to \mathcal O_S(K_S+Y_i) 
\to \omega_{Y_i} \to 0,$$
we obtain the surjection 
$$H^0(S, \mathcal O_S(K_S+Y_i)) \to H^0(Y_i, \omega_{Y_i})\simeq H^0(Y_i, \mathcal O_{Y_i}).$$
Thus there exists $W_i \in |K_S+Y_i|$ such that 
$W_i$ has no components of $Y_i$. 
For $\tilde{Z}_i=Z-\mu_iY_i$, we obtain the equation 
$$m\mu_iW_i+m\mu_i\Delta_S+m\tilde{Z}_i=(\mu_i+m)Z.$$ 
Note that this equality holds as Weil divisors because 
$\kappa(S, Z)=0$. 
From this equation, $\Supp\Delta_S \subset \Supp Z$. 
Since $W_i$ and $\tilde{Z}_i$ are free from the components of $Y_i$, 
we have $\Delta_S=\sum\frac{\mu_i+m}{m}Y_i$. 
We set $y_i:=\frac{\mu_i+m}{m}>1$. 
\end{step}

We fix $i$, and we denote $Y$ instead of $Y_i$. 

\begin{step}
In this step, we prove the desired assertion. 
By (2) in Proposition~\ref{icct1}, it is sufficient to prove that 
$\mathcal O_{Y}(a{Y})\simeq \mathcal O_{Y}$ for 
some positive integer $a$. 
We set $Y_{(k)}:=Y$ and construct $Y_{(j)}$ inductively. 
It is easy to see that 
$\varphi_j:V_j\to V_{j+1}$ is the blow-up at $P_{j+1}$ with 
$\mult _{P_{j+1}}\Delta_{V_{j+1}}\geq 1$ 
for every $j$ since $\Delta_{V_j}$ is effective. 
If $\mult _{P_{j+1}}Y_{(j+1)}=0$, 
then we set $Y_{(j)}=\varphi^*_{j}Y_{(j+1)}$. 
If $\mult _{P_{j+1}}Y_{(j+1)}>0$, 
then we set $Y_{(j)}=\varphi^*_{j}Y_{(j+1)}-C_j$, where 
$C_j$ is the exceptional curve of $\varphi_j$. 
Thus, we obtain $Y_{(0)}$ on $V_0=V$. 
Note that $\mult _P\Delta_{V_{j+1}}>\mult _PY_{(j+1)}$ for 
every $P\in \Supp~Y_{(j+1)}$ by Step~5 and 
the above inductive construction. 
Moreover, since $\mult _PY_{(j+1)}\in \mathbb Z$, 
we see that $Y_{(0)}$ is effective and that 
$\Supp Y_{(0)}\subset \Supp \Delta^{>1}_V$ 
where, for the prime decomposition 
$\Delta_V=\sum \delta_l\Delta_{V,l}$, 
we define $\Delta^{>1}_V:=\sum_{\delta_l>1} \delta_l\Delta_{V,l}.$ 
Then, we have 
$\varphi_{j*}\mathcal O_{Y_{(j)}}\simeq \mathcal O_{Y_{(j+1)}}$ 
for every $j$. 
Indeed, $\varphi_{j*}\mathcal O_{V_j}(-Y_{(j)})\simeq 
\mathcal O_{V_{j+1}}(-Y_{(j+1)})$ and 
$R^1\varphi_{j*}\mathcal O_{V_{j}}(-Y_{(j)})=0$ for every $j$. 
See the following commutative diagram. 
$$\begin{CD}
0@>>>
\mathcal O_{V_{j+1}}(-Y_{(j+1)})@>>>
\mathcal O_{V_{j+1}}@>>>
\mathcal O_{Y_{(j+1)}}@>>>0\\
@. @VV{\simeq }V @VV{\simeq }V @VVV \\
0@>>>
\varphi_{j*}\mathcal O_{V_j}(-Y_{(j)})@>>>
\varphi_{j*}\mathcal O_{V_{j}}@>>>
\varphi_{j*}\mathcal O_{Y_{(j)}}@>>>0
\end{CD}$$
Therefore, we obtain $\varphi_*\mathcal O_{Y_{(0)}}\simeq \mathcal O_Y$. 
Since $\Supp Y_{(0)}\subset \Supp \Delta^{>1}_V$, we see that 
$Y_{(0)}$ is $f$-exceptional. 
Since $K_V+\Delta_V=f^*(K_X+\Delta)$, 
we obtain 
$\mathcal O_{Y_{(0)}}(b(K_V+\Delta_V))\simeq 
\mathcal O_{Y_{(0)}}$ for 
some positive divisible integer $b$. 
Thus, 
$$\mathcal O_Y(b(K_S+\Delta_S))\simeq
\varphi_*\mathcal O_{Y_{(0)}}(b(K_V+\Delta_V)) \simeq 
\varphi_*\mathcal O_{Y_{(0)}} \simeq
\mathcal O_Y. $$ 
This means 
$$\mathcal O_Y(\mu bY)\simeq
\mathcal O_Y(bZ)\simeq
\mathcal O_Y(bm(K_S+\Delta_S))\simeq
\mathcal O_Y.$$
This completes the proof of the rational case. 
\end{step}

We assume that $S$ is an irrational ruled surface. 
Let $\pi :S\to B$ be its ruling and 
let $F$ be one of its smooth fibers.

\begin{step}
In this step, we prove that 
each connected component $\mu Y$ of $Z$ satisfies $F\cdot Y>0$. 

We assume that $F\cdot Y=0$ and derive a contradiction. 
Since $Y$ is connected, $Y$ is contained in some fiber $F_0$. 
Then we have the equality 
$$F_0=yY+Y'$$
for some effective \Q-divisor $Y'$ with $\Supp Y \not\subset \Supp Y'$ and 
for some positive rational number $y$. 
By $K_S\cdot F_0=-2$ and $K_S\cdot Y=0$, 
it is sufficient to prove $Y'=0$. 
Thus assume $Y'\neq 0.$ 
Take a prime component $Y_{(1)}$ of $Y$ which is not a component of $Y'$. 
The equalities $F_0\cdot Y_{(1)}=Y\cdot Y_{(1)}=0$ show $Y'\cdot Y_{(1)}=0.$ 
Thus if $Y_{(2)}$ is a prime component of $Y$ such that 
$Y_{(1)} \cap Y_{(2)} \neq \emptyset$, then 
$Y_{(2)}$ is not a component of $Y'$. 
By repeating this procedure, we see that 
$Y_{(i)}$ is not a prime component of $Y'$ 
for each prime component $Y_{(i)}$ of $Y$. 
Since $Y'\neq 0$, there exists a prime component $Y_{(j)}$ of $Y$ 
with $Y_{(j)} \cap Y' \neq \emptyset$. 
This leads to the following contradiction
$$F_0\cdot Y_{(j)}=Y\cdot Y_{(j)}=0\,\,{\rm and}\,\,Y'\cdot Y_{(j)}\neq 0.$$
\end{step}

\begin{step}
In this step, we prove that both $B$ and $Y$ are elliptic curves. 
In particular, $Y^2=0$. 

By Step~7, $Y$ has a prime component $Y_{(0)}$ with $\pi(Y_{(0)})=B$. 
Because $(K_S+Y_{(0)})\cdot Y_{(0)} \leq 0$ by Step~1, 
$Y_{(0)}$ is a rational curve or an elliptic curve. 
But, since $B$ is irrational and $\pi(Y_{(0)})=B$, 
$Y_{(0)}$ must be an elliptic curve. 
Then, so is $B$. 
Moreover if an indecomposable curve 
of canonical type $Y$ is reducible, 
then every prime component of $Y$ must be $\mathbb{P}^1$. 

Indeed, for every prime component $Y_{(i)}$, 
we have $(K_S+Y_{(i)})\cdot Y_{(i)}\leq 0$. 
Assume that a prime component $Y_{(0)}$ satisfies 
$(K_S+Y_{(0)})\cdot Y_{(0)}=0$. 
Then, by $K_S\cdot Y_{(0)}=0$, 
we have $Y_{(0)}^2=0$. 
Since $Y\cdot Y_{(0)}=0$, we obtain $Y_{(i)}\cdot Y_{(0)}=0$ 
for every prime component $Y_{(i)}$. 
Because $Y$ is connected, $Y$ must be irreducible. 
\end{step}

\begin{step}
In this step, we prove that the coefficient $\delta$ of $Y$ in $\Delta_S$ 
satisfies $0\leq \delta \leq 1$. 

Assume the contrary, that is, we assume $\delta>1$ and 
derive a contradiction. 
Take the proper transform $Y_V$ of $Y$ in $V$. 
We see that the coefficient of $Y_V$ in $\Delta_V$ is $\delta$. 
Then $Y_V$ is contracted by $f$ because the assumption of the boundary. 
Since $X$ is \Q-factorial, we can apply Theorem~\ref{irratqfac} and 
this is a contradiction. 
\end{step}

\begin{step}
In this step, we prove that $\mathcal{O}_Y(Y)$ is a torsion. 

By Step~1, we have $Y\cdot (\Delta_S-\delta Y)=0$. 
This means $\Supp Y\cap \Supp(\Delta_S-\delta Y)=\emptyset$. 
Thus, in $\Pic~Y$, we obtain 
$$\mu Y=Z=m(K_S+\Delta_S)=m(K_S+\delta Y)=m(-Y+\delta Y). $$
Therefore we have $(m(1-\delta)+\mu)Y=0$ in $\Pic~Y$. 
By $m(1-\delta)+\mu>m(1-\delta)\geq 0$, 
$\mathcal{O}_Y(Y)$ must be a torsion. 
\end{step}

\begin{step}
Let $r$ be the order of the torsion $\mathcal{O}_Y(Y)$. 
In this step, we prove that 
$$H^1(S, \mathcal O_S(K_S+tY))=0$$
for $1\leq t \leq r$ by induction. 

Let us consider the exact sequence
$$0 \to \mathcal{O}_S(K_S) \to \mathcal{O}_S(K_S+Y) 
\to \omega_Y\to 0.$$
If the induced map 
$$H^0(S, \mathcal{O}_S(K_S+Y)) \to H^0(Y, \omega_Y)=k$$ 
is surjective, we get a contradiction by the same argument as Step~5. 
Therefore, this map is zero. 
Then, the injective map 
$$k=H^0(Y, \omega_Y) \to H^1(S, \mathcal{O}_S(K_S))=k$$
is bijective. 
This means that the map
$$H^1(S, \mathcal{O}_S(K_S+Y))\to H^1(Y, \omega_Y)$$
is injective. 
On the other hand, we have 
$h^2(S, \mathcal{O}_S(K_S+Y))=0$ by Serre duality. 
Then we obtain the surjective map 
$$H^1(Y, \omega_Y) \to H^2(S, \mathcal{O}_S(K_S)).$$
But this is bijective by Serre duality. 
Therefore we obtain 
$$H^1(S, \mathcal{O}_S(K_S+Y))=0$$
and this proves the case where $t=1$. 
When $1<t\leq r$, we have the exact sequence: 
$$H^1(S, K_S+(t-1)Y)\to 
H^1(S, K_S+tY)\to 
H^1(Y, K_S+tY).$$
By the induction hypothesis, we have $H^1(S, K_S+(t-1)Y)=0$. 
Moreover, we obtain 
$h^1(Y, \mathcal{O}_Y(K_S+tY))=h^0(Y, \mathcal{O}_Y(-(t-1)Y))=0$ 
since $r$ is the order of $\mathcal{O}_Y(Y)$. 
Thus we see $H^1(S, K_S+tY)=0$ for $1\leq t \leq r$. 
\end{step}

\begin{step}
By Step~11, we obtain a surjection 
$$H^0(S, \mathcal O_S(K_S+(r+1)Y)) \to H^0(Y, \mathcal O_Y(K_S+(r+1)Y)).$$
By 
$
\mathcal{O}_Y(K_S+(r+1)Y)=
\mathcal{O}_Y(K_Y+rY)=
\mathcal{O}_Y
$, 
there exists an effective member $W\in |K_S+(r+1)Y|$ free from $Y$. 
Set $\tilde{Z}:=Z-\mu Y$. 
We obtain the equation 
\begin{eqnarray*}
\mu Z+(r+1)mZ
&=&\mu m(K_S+\Delta_S)+(r+1)m(\mu Y+\tilde{Z})\\
&=&\mu m(K_S+(r+1)Y)+\mu m\Delta_S+(r+1)m\tilde{Z}\\
&=&\mu mW+\mu m\Delta_S+(r+1)m\tilde{Z}
\end{eqnarray*}
as Weil divisors. 
By considering the coefficients of $Y$ in both sides, 
we obtain 
$$(\mu+(r+1)m)\mu=\mu m \delta. $$
\end{step}
But these two numbers are different by $0 \leq \delta \leq 1$. 
This is a contradiction. 
\end{proof}

\begin{rem}
If ${\rm char}~k>0$, then we do not need 
Step~11 and Step~12 in the proof of 
Theorem~\ref{qabundancekappazero} 
by using (1) of Proposition~\ref{icct1}. 
\end{rem}

\subsection{Abundance theorem for $\mathbb{R}$-divisors ($k\neq\overline{\mathbb F}_p$)}

In this section, we establish the abundance theorem 
in the case where $\Delta$ is an $\mathbb{R}$-boundary. 
We fix the following notations. 

\begin{nota}
Let $X$ be a projective normal \Q-factorial surface. 
We fix prime divisors $\Lambda_1,\cdots, \Lambda_s$ and 
a positive integer $\lambda\in\mathbb Z_{>0}$. 
Let 
$$\mathcal{L}:=\{B\in \sum\mathbb{R}\Lambda_i\,|\,0\leq B\leq \lambda\sum \Lambda_i\}.$$
Let 
$$M(X, \mathcal{L}):=\max (\{3\}\cup \{-(K_X+\lambda\Lambda_{i})\cdot \Lambda_{i}\})$$ 
where $i$ ranges over $1\leq i\leq s$. 
\end{nota}

Lemma~\ref{extlengthlem} and Proposition~\ref{ratpolytope} 
play key roles in this section.
The arguments are extracted from \cite[Section 3]{Birkar}. 

\begin{lem}\label{extlengthlem}
If $R$ is an extremal ray of $\overline{NE}(X)$ spanned by a curve, 
then there exists a curve $C$ such that $R=\mathbb{R}_{\geq 0}[C]$ and 
$-(K_X+B)\cdot C \leq M(X, \mathcal{L})$ for 
all $B\in \mathcal{L}$. 
\end{lem}

\begin{proof}
Let $H$ be an ample line bundle on $X$. 
Take a curve $C$ with 
$$R=\mathbb{R}_{\geq 0}[C]\,\,\,{\rm and}\,\,\,H\cdot C=\min\{H\cdot D\}$$
where $D$ ranges over curves generating $R$. 
We want to prove that $C$ satisfies the desired condition. 
Set $B\in \mathcal{L}$. 
If $-(K_X+B)\cdot C \leq 0$, then there is nothing to prove. 
Thus we may assume that $-(K_X+B)\cdot C > 0$. 
This means that $R$ is a $(K_X+B)$-negative extremal ray. 
Then, by Theorem~\ref{ct} and Remark~\ref{ray3}, 
there exists a curve $C'$ such that 
$R=\mathbb{R}_{\geq 0}[C']$ and 
$$-(K_X+B)\cdot C' \leq L(X, B)=\max(\{3\}\cup \{-(K_X+B)\cdot \Lambda_{\mu}\})$$ 
where $\Lambda_{\mu}$ ranges over the prime components $\Lambda_{\mu}$ of $B$ with $\Lambda_{\mu}^2<0$. 
Here, by the definition of $M(X, \mathcal{L})$, we have 
$L(X, B)\leq M(X, \mathcal{L}).$ 
Thus we obtain 
$$-(K_X+B)\cdot C' \leq M(X, \mathcal{L}).$$
By 
$$\frac{-(K_X+B)\cdot C}{H\cdot C}=\frac{-(K_X+B)\cdot C'}{H\cdot C'},$$
we have 
\begin{eqnarray*}
{-(K_X+B)\cdot C}&=&(-(K_X+B)\cdot C')\frac{H\cdot C}{H\cdot C'}\\
&\leq& -(K_X+B)\cdot C'\\
&\leq& M(X, \mathcal{L}).
\end{eqnarray*}
This completes the proof.
\end{proof}

\begin{dfn}
For an $\mathbb{R}$-divisor $B\in \sum_{i=1}^s\mathbb R\Lambda_i$ and 
for its prime decomposition $B=\sum r_i\Lambda_i$, 
we define 
$$||B||:=(\sum |r_i|^2)^{\frac{1}{2}}$$
where $|r_i|$ is the absolute value of $r_i$. 
\end{dfn}

\begin{prop}\label{ratpolytope}
Let $\Gamma$ be a $\mathbb{Q}$-divisor on $X$. 
Let $M$ be a positive real number. 
\begin{enumerate}
\item{Let $\Delta\in\mathcal L$. 
Then, there exists a positive real number $\epsilon$ 
depending on $X$, $\mathcal L$, $\Delta$, $\Gamma$ and $M$, 
which satisfy the following property. 
Let $C$ be a curve on $X$ such that 
$-(K_X+\Gamma+B)\cdot C \leq M$ for 
all $B\in \mathcal{L}$. 
If $(K_X+\Gamma+\Delta)\cdot C>0$, then 
$(K_X+\Gamma+\Delta)\cdot C>\epsilon$.}
\item{Let $\Delta\in\mathcal L$. 
Then, there exists a positive real number $\delta$, 
depending on $X$, $\mathcal L$, $\Delta$, $\Gamma$ and $M$, 
which satisfy the following property. 
If a curve $C'$ in $X$ and an $\mathbb{R}$-divisor 
$B_0\in \mathcal{L}$ satisfy $||B_0-\Delta||<\delta$, 
$(K_X+\Gamma+B_0)\cdot C'\leq 0$ and 
$-(K_X+\Gamma+B)\cdot C' \leq M$ for 
all $B\in \mathcal{L}$, then 
$(K_X+\Gamma+\Delta)\cdot C'\leq 0$.} 
\item{Let $\{C_t\}_{t\in T}$ be 
a set of curves 
such that $-(K_X+\Gamma+B)\cdot C_t\leq M$ 
for all $B\in \mathcal{L}$. 
Then, 
the set 
$$\mathcal{N}_T(\Gamma):=\{B\in\mathcal{L}\,|\,
(K_X+\Gamma+B)\cdot C_t\geq 0 \,\,for\,\, any \,\,t \in T\}$$
is a rational polytope.}
\end{enumerate}
\end{prop}

\begin{proof}
Note that, for every $B\in \mathcal{L}$, 
we obtain the irreducible decomposition  
$$B=\sum_{i=1}^s l_i\Lambda_i$$
for some real numbers $l_i$ with $0\leq l_i \leq \lambda$. 

(1)
We can write $\Delta:=\sum l_i\Lambda_i$ as above. 
Then we have 
$$(K_X+\Gamma+\Delta)\cdot C=\sum l_i(K_X+\Gamma+\Lambda_i)\cdot C.$$
Suppose $(K_X+\Gamma+\Delta)\cdot C<1$. 
Then we have 
\begin{eqnarray*}
l_i(K_X+\Gamma+\Lambda_i)\cdot C 
&<&1-\sum_{j\neq i}l_j(K_X+\Gamma+\Lambda_j)\cdot C\\
&\leq&1+\sum_{j\neq i}l_jM\\
&\leq&1+(s-1)\lambda M.
\end{eqnarray*}
Thus, if $l_i\neq 0$, then we obtain 
$$-M\leq (K_X+\Gamma+\Lambda_i)\cdot C 
<\frac{1}{l_i}(1+(s-1)\lambda M).$$
Since $X$ is \Q-factorial, the \Q-divisor $K_X+\Gamma+\Lambda_i$ 
is \Q-Cartier. 
This means that there are only finitely many possibilities 
for the number $(K_X+\Gamma+\Lambda_i)\cdot C$. 

Thus, if $(K_X+\Gamma+\Delta)\cdot C<1$, 
then there are only finitely many possibilities 
for the number 
$(K_X+\Gamma+\Delta)\cdot C=\sum l_i(K_X+\Gamma+\Lambda_i)\cdot C$. 
Therefore we can find the desired number $\epsilon$.

(2)
Suppose that the statement is not true. 
Then, for an arbitrary positive real number $\delta$, 
there exist a curve $C'$ and an $\mathbb{R}$-divisor 
$B_0\in \mathcal{L}$ 
which satisfy $||B_0-\Delta||<\delta$, 
$(K_X+\Gamma+B_0)\cdot C'\leq 0$, 
$-(K_X+\Gamma+B)\cdot C' \leq M$ for 
all $B\in \mathcal{L}$ and
$(K_X+\Gamma+\Delta)\cdot C'> 0.$ 
Set $\delta:=1/m$ for any $m\in \mathbb{Z}_{>0}.$ 
Then we obtain an infinite sequence of 
curves $C_m$ and $B_m\in \mathcal{L}$ which satisfy 
$$(K_X+\Gamma+B_m)\cdot C_m\leq 0,$$
$$-(K_X+\Gamma+B)\cdot C_m \leq M\,\,\,\,
{\rm for\,\,all\,\,} B\in \mathcal{L} {\rm \,\,and}
\,\,\,\,$$
$$(K_X+\Gamma+\Delta)\cdot C_m>0$$
and $||B_m-\Delta||$ converges to zero. 
Let $\Delta=\sum l_i\Lambda_i$ and $B_m=\sum l_{i,m}\Lambda_i$ as above. 
Then we see $l_i=\lim l_{i,m}$. 
Here, for each $j$, the set 
$\{(K_X+\Gamma+\Lambda_j)\cdot C_m\}_{m}$ has a lower bound $-M$. 

Let us show that, 
if $l_j\neq 0$, then the set $\{(K_X+\Gamma+\Lambda_j)\cdot C_m\}_{m}$ has an upper bound. 
Since $0<l_j=\lim l_{j,m}$, we may assume 
$l_{j,m}>0$ for all $m$ by replacing the sequence with 
a suitable sub-sequence. 
By the inequality 
$$0 \geq (K_X+\Gamma+B_m)\cdot C_m
=\sum l_{i,m}(K_X+\Gamma+\Lambda_i)\cdot C_m,$$
we have 
\begin{eqnarray*}
(K_X+\Gamma+\Lambda_j)\cdot C_m&\leq& \frac{1}{l_{j,m}}
(-\sum_{i\neq j} l_{i,m}(K_X+\Gamma+\Lambda_i)\cdot C_m)\\
&\leq&\frac{1}{l_{j,m}}
(\sum_{i\neq j} l_{i,m}M)\\
&\leq&\frac{1}{l_{j,m}}
(s-1)\lambda M
\end{eqnarray*}
Since the set $\{1/l_{j,m}\}_m$ has an upper bound, 
the set $\{(K_X+\Gamma+\Lambda_j)\cdot C_m\}_m$ also has an upper bound. 
This is what we want to show. 

Then, for $m\gg 0$, we have 
\begin{eqnarray*}
&&(K_X+\Gamma+B_m)\cdot C_m\\
&=&
(K_X+\Gamma+\Delta)\cdot C_m+
\sum (l_{i,m}-l_i)(K_X+\Gamma+\Lambda_i)\cdot C_m\\
&>&\epsilon+\sum (l_{i,m}-l_i)(K_X+\Gamma+\Lambda_i)\cdot C_m\\
&=&\epsilon+\sum_{l_i\neq 0} (l_{i,m}-l_i)(K_X+\Gamma+\Lambda_i)\cdot C_m
+\sum_{l_i=0} l_{i,m}(K_X+\Gamma+\Lambda_i)\cdot C_m\\
&\geq&\epsilon
+\sum_{l_i\neq 0} (l_{i,m}-l_i)(K_X+\Gamma+\Lambda_i)\cdot C_m
+\sum_{l_i=0} l_{i,m}(-M)\\
&>&0.
\end{eqnarray*}
The first inequality follows from (1). 
The third inequality follows when $m\gg 0$. 
Note that, if $l_i\neq 0$, then 
the set $\{(K_X+\Gamma+\Lambda_i)\cdot C_m\}_i$ is bounded from the both sides. 
This is a contradiction. 

(3)
We show the assertion by the induction on 
$\dim \mathcal L$. 
If $\dim \mathcal L=0$, then there is nothing to show. 
Thus, we assume $\dim \mathcal L>0$. 
We may assume that for each $t \in T$ 
there exists $B\in \mathcal{L}$ with $(K_X+\Gamma+B)\cdot C_t<0$.

We see that $\mathcal{N}_T(\Gamma)$ is a compact set. 
Then, by (2) and by the compactness of $\mathcal{N}_T(\Gamma)$, 
there exist \R-divisors $\Delta_1, \cdots, \Delta_n \in \mathcal{N}_T(\Gamma)$ 
and positive real numbers $\delta_1>0, \cdots, \delta_n>0$ such that
$\mathcal{N}_T(\Gamma)$ is covered by 
$\mathcal B_i:=\{B\in \mathcal L\,|\,||B-\Delta_i||<\delta_i\}$ 
and that if $B\in \mathcal B_i$ 
with $(K_X+\Gamma+B)\cdot C_t<0$ for some $t\in T$, 
then $(K_X+\Gamma+\Delta_i)\cdot C_t=0$. 
Set 
$$T_i:=\{t\in T\,|\,(K_X+\Gamma+B)\cdot C_t<0 
\,\,{\rm for\,\, some}\,\, B\in\mathcal B_i\}. $$
Then, for every $t\in T_i$, we have 
$(K_X+\Gamma+\Delta_i)\cdot C_t=0$. 
In particular, $\Delta_i$ is a \Q-divisor. 

Here, we prove 
$$\mathcal N_T(\Gamma)=\bigcap\mathcal N_{T_i}(\Gamma).$$
The inclusion $\mathcal N_T(\Gamma)\subset \bigcap\mathcal N_{T_i}(\Gamma)$ 
is obvious. 
Thus we want to prove $\mathcal N_T(\Gamma)\supset \bigcap\mathcal N_{T_i}(\Gamma)$. 
Let $B \not\in \mathcal N_T(\Gamma)$. 
Since $\mathcal{N}_T(\Gamma)$ is compact, 
we can find an element $B'\in \mathcal N_T(\Gamma)$ with 
$$||B'-B||=\min \{||B^*-B||\,\,|\,B^*\in \mathcal N_T(\Gamma)\}.$$
Here we have $B'\in \mathcal B_i$ for some $i$. 
Since $\mathcal B_i\cap \overline{BB'}$ is 
an open subset of $\overline{BB'}$ 
where $\overline{BB'}$ is the line segment, 
we have an element $B''$ 
such that $B''\in \mathcal B_i\cap \overline{BB'}$, 
$B''\neq B$ and $B''\neq B'$. 
This means that there is a real number $\beta$ with $0<\beta<1$ such that 
$$\beta B+(1-\beta)B'=B''.$$ 
We obtain 
$$\beta (K_X+\Gamma+B)+(1-\beta)(K_X+\Gamma+B')=K_X+\Gamma+B''.$$
Moreover, we see that $B''\not\in \mathcal N_T(\Gamma)$. 
Here, since $B''\in \mathcal B_i \setminus \mathcal N_T(\Gamma)$, 
we have $(K_X+\Gamma+B'')\cdot C_t<0$ for some $t\in T_i$. 
Thus we obtain the following inequality 
\begin{eqnarray*}
&&\beta(K_X+\Gamma+B)\cdot C_t\\
&=&
(K_X+\Gamma+B'')\cdot C_t-(1-\beta)(K_X+\Gamma+B')\cdot C_t\\
&<&-(1-\beta)(K_X+\Gamma+B')\cdot C_t\\
&\leq&0.
\end{eqnarray*} 
Therefore we have $(K_X+\Gamma+B)\cdot C_t<0$. 
This means $B\not\in \mathcal N_{T_i}(\Gamma)$. 

Therefore it is enough to prove that each $\mathcal N_{T_i}(\Gamma)$ 
is a rational polytope. 
By replacing $T$ with $T_i$, 
we may assume that there exists 
a \Q-divisor $\Delta_{0}\in \mathcal N_T(\Gamma)$ 
such that $(K_X+\Gamma+\Delta_{0})\cdot C_t=0$ for every $t\in T$. 
Let $\mathcal L^1, \cdots ,\mathcal L^u$ be 
the proper faces of $\mathcal L$ whose codimensions are one.  
Note that, for every $1\leq u'\leq u$, 
there exists a positive integer $i'$ such that 
\begin{equation}
\mathcal L^{u'}=\{B\in \sum_{i\neq i'}\mathbb{R}\Lambda_i\,
|\,0\leq B\leq \lambda\sum_{i\neq i'}\Lambda_i\}
\tag{I}\label{I}
\end{equation}
or that 
\begin{equation}
\mathcal L^{u'}=\lambda\Lambda_{i'}+\{B\in \sum_{i\neq i'}\mathbb{R}\Lambda_i\,
|\,0\leq B\leq \lambda\sum_{i\neq i'}\Lambda_i\}.
\tag{II}\label{II}
\end{equation}

Let us prove that each 
$\mathcal N_T^{u'}(\Gamma):=
\mathcal N_T(\Gamma)\cap \mathcal L^{u'}$ is 
a rational polytope. 
If $\mathcal L^{u'}$ satisfies the above equation (\ref{I}), 
then we see that 
$$\mathcal{N}_T^{u'}(\Gamma)=\{B\in\mathcal{L}^{u'}\,|\,
(K_X+\Gamma+B)\cdot C_t\geq 0 \,\,
{\rm for}\,\, {\rm any} \,\,t \in T\}.$$
Hence $\mathcal{N}_T^{u'}(\Gamma)$ is a rational polytope 
by the induction hypothesis. 
Thus assume that 
$\mathcal L^{u'}$ satisfies the above equation (\ref{II}). 
Set 
$\mathcal L^{u'}_0:=\{B\in \sum_{i\neq i'}\mathbb{R}\Lambda_i|
0\leq B\leq \lambda\sum_{i\neq i'}\Lambda_i\}.$
The equation (\ref{II}) implies 
$$\mathcal L^{u'}=\lambda\Lambda_{i'}+\mathcal L^{u'}_0.$$
Then we see that 
\begin{eqnarray*}
&&\mathcal{N}_T^{u'}(\Gamma)\\
&=&\{B\in\mathcal{L}^{u'}\,|\,
(K_X+\Gamma+B)\cdot C_t\geq 0 \,\,
{\rm for}\,\, {\rm any} \,\,t \in T\}\\
&=&\lambda\Lambda_{i'}+\{B_0\in\mathcal{L}^{u'}_0\,|\,
(K_X+\Gamma+\lambda\Lambda_{i'}+B_0)\cdot C_t\geq 0 \,\,
{\rm for}\,\, {\rm any} \,\,t \in T\}. 
\end{eqnarray*}
For all $B_0\in\mathcal{L}^{u'}_0$, 
we have the following inequality
$$-(K_X+\Gamma+\lambda\Lambda_{i'}+B_0)\cdot C_t\leq M.$$
Thus the set 
\begin{eqnarray*}
&&\mathcal{N}_T(\mathcal {L}^{u'}_0, \Gamma+\lambda\Lambda_{i'})\\
&:=&
\{B_0\in\mathcal{L}^{u'}_0\,|\,
(K_X+\Gamma+\lambda\Lambda_{i'}+B_0)\cdot C_t\geq 0 \,\,
{\rm for}\,\, {\rm any} \,\,t \in T\}
\end{eqnarray*}
is a rational polytope by the induction hypothesis. 
Therefore $\mathcal{N}_T^{u'}(\Gamma)$ 
is also a rational polytope and 
this is what we want to show. 

Here, take an arbitrary element 
$B\in \mathcal N_T(\Gamma)$ with $B\neq \Delta_{0}$. 
Then we can find $B'\in \mathcal L^{u'}$ for some $1\leq u'\leq u$ such that 
$B$ is on the line segment defined by $\Delta_0$ and $B'$. 
Since $(K_X+\Gamma+\Delta_0)\cdot C_t=0$ for all $t\in T$, 
we have $B'\in \mathcal N_T^{u'}(\Gamma)$. 
Thus we see that $\mathcal N_T(\Gamma)$ is the convex hull of 
$\Delta_0$ and all the $\mathcal N_T^{u'}(\Gamma)$. 
Hence $\mathcal N_T(\Gamma)$ is a rational polytope. 
\end{proof}

\begin{cor}\label{ratpolytopecor}
Let $\{R_t\}_{t\in T}$ be a family of extremal rays of 
$\overline{NE}(X)$ spanned by curves. 
Then the set 
$$\mathcal{N}_T:=\{B\in\mathcal{L}\,|\,
(K_X+B)\cdot R_t\geq 0 \,\,for\,\, any \,\,t \in T\}$$
is a rational polytope.
\end{cor}

\begin{proof}
By Lemma~\ref{extlengthlem}, 
for every $t\in T$, there exists a curve $C_t$ 
such that $R_t=\mathbb R_{\geq 0}[C_t]$ and 
$-(K_X+B)\cdot C_t\leq M(X, \mathcal L)$ for all $B\in \mathcal L$. 
Let $\Gamma:=0$ and $M:=M(X, \mathcal L)$. 
Then, we can apply Proposition~\ref{ratpolytope}. 
Therefore, the set $\mathcal{N}_T=\mathcal{N}_T(0)$
is a rational polytope.
\end{proof}

Now, we prove the abundance theorem with $\mathbb{R}$-coefficients. 

\begin{thm}\label{abundance}
Let $X$ be a projective normal \Q-factorial surface over $k$ and 
let $\Delta$ be an $\mathbb{R}$-boundary. 
If $k$ is not the algebraic closure of a finite field and 
$K_X+\Delta$ is nef, then $K_X+\Delta$ is semi-ample. 
\end{thm}

\begin{proof}
Let $\{R_t\}_{t\in T}$ be the set of all the extremal rays of 
$\overline{NE}(X)$ spanned by curves. 
Then 
\begin{eqnarray*}
\mathcal N_T&:=&\{B\in \mathcal L \,|\,
(K_X+B)\cdot R_t\geq 0\,\,{\rm for\,\,every}\,\,t\in T\}
\end{eqnarray*}
is a rational polytope by Corollary~\ref{ratpolytopecor}. 
Moreover, by Theorem~\ref{ct}, we see that 
\begin{eqnarray*}
\mathcal N_T&=&\{B\in \mathcal L \,|\,
(K_X+B)\cdot R_t\geq 0\,\,{\rm for\,\,every}\,\,t\in T\}\\
&=&\{B\in \mathcal L \,|\,\,\,K_X+B\,\,{\rm is\,\,nef}\}.
\end{eqnarray*}
Since $\Delta\in \mathcal N_T$, 
we can find \Q-divisors $\Delta_1, \cdots, \Delta_l$  
such that $\Delta_i\in \mathcal N_T$ for all $i$ and 
that $\sum r_i\Delta_i=\Delta$ 
where positive real numbers $r_i$ satisfy $\sum r_i=1$. 
Thus we have 
$$K_X+\Delta=\sum r_i(K_X+\Delta_i)$$
and $K_X+\Delta_i$ is nef. 
By Theorem~\ref{qabundance}, 
$K_X+\Delta_i$ is semi-ample. 
\end{proof}

\section{Normal surfaces over $\overline{\mathbb{F}}_p$}

\subsection{Contraction problem}

In this section, let $k$ be an arbitrary algebraically closed field 
and ${\rm char}\,k=p\geq 0$. 
As the introduction of this part, 
we consider the following question. 

\begin{ques}[Contraction problem]\label{contprob}
Let $X$ be a smooth projective surface over $k$ and 
let $C$ be a curve in $X$. 
If $C^2<0$, then $C$ is contractable? 
(i.e. Does there exist a birational morphism $f:X \to Y$
to an algebraic surface $Y$  such that $f(C')$ is one point 
iff $C'=C$ for every curve $C'$?) 
\end{ques}

\begin{ans}\label{Hironaka}
If $k \neq \overline{\mathbb{F}}_p$, 
then the answer to Question~\ref{contprob} is 
NO in general. 
\end{ans}

We only recall the method of its construction. 
For more details, see \cite[Example~5.7.3]{Hartshorne}.

\begin{proof}[Construction]
If we obtain an elliptic curve $C_0$ in \pr2 with rank $\geq$ 10, 
then we can construct a counter-example as follows. 
There are 10 points in $C_0$ which are linearly independent. 
Blow-up \pr2 at these 10 points. 
The proper transform $C$ of $C_0$ is not contractable. 
\end{proof}

By Fact~\ref{principle}, if $k \neq \overline{\mathbb{F}}_p$, 
then we can use this construction. 
On the other hand, if $k=\fff$, 
then we have the opposite answer. 

\begin{ans}[\cite{Artin}]
If $k=\fff$, then the answer to Question~\ref{contprob} is YES.
\end{ans}

To see this answer and its mechanism of this proof, 
we divide the verification into small pieces and 
prove more general following result.

\begin{prop}\label{fffcont}
Let $X$ be a projective normal \Q-factorial surface over $k$ and 
let $C$ be a curve in X. 
\begin{enumerate}
\item{If $C^2<0$, then there exists a nef and big divisor G such that 
$G\cdot C'=0$ iff $C'=C$ for any curve $C'$ in $X$. }
\item{If the restriction $G|_{C}$ of the divisor $G$ in $(1)$
is a torsion and ${\rm char}\,k=p>0$, then $G$ is semi-ample.}
\item{If $k=\fff$, then $G|_{C}$ is torsion.}
\end{enumerate}
\end{prop}

\begin{proof}
(1)Let $H$ be an ample divisor on $X$. 
We define a \Q-divisor $G$ and $q\in\mathbb Q_{>0}$ 
by $G=H+qC$ and $G\cdot C=0$. 
It is easy to check that $G$ satisfies the above conditions.\\
(2)Since $p>0$, we can use Keel's result (Theorem~\ref{keel}). 
Therefore, the semi-ampleness of $G$
is equivalent to the semi-ampleness of $G|_C$. 
But $G|_C$ is a torsion by the assumption. 
Thus, $G$ is semi-ample. \\
(3)This is an immediate consequence of Corollary~\ref{pic}. 
\end{proof}

For related results to this section, see \cite{Artin} and \cite{Badescu2}. 

\subsection{\Q-factoriality}

In this section, we prove the following two theorems. 

\begin{thm}\label{Q-factorial}
If $X$ is a normal surface over \ff, then $X$ is \Q-factorial. 
\end{thm}

\begin{thm}\label{decomposition}
Let $f:X \to Y$ be a proper birational morphism 
between normal surfaces over \ff, then 
$f$ factors into contractions of one curve. 
More precisely, there exist proper birational morphisms 
such that each 
$g_i:X_i \to X_{i+1}$ is a proper birational morphism 
between normal surfaces such that ${\rm Ex}(g_i)$ is an irreducible curve.
\end{thm}

The following lemma is the key in this section. 

\begin{lem}
Let $f:X \to Y$ be a proper birational morphism over $\overline{\mathbb{F}}_p$ 
from a normal \Q-factorial surface $X$ to a normal surface $Y$. 
Let ${\rm Ex}(f)=C_1 \cup \dots \cup C_r$. 
\begin{enumerate}
\item{There exists 
a proper birtional morphism $g:X \to Z$ to a normal surface $Z$ 
such that ${\rm Ex}(g)=C_1$.} 
\item{The morphism $f$ factors through Z.}
\item{$Z$ is \Q-factorial.}
\end{enumerate}
\end{lem}

\begin{proof}(1) If $X$ and $Y$ are proper, then 
the assertion follows from Proposition~\ref{fffcont}. 
Note that proper \Q-factorial surfaces are projective 
(cf. \cite[Lemma~2.2]{Fujino}). 
In general case, take the Nagata compactification. 
Note that normality and \Q-factoriality may break up by compactification. 
But by taking the normalization and the resolution of the locus of $\overline{X} \setminus X$, 
we may assume that these assumption. \\
(2) This is obvious.\\
(3) The assertion immediately follows from 
Proposition \ref{(3)}, Proposition \ref{(5)} and Proposition \ref{fffcont}.
\end{proof}

\begin{cor}\label{fpqfaccor}
Let $f:X \to Y$ be a proper birational morphism over $\overline{\mathbb{F}}_p$ 
from a normal \Q-factorial surface $X$ 
to a normal surface $Y$. Then $Y$ is \Q-factorial.
\end{cor}

\begin{proof}
By using the above lemma repeatedly, 
$f$ is factored into contractions of one curve and 
\Q-factoriality of $X$ descends to $Y$.
\end{proof}

By the same argument, Theorem~\ref{decomposition} follows from Theorem~\ref{Q-factorial}. 
Thus we only prove Theorem~\ref{Q-factorial}.

\begin{proof}[Proof of Theorem \ref{Q-factorial}]
Let $f:X' \to X$ be its resolution of singularities. 
Of course $X'$ is \Q-factorial. 
Therefore $X$ is also \Q-factorial by Corollary~\ref{fpqfaccor}. 
\end{proof}

\begin{rem}
Theorem~\ref{Q-factorial} follows from 
\cite[Corollary~14.22]{Badescu} and \cite[(24.E)]{Matsumura}. 
\end{rem}

\subsection{Theorems in Section 3}

In this section, we establish the theorems,
which we discussed in Section 3, over $\overline{\mathbb{F}}_p$ 
under much weaker assumptions.

\begin{thm}[Contraction theorem]\label{fpcont}
Let $X$ be a projective normal surface over $\overline{\mathbb{F}}_p$ and 
let $\Delta$ be an effective $\mathbb{R}$-divisor. 
Let $R=\mathbb{R}_{\geq 0}[C]$ be a $(K_X+\Delta)$-negative extremal ray. 
Then there exists a surjective morphism $\phi_R: X \to Y$ 
to a projective variety $Y$ with the following properties{\em{:}} 
{\em{(1)-(4)}}.
\begin{enumerate}
\item{Let $C'$ be a curve on $X$. 
Then $\phi_R(C')$ is one point iff $[C'] \in R$.}
\item{$(\phi_R)_*(\mathcal{O}_X)=\mathcal{O}_Y$.}
\item{If $L$ is an invertible sheaf with $L\cdot C=0$, 
then $nL=(\phi_R)^*L_Y$ for some invertible sheaf $L_Y$ on $Y$ and 
for some positive integer $n$.}
\item{$\rho(Y)=\rho(X)-1$.}
\end{enumerate}
\end{thm}

\begin{proof}
If $C^2\geq 0$, then we have
$$K_X\cdot C \leq (K_X+\Delta)\cdot C<0.$$
Then we can apply Theorem~\ref{cont}. 
Thus we may assume $C^2<0$. 
But this curve is contractable and 
the proofs of the remaining properties are the same as Theorem~\ref{cont}. 
\end{proof}

The following theorem is a known result (\cite[Corollary~14.29]{Badescu}). 
We give a minimal model theoretic proof. 

\begin{thm}[Finite generation theorem]
Let $X$ be a projective normal surface over $\overline{\mathbb{F}}_p$ and 
let $D$ be a \Q-divisor. 
Then $R(X, D)=\bigoplus\limits_{m \geq 0} H^0(X, \llcorner mD\lrcorner)$
is a finitely generated $\overline{\mathbb{F}}_p$-algebra. 
\end{thm}

\begin{proof}
We may assume that $\kappa(X, D) \geq 1$. 
Then in particular $D$ is effective. 
If there is a curve with $D\cdot C<0$, then $C^2<0$ and $C$ is contractable. 
Let $f:X\to Y$ be the contraction of $C$. 
Note that we obtain 
$D=f^*f_*D+qC$, for a positive rational number $q$. 
Therefore we may assume that $D$ is nef. 
If $\kappa(X, D)=1$, then $D$ is semi-ample by Proposition \ref{kappa1}.
If $\kappa(X, D)=2$, then $D$ is semi-ample by the following proposition.
\end{proof}

\begin{prop}\label{fpnefbig}
Let $X$ be a projective normal surface over $\overline{\mathbb{F}}_p$. 
If $D$ is a nef and big \Q-divisor, then $D$ is semi-ample.
\end{prop}

\begin{proof}
If there is a curve $C$ such that $D\cdot C=0$, 
then $C^2<0$ and $C$ is contractable. 
Let $f:X \to Y$ be its contraction and $f^*D_Y=D$. 
It is sufficient to prove $D_Y$ is semi-ample. 
Repeating the same procedure, 
we see that $D$ is a pull-back of an ample divisor.
\end{proof}

\begin{thm}[Non-vanishing theorem]\label{fpnonvanishing}
Let $X$ be a projective normal surface over $\overline{\mathbb{F}}_p$ and 
let $\Delta$ be an effective \Q-divisor. 
If $K_X+\Delta$ is nef, then $\kappa(X, K_X+\Delta) \geq 0$
\end{thm}

The proof of this theorem heavily depends on the argument in \cite{Masek}.

\begin{proof}
We may assume that $X$ is smooth by replacing it with its minimal resolution. 

\setcounter{step}{0}
\begin{step}
If $\kappa(X, K_X) \geq 0$, then 
$\kappa(X, K_X+\Delta) \geq \kappa(X, K_X) \geq 0$. 
Thus we may assume that $\kappa(X, K_X)=-\infty$.
\end{step}

\begin{step}
In this step, we show that we may assume 
$K_X+\Delta$ is not numerically trivial and 
$h^2(X, m(K_X+\Delta))=0$ for $m \gg 0$. 

If $K_X+\Delta$ is numerically trivial, then 
$K_X+\Delta$ is a torsion by Fact~\ref{principle}. 
Thus we obtain $n(K_X+\Delta)\sim 0$ for some integer $n$ 
and $\kappa(X, K_X+\Delta)=0.$ 
Therefore we may assume that $K_X+\Delta$ is not numerically trivial. 
Then we obtain $h^2(X, m(K_X+\Delta))=h^0(X, K_X-m(K_X+\Delta))=0$
for $m \gg 0$. 
(We have $(K_X+\Delta)\cdot C>0$ for some curve. 
Then there exists an ample divisor $A$ and an effective divisor $E$ 
such that $A=C+E$. 
By the nefness of $K_X+\Delta$, we obtain $(K_X+\Delta)\cdot A>0$. 
Then since $(K_X-m(K_X+\Delta))\cdot A<0$ for sufficiently large integer $m$, 
we obtain $h^0(X, K_X-m(K_X+\Delta))=0$.)
\end{step}

\begin{step}
In this step we show that we may assume $(K_X+\Delta)^2=0$. 

Suppose the contrary, that is, suppose $(K_X+\Delta)^2>0$. 
Then $K_X+\Delta$ is nef and big. 
Then we obtain $h^0(X, m(K_X+\Delta))>0$ for some positive integer $m$, 
and $\kappa(X, K_X+\Delta)\geq 0$. 
\end{step}

We consider the two cases: $X$ is rational or irrational.

\begin{step}
In this step, we prove the assertion when $X$ is rational. 

Now $\chi(\mathcal{O}_X)$=1 because $X$ is rational. 
Then, the Riemann--Roch theorem shows that 
\begin{eqnarray*}
&&h^0(X, m(K_X+\Delta))\\
&=&h^1(X, m(K_X+\Delta))+1+\frac{1}{2}m(K_X+\Delta)\cdot(m(K_X+\Delta)-K_X) 
\end{eqnarray*}
where $m \gg 0$.
The right hand side is positive, because 
\begin{eqnarray*}
&&m(K_X+\Delta)\cdot(m(K_X+\Delta)-K_X)\\
&=&m(K_X+\Delta)\cdot((m-1)(K_X+\Delta)+\Delta)\geq 0
\end{eqnarray*}
by the nefness of $K_X+\Delta$. 
This is what we want to show. 
\end{step}

Thus we may assume that $X$ is an irrational ruled surface. 
We divide the proof into three cases: 
$(K_X+\Delta)\cdot K_X<0$, 
$(K_X+\Delta)\cdot K_X>0$ and 
$(K_X+\Delta)\cdot K_X=0$.

\begin{step}
We assume that $X$ is irrational and $(K_X+\Delta)\cdot K_X<0$. 

By Step~2, Step~3 and the Riemann-Roch theorem, 
$h^0(X, m(K_X+\Delta))>0$ for some large integer $m$. 
This is what we want to show. 
\end{step}

\begin{step}
We assume that $X$ is irrational and $(K_X+\Delta)\cdot K_X>0$. 

Since $(K_X+\Delta)^2=0$ and $(K_X+\Delta)\cdot K_X>0$, 
we obtain $(K_X+\Delta)\cdot \Delta<0$. 
This contradicts the nefness of $K_X+\Delta$. 
\end{step}

\begin{step}
We assume that $X$ is irrational and $(K_X+\Delta)\cdot K_X=0$. 

We assume $\kappa(X, K_X+\Delta)=-\infty$ and derive a contradiction. 
By $(K_X+\Delta)\cdot K_X=0$ and $(K_X+\Delta)^2=0$, 
we obtain $(K_X+\Delta)\cdot \Delta=0$. 
Let $C$ be an arbitrary prime component of $\Delta$. 
Since $\Delta\neq 0$, we can take such a curve. 
(Indeed, if $\Delta=0,$ 
then $K_X$ is nef. 
This contradicts that $X$ is a ruled surface. )
By $(K_X+\Delta)\cdot \Delta=0$ and the nefness of $K_X+\Delta$, 
we have $(K_X+\Delta)\cdot C=0$. 
By Fact~\ref{principle}, 
we obtain $n_1(K_X+\Delta)|_C \sim 0$ for some $n_1\in\mathbb Z_{>0}$. 
Then we get the following exact sequence: 
\[0 \to 
\mathcal{O}_X(n_1n_2(K_X+\Delta)-C) \to 
\mathcal{O}_X(n_1n_2(K_X+\Delta)) \to 
\mathcal{O}_C \to 
0\] 
for every $n_2\in\mathbb Z_{>0}.$ 
Here we want to prove that, for every $n_2\gg 0,$ 
$$h^2(X, n_1n_2(K_X+\Delta)-C)=0.$$
By Serre duality, 
we obtain 
$h^2(X, n_1n_2(K_X+\Delta)-C)=
h^0(X, K_X+C-n_1n_2(K_X+\Delta))$. 
This is zero, by the same argument as Step~2. 

Fix $n_2\gg 0$ and let $n:=n_1n_2$. 
By $h^2(X, n(K_X+\Delta)-C)=0$, we have a surjection 
$H^1(X, n(K_X+\Delta)) \to H^1(C, \mathcal{O}_C)$. 
This means $$h^1(X, n(K_X+\Delta)) \geq h^1(C, \mathcal{O}_C).$$ 
On the other hand, 
by $h^0(X, n(K_X+\Delta))=h^2(X, n(K_X+\Delta))=0$ and 
the Riemann--Roch theorem, 
\begin{eqnarray*}
-h^1(X, n(K_X+\Delta))&=&
\chi(\mathcal{O}_X)+\frac{1}{2}n(K_X+\Delta)\cdot\{n(K_X+\Delta)-K_X\}\\
&=&\chi(\mathcal{O}_X)=1-h^1(B, \mathcal{O}_B).
\end{eqnarray*} 
where $\pi:X \to B$ is the ruling. 
Hence we have 
$$h^1(B, \mathcal{O}_B)-1=
h^1(X, n(K_X+\Delta)) \geq h^1(C, \mathcal{O}_C). $$
This shows that $C$ is in some fiber of $\pi$. 
In particular, for a smooth fiber $F$, we have $C\cdot F=0$. 
Recall that $C$ is an arbitrary prime component of $\Delta$, 
then we obtain $\Delta\cdot F=0$. 
Thus we have 
$$0 \leq (K_X+\Delta)\cdot F=K_X\cdot F=-2.$$
\end{step}
This is a contradiction. 
\end{proof}

\begin{thm}[Abundance theorem]\label{fpabundance}
Let $X$ be a projective normal surface over $\overline{\mathbb{F}}_p$ and 
let $\Delta$ be an effective $\mathbb{R}$-divisor.
If $K_X+\Delta$ is nef, then $K_X+\Delta$ is semi-ample. 
\end{thm}

\begin{proof}
By the same proof as Theorem~\ref{abundance}, 
we may assume that $\Delta$ is a \Q-divisor. 
By Theorem~\ref{fpnonvanishing}, we have $\kappa(X, K_X+\Delta)\geq 0$. 
By Proposition~\ref{kappa1} and Proposition~\ref{fpnefbig}, 
we may assume $\kappa(X, K_X+\Delta)=0$. 
Then we can apply the argument 
of Step~1 and Step~2 in Theorem~\ref{qabundancekappazero}. 
By $(1)$ of Proposition~\ref{icct1}, 
we have $\kappa(S, Y)=1$ for indecomposable curves of canonical type $Y$
in $S$ over $\overline{\mathbb{F}}_p$. 
This contradicts $Z\neq0$ and $\kappa(S, Z)=0$.
\end{proof}

As an immediate corollary, 
we obtain the following basepoint free theorem.

\begin{thm}[Basepoint free theorem]\label{fpbpf}
Let $X$ be a projective normal surface over $\overline{\mathbb{F}}_p$ and 
let $D$ be a nef divisor. 
If $\kappa(X, qD-K_X) \geq 0$ for some positive rational number $q$, 
then $D$ is semi-ample. 
\end{thm}

\begin{proof}
Take $qD-K_X \sim_{\mathbb{Q}} \Delta$.
We obtain $qD \sim_{\mathbb{Q}} K_X+\Delta$ 
and can apply the abundance theorem.
\end{proof}

\subsection{Examples}

In this section, 
let $k$ be an algebraically closed field of arbitrary characteristic. 
We want to see the difference 
between $k=\overline{\mathbb{F}}_p$ and $k\neq \overline{\mathbb{F}}_p$ 
by looking at some examples. 

\begin{ex}[cf. Theorem~\ref{cont} and Theorem~\ref{fpcont}]
If $k\neq\fff$, then 
there exist a smooth projective surface $X$ over $k$ 
and an elliptic curve $C$ in $X$ such that, 
for an arbitrary positive real number $\epsilon$, 
$(K_X+(1+\epsilon)C)\cdot C<0$, $C^2<0$ and 
$C$ is not contractable. 
\end{ex}

\begin{proof}[Construction]
Consider the Answer~\ref{Hironaka} and its construction. 
There exist a smooth projective surface $X$ and 
an elliptic curve $C$ in $X$ such that 
$C^2=-1$ and $C$ is not contractable. 
Moreover, we have 
\begin{eqnarray*}
(K_X+(1+\epsilon)C)\cdot C
&=&(K_X+C)\cdot C+\epsilon C\cdot C\\
&=&\epsilon C\cdot C<0.
\end{eqnarray*} 
This is what we want to show.
\end{proof}

\begin{ex}[cf. Theorem~\ref{nonv} and Theorem~\ref{fpnonvanishing}]
If $k\neq\overline{\mathbb{F}}_p$, then 
there exist a smooth projective surface $X$ over $k$ and 
curves $C_1$ and $C_2$ in $X$ such that 
\begin{eqnarray*}
K_X+(1+\epsilon)C_1+(1-\epsilon)C_2 \equiv 0\,\,\,\,\,\,{\rm and}\\
\kappa(X, K_X+(1+\epsilon)C_1+(1-\epsilon)C_2)=-\infty
\end{eqnarray*}
for an arbitrary positive rational number $\epsilon$.
\end{ex}

\begin{proof}[Construction]
Let $P:=\mathbb{P}^1$ and let $E$ be an arbitrary elliptic curve. 
Set $X_0:=P \times E$. 
We construct $X$ by applying 
the elementary transform to \pr1-bundle $X_0$ at appropriate two points. 
Let $e_1$ and $e_2$ be points in $E$ which are linearly independent. 
Fix two different points $p_1$ and $p_2$ in $P$ and 
set $S_1:=\{p_1\} \times E$ and $S_2:=\{p_2\} \times E$. 
Then we see 
$$K_{X_0} \sim_{\mathbb{Q}}-(1+\epsilon)S_1-(1-\epsilon)S_2$$ 
for an arbitrary rational number $\epsilon$. 
Let $x_1:=(p_1, e_1)$ and $x_2:=(p_2, e_2)$. 
We take the elementary transform of $X_0$ at $x_1$ and $x_2$, 
and obtain $X$. 
(First blowup at $x_1$. 
Then the proper transform of the fiber 
through $x_1$ is a $(-1)$-curve. 
Second contract this $(-1)$-curve and we get another \pr1-bundle.
Repeat the same thing to $x_2$.) 
Let $C_1$ and $C_2$ be 
the proper transforms of $S_1$ and $S_2$ respectively, and  
$F_1$ and $F_2$ be the fibers corresponding to $x_1$ and $x_2$ respectively. 
Then we see 
$$K_X \sim_{\mathbb{Q}}
-(1+\epsilon)C_1-(1-\epsilon)C_2-\epsilon F_1+\epsilon F_2,$$ 
which implies 
$$K_X+(1+\epsilon)C_1+(1-\epsilon)C_2 \sim_{\mathbb{Q}} 
\epsilon (-F_1+F_2).$$ 
This divisor is numerically trivial. 
Here we want to show that $\kappa(X, -F_1+F_2)=-\infty$, 
that is, $-F_1+F_2$ is not a torsion. 
Consider the ruling $\pi: X \to E$ and
one of its sections $\sigma: E \to X$. 
Then we have $F_1=\pi^*e_1$ and $F_2=\pi^*e_2$. 
Linear independence of $e_1$ and $e_2$ shows 
$$\mathcal O_X(n(-F_1+F_2)) \not\simeq \mathcal O_X.$$ 
Indeed, if $\mathcal O_X(n(-F_1+F_2)) \simeq \mathcal O_X$, 
then we have $\pi^*\mathcal O_E(n(-e_1+e_2)) \simeq \mathcal O_X$. 
Then, we obtain 
$$\mathcal O_E(n(-e_1+e_2))\simeq \sigma^*\pi^*\mathcal O_E(n(-e_1+e_2)) 
\simeq \sigma^*\mathcal O_X \simeq \mathcal O_E.$$ 
This is a contradiction.
\end{proof}

\begin{ex}[cf. Theorem~\ref{qabundance} and Theorem~\ref{fpabundance}]
If $k\neq\overline{\mathbb{F}}_p$, then
there exist a projective smooth surface $X$ over $k$ and 
an elliptic curve $C$ in $X$ such that, 
for an arbitrary positive rational number $\epsilon$, 
$K_X+(1+\epsilon)C$ is nef, 
$\kappa(X, K_X+(1+\epsilon)C) \geq 0$ and 
$K_X+(1+\epsilon)C$ is not semi-ample. 
\end{ex}

\begin{proof}[Construction]
Set $X_0:=\mathbb{P}^2$. 
Let $C_0$ be an arbitrary elliptic curve in $X_0$ and 
let $P_1, \cdots, P_9$ be points in $C_0$ which are linearly independent. 
Blowup at these nine points, then we obtain the surface $X$ and 
let $C$ be the proper transform of $C_0$. 
By $K_{X_0}=-C_0$, we have $K_X=-C$. 
Then 
$$K_X+(1+\epsilon)C=\epsilon C$$ 
is nef by $C^2=0$. 
It is obvious that $\kappa(X, K_X+(1+\epsilon)C) \geq 0$. 
We prove that $K_X+(1+\epsilon)C$ is not semi-ample. 
It is sufficient to prove $\kappa(X, C)=0$. 
Suppose the contrary, 
that is, suppose $\kappa(X, C)\geq 1$. 
Then we obtain $nC \sim D$ for some non-zero effective divisor $D$ 
with $C\not\subset \Supp D$. 
Since $C\cdot D=0$, $\Supp(f_*(D)|_{C_0})$ must be contained in $P_1, \cdots, P_9$. 
This means 
$n_1P_1+ \cdots +n_9P_9:=f_*(D)|_{C_0} \sim 3nL|_{C_0}$. 
Here $L$ is a line in $X_0$. 
But this means $n_1P_1+ \cdots +n_9P_9=0$ 
in the group structure of $C_0$. 
This is a contradiction.
\end{proof}

\section{Log canonical surfaces}

\subsection{Log canonical singularities}
In this section, 
we describe the log canonical singularities in surfaces 
by using the contraction theorem (Theorem~\ref{cont}). 

\begin{dfn}
We say a pair $(X, \Delta)$ is {\em a log canonical surface} if 
a normal surface $X$ and an $\mathbb{R}$ divisor $\Delta$ 
satisfy the following properties. 
\begin{enumerate}
\item{$K_X+\Delta$ is $\mathbb{R}$-Cartier.}
\item{For an arbitrary proper birational morphism 
$f:Y \to X$ and the divisor $\Delta_Y$ defined by 
$$K_Y+\Delta_Y=f^*(K_X+\Delta),$$
the inequality $\Delta_Y\leq 1$ holds.}
\item{$\Delta$ is effective.}
\end{enumerate}
\end{dfn}

First, we pay attention to only one singular point.

\begin{dfn}
We say $(X, \Delta)$ is {\em a local situation of a log canonical surface} if 
it satisfies the following properties. 
\begin{enumerate}
\item{The pair $(X, \Delta)$ is a log canonical surface.}
\item{There exists only one singular point $x\in X$.}
\item{All prime components of $\Delta$ contain $x$.}
\end{enumerate}
\end{dfn}

\begin{thm}\label{lcgermsing}
Let $(X, \Delta)$ be a local situation of a log canonical surface and 
let $f:Y \to X$ be the minimal resolution of $X$. 
Then, there exists a sequence of proper birational morphisms 
$$
f:Y=:Y_0\overset{f_0}\to Y_1\overset{f_1}\to \cdots 
\overset{f_{m-1}}\to Y_m=:Z\overset{g}\to X 
$$
with the following properties. 
\begin{enumerate}
\item[(1)]{Each $Y_i$ is a normal \Q-factorial surface.}
\item[(2)]{Each $f_i$ is a proper birational morphism and 
$E_i:={\rm Ex}(f_i)$ is an irreducible curve.}
\item[(3)]{Each $E_i$ satisfies $(K_{Y_i}+E_i)\cdot E_i<0$.}
\item[(4)]{One of {\em{(a)}} and {\em{(b)}} holds. 
\begin{enumerate}
\item[(a)]{$g$ is an isomorphism.}
\item[(b)]{$\Delta=0$ and $E:={\rm Ex}(g)$ is an irreducible curve such that $(K_Y+E)\cdot E=0$.}
\end{enumerate}}\end{enumerate}
\end{thm}

\begin{proof}
We assume that we obtain 
$$f:Y=:Y_0\overset{f_0}\to Y_1\overset{f_1}\to\cdots\overset{f_{j-1}}\to Y_j\overset{G}\to X$$ 
such that 
each $Y_i$ and each $f_i$ satisfy (1), (2) and (3). 

We prove that, if we can find a $G$-exceptional proper curve $E_j$ such that 
$(K_{Y_j}+E_j)\cdot E_j<0$, then 
we obtain a contraction of $E_j$ 
$$f_j:Y_j\to Y_{j+1}$$
to a \Q-factorial surface $Y_{j+1}$. 
If $X$ and $Y_j$ are proper, then we obtain the required morphism $f_j$
by Theorem~\ref{adjunction}, Proposition~\ref{p1cont}, Proposition~\ref{(3)} and 
Proposition~\ref{(5)}. 
Note that proper \Q-factorial surface is projective (cf. \cite[Lemma~2.2]{Fujino}). 
For the general case, take compactifications as follows. 
Let $\overline X$ be a proper normal surface and 
let $\overline \Delta$ be an \R-divisor on $\overline X$ such that 
$X\hookrightarrow \overline X$ is an open immersion, 
$(\overline X, \overline \Delta)$ is a local situation of log canonical surface and $\overline \Delta|_X=\Delta.$ 
We define $\overline{Y_j}$ by patching $Y_j$ and $\overline X$ along 
$Y_j\setminus {\rm Ex}(G)\simeq X\setminus \{x\}.$ 
Then, $\overline{Y_j}$ is \Q-factorial. 
Thus, we can reduce the problem to the case where $X$ and $Y_j$ are proper. 

If $G$ is an isomorphism, then we obtain (a). 
Thus, we may assume $G$ is not an isomorphism. 
Then, we can take a $G$-exceptional curve $E_j$. 
We obtain 
$$(K_{Y_j}+E_j)\cdot E_j\leq (K_{Y_j}+\Delta_j)\cdot E_j
=G^*(K_X+\Delta)\cdot E_j=0$$ 
where $\Delta_j$ is defined by $K_{Y_j}+\Delta_j=G^*(K_X+\Delta)$. 
We may assume $(K_{Y_j}+E_j)\cdot E_j=0$. 
In this case, the coefficient of $E_j$ in $\Delta_j$ is one. 

First, assume ${\rm Ex}(G)$ is reducible. 
Then, there exists a $G$-exceptional curve $E_j'$ 
such that $E_j \cap E_j'\neq \emptyset.$
Then, we have 
$$(K_{Y_j}+E_j')\cdot E_j'<(K_{Y_j}+\Delta_j)\cdot E_j'
=G^*(K_X+\Delta_j)\cdot E_j'=0.$$ 
This is what we want to show. 

Second, assume $E:={\rm Ex}(G)$ is irreducible. 
Since $(K_{Y_j}+E)\cdot E\leq 0$, we consider the two cases: 
$(K_{Y_j}+E)\cdot E<0$ or $(K_{Y_j}+E)\cdot E=0$. 
If $(K_{Y_j}+E)\cdot E<0$, then this means (a). 
Assume $(K_{Y_j}+E)\cdot E=0$. 
We show $\Delta=0$. 
If $\Delta\neq 0$, then we have
$$(K_{Y_j}+E)\cdot E<(K_{Y_j}+\Delta_j)\cdot E=0.$$ 
This means (b). 
\end{proof}

This theorem teaches us that 
non-\Q-factorial log canonical singularities are made by the case (b). 
Applying the same argument as above, 
we obtain the global version as follows.

\begin{thm}\label{lcsing}
Let $(X, \Delta)$ be a log canonical surface and 
let $f:Y \to X$ be the minimal resolution of $X$. 
Then, there exists a sequence of proper birational morphisms 
$$
f:Y=:Y_0\overset{f_0}\to Y_1\overset{f_1}\to \cdots 
\overset{f_{m-1}}\to Y_m=:Z\overset{g}\to X 
$$
with the following properties. 
\begin{enumerate}
\item[(1)]{Each $Y_i$ is a normal \Q-factorial surface.}
\item[(2)]{Each $f_i$ is a proper birational morphism and 
$E_i:={\rm Ex}(f_i)$ is an irreducible curve.}
\item[(3)]{Each $E_i$ satisfies $(K_{Y_i}+E_i)\cdot E_i<0$.}
\item[(4)]{One of {\em{(a)}} and {\em{(b)}} holds. 
\begin{enumerate}
\item[(a)]{$g$ is an isomorphism.}
\item[(b)]{$g({\rm Ex}(g))\cap\Supp\Delta=\emptyset$ and, for every point $Q\in g({\rm Ex}(g))$, 
$g^{-1}(Q)=:E$ is a proper irreducible curve such that $(K_Y+E)\cdot E=0$.}
\end{enumerate}}\end{enumerate}
In particular, $K_X$ and all prime components of $\Delta$ are \Q-Cartier.
\end{thm}

\begin{proof}
This follows from the same argument as Theorem~\ref{lcgermsing}.
\end{proof}

\begin{rem}
By Theorem~\ref{p1vanishing}, 
we see that $Z$ has at worst rational singularities. 
But we do not use this fact in this paper. 
\end{rem}

\subsection{Minimal model theory for log canonical surfaces}

In this section, 
we consider the minimal model theory for log canonical surfaces. 
We have already proved the cone theorem in Section~4. 
Thus let us consider the contraction theorem.

\begin{thm}[Contraction theorem]\label{lccont}
Let $(X, \Delta)$ be a projective log canonical surface and 
let $R=\mathbb{R}_{\geq 0}[C]$ be a $(K_X+\Delta)$-negative extremal ray. 
Then there exists a morphism $\phi_R: X \to Y$ 
to a projective variety $Y$ with the following properties{\em{:}} {\em{(1)-(5)}}.
\begin{enumerate}
\item{Let $C'$ be a curve on $X$. 
Then $\phi_R(C')$ is one point iff $[C'] \in R$.}
\item{$\phi_*(\mathcal{O}_X)=\mathcal{O}_Y$
\item{If $L$ is a line bundle with $L\cdot C=0$, 
then $nL=(\phi_R)^*L_Y$ for some line bundle $L_Y$ on $Y$ and 
for some positive integer $n$.}
\item{$\rho(Y)=\rho(X)-1$.}
\item{$(Y, (\phi_R)_*(\Delta))$ is a log canonical surface if $\dim Y=2$.}}
\end{enumerate}
\end{thm}

\begin{proof}[Proof of the case where $C^2>0$]
First, we prove that there exists a curve $D$ in $X$ 
such that 
$D$ is Cartier, $D$ is ample and $\mathbb{R}_{\geq 0}[C]=\mathbb{R}_{\geq 0}[D]$. 
Since $X$ is a projective normal surface, 
we can apply Bertini's theorem. 
Then the complete linear system of a very ample divisor has 
a smooth member $D$ 
such that $D \cap {\rm Sing}(X)=\emptyset$. 
Note that $D$ is a Cartier divisor. 
Let $f:X' \to X$ be the minimal resolution and 
let $D'$ be the proper transform of $D$. 
Since $f^*(C)$ is a nef and big divisor, 
we obtain 
$$nf^*(C)\sim D'+E$$
for some effective divisor $E$ and some positive integer $n$. 
By sending this equation by $f_*$, we obtain 
$$nC\sim D+f_*(E).$$
Since $\mathbb{R}_{\geq 0}[C]$ is extremal, 
we have $\mathbb{R}_{\geq 0}[C]=\mathbb{R}_{\geq 0}[D]$. 
Thus we obtain $\rho(X)=1$, 
because we can apply the same argument 
as the one in the proof of Theorem~\ref{cont}. 
Set $Y:={\rm Spec}\,k$. 
Then $\phi_R: X \to Y$ satisfies (1), (2) and (4).
We want to prove (3).
This follows from Lemma~\ref{lcabundancelemma} 
because $K_X+\Delta$ is anti-ample. 
\end{proof}

\begin{lem}\label{lcabundancelemma}
Let $(X, \Delta)$ be a projective log canonical surface. 
Let $L$ be a nef line bundle such that $L-(K_X+\Delta)$ is ample.
Then, $L$ is semi-ample.
\end{lem}

\begin{proof}
By Bertini's theorem, there exists a smooth curve $C$ such that 
$$nL-n(K_X+\Delta) \sim C,$$ 
$C \cap {\rm Sing}(X)=\emptyset$ and $C$ is not a component of $\Delta$. 
Let $f:X' \to X$ be the minimal resolution and 
let $C'$ be the proper transform of $C$. 
Then we obtain 
$$nf^*(L)-nf^*(K_X+\Delta) \sim f^*(C).$$ 
Since $f^*(C)=C'$, we have 
$$f^*(L)\sim_{\mathbb{Q}} K_{X'}+\Delta'+\frac{1}{n}C'$$
where $\Delta'$ is defined by $K_{X'}+\Delta'=f^*(K_X+\Delta)$. 
Since $\Delta'+(1/n)C'$ is a boundary, 
$f^*(L)$ is semi-ample
by Theorem~\ref{abundance} and Theorem~\ref{fpabundance}. 
Therefore so is $L$.
\end{proof}

In the proof of the case where $C^2\leq 0$ in Theorem~\ref{cont}, 
we only use the assumption of \Q-factoriality in the form that 
$K_X$ and $C$ are \Q-Cartier and 
$K_X+\Delta$ is $\mathbb{R}$-Cartier. 
Since $K_X$ is \Q-Cartier and 
$K_X+\Delta$ are $\mathbb{R}$-Cartier by Theorem~\ref{lcsing}, 
it is sufficient to prove that $C$ is \Q-Cartier.

\begin{proof}[Proof of the case where $C^2=0$]
It is sufficient to prove that 
$\mathbb{R}_{\geq 0}[C]=\mathbb{R}_{\geq 0}[D]$ 
for some \Q-Cartier curve $D$. 
Let $f:X' \to X$ be the minimal resolution. 
Since $f^*(C)^2=0$ and $f^*(C)\cdot K_{X'}<0$, 
we obtain $\kappa(X', f^*(C))=1$. 
Therefore, $f^*(C)$ is semi-ample by Proposition~\ref{kappa1}. 
We consider the fibration $\pi:X' \to B$ obtained by 
the complete linear system $|nf^*(C)|$ for some $n\gg 0$. 
For an arbitrary $f$-exceptional curve $E$, we have $E\cdot f^*(C)=0$. 
This means that an arbitrary exceptional curve is in some fiber of $\pi$. 
Thus there exists an integral fiber $D'$ of $\pi$ 
with $D' \cap {\rm Ex}(f)=\emptyset$
by Proposition~\ref{genericintegral}. 
This means that $f(D')=D$ is Cartier and $nC \equiv D$. 
This is what we want to show.
\end{proof}

\begin{prop}\label{genericintegral}
Let $\pi: X \to S$ be a dominant morphism 
from a normal surface $X$ to a curve $S$ 
with $\pi_*\mathcal{O}_X=\mathcal{O}_S$. 
Then there exists a non-empty open subset $S'$ in $S$ such that 
all scheme-theoretic fibers of 
$\pi|_{\pi^{-1}(S')}:\pi^{-1}(S') \to S'$ are integral. 
\end{prop}

\begin{proof}
See, for example, \cite[Cororally~7.3]{Badescu}.
\end{proof}

For the proof of the case where $C^2<0$, 
we consider the relation between the non-\Q-factorial 
log canonical singularities and 
extremal curves $C$ with $C^2<0$. 
Since we want to prove that $C$ is \Q-Cartier, 
it is necessary to consider the case where 
$C$ passes through the singular points of (b) in Theorem~\ref{lcgermsing}. 
The following lemma teaches us these singularities are actually \Q-factorial.

\begin{lem}\label{qcartierlemma}
Let $(X, \Delta=0)$ be a local situation of a log canonical surface and 
let $x$ be the singular point of $X$. 
Assume that this singularity is {\em (b)} in Theorem~\ref{lcgermsing}. 
If a proper curve $C$ in $X$ satisfies 
$C\cdot K_X<0$, $C^2<0$ and $x \in C$, then $X$ is \Q-factorial. 
\end{lem}

\begin{proof}
We use the notation in (b) of Theorem~\ref{lcsing}. 
It is sufficient to prove that $E\simeq \mathbb{P}^1$ 
by Proposition~\ref{(3)} and Proposition~\ref{(5)}. 
Let $C_Z$ be the proper transform of $C$. 
Then, we obtain
\begin{eqnarray*}
C_Z^2 &\leq& C_Z\cdot g^*(C)=C^2<0\\
C_Z\cdot K_Z &\leq& C_Z\cdot (K_Z+E)=C_Z\cdot g^*(K_X)=C\cdot K_X<0.
\end{eqnarray*}
Thus we obtain $C_Z \simeq \mathbb{P}^1$ and 
$C_Z$ is a curve generating a $K_Z$-negative extremal ray. 
Let $\phi:Z \to Z'$ be the contraction of $C_Z$. 
Since $\phi:E \to \phi(E)=:E'$ is a birational morphism, 
it is sufficient to prove that $E'\simeq \mathbb{P}^1$. 
We would like to prove 
$$(K_{Z'}+E')\cdot E'<0.$$ 
Let us consider the discrepancy $d$ defined by 
$$K_Z+E=\phi^*(K_{Z'}+E')+dC_Z.$$
Here, by taking the intersection with $E$, we obtain 
$$0=(K_{Z'}+E')\cdot E'+dC_Z\cdot E$$
by $(K_Z+E)\cdot E=0$. 
By $x \in C$, we see that $C_Z\cdot E$ is a positive number. 
Thus it is sufficient to prove that $d$ is a positive number. 
The following inequality 
$$0>K_X\cdot C=g^*(K_X)\cdot C_Z=(K_Z+E)\cdot C_Z=dC_Z^2$$
shows that $d$ is positive. 
\end{proof}

\begin{prop}\label{qcartier}
Let $(X, \Delta)$ be a log canonical surface. 
If a proper curve $C$ in $X$ satisfies 
$C\cdot (K_X+\Delta)<0$ and $C^2<0$, then $C$ is \Q-Cartier. 
\end{prop}

\begin{proof}
By Theorem~\ref{lcgermsing} and Lemma~\ref{qcartierlemma}, 
$C$ passes through only \Q-factorial points. 
\end{proof}

Thus we complete the proof of the Theorem~\ref{lccont}.
Next, we consider the abundance theorem. 
But this immediately follows from the \Q-factorial case. 

\begin{thm}
Let $(X, \Delta)$ be a proper log canonical surface. 
If $K_X+\Delta$ is nef, then $K_X+\Delta$ is semi-ample. 
\end{thm}

\begin{proof}
Take the minimal resolution and apply Theorem~\ref{abundance} and 
Theorem~\ref{fpabundance}. 
\end{proof}

\section{Relativization}

\subsection{Relative cone theorem}

In this section, we consider the relativization of the cone theorem. 
But this is not difficult by the following proposition. 

\begin{prop}\label{relctt}
Let $\pi:X \to S$ be a proper morphism 
from a normal surface $X$ 
to a variety $S$. 
If $\dim \pi(X) \geq 1$ 
where $\pi(X)$ is the scheme-theoretic image of $\pi$, 
then we have 
$$\overline{NE}(X/S)=NE(X/S)=\sum\limits_{\rm finite}\mathbb{R}_{\geq 0}[C_i].$$ 
Moreover, 
the Stein factorization $\pi:X \overset{\theta}\to T \to S$ satisfies 
one of the following assertions{\em{:}} 
\begin{enumerate}
\item[(1-irr)]{If $\dim \pi(X)=1$ and all fibers of $\theta$ are irreducible, 
then $NE(X/S)=\mathbb{R}_{\geq 0}[C]$ and $C^2=0$. 
In particular, $\rho(X/S)=1$.}
\item[(1-red)]{If $\dim \pi(X)=1$ and $\theta$ has at least one reducible fiber, 
then each $C_i$ has negative self-intersection number.}
\item[(2)]{If $\dim \pi(X)=2$, 
then each $C_i$ has negative self-intersection number.}
\end{enumerate}
\end{prop}

\begin{proof}
Note that $\dim \pi(X)=\dim T$ and $NE(X/S)=NE(X/T)$. 

(1-irr) 
All fibers are numerically equivalent. 
This is what we want to show. 

(1-red) 
By Proposition~\ref{genericintegral}, 
general fibers of $\theta$ are irreducible. 
Therefore, there are only finitely many reducible fibers. 
Since all fibers are numerically equivalent, 
$NE(X/S)$ is generated by the curves in the reducible fibers. 
Because all fibers of $\theta$ are connected, 
curves in reducible fibers have negative 
self-intersection number.  

(2) 
By $\theta_*\mathcal{O}_X=\mathcal{O}_T$, 
we see that $\theta$ is birational. 
Since the exceptional locus is a closed set, 
there are only finitely many curves contracted by $\theta$. 
Each contracted curve has negative self-intersection number. 
\end{proof}

Using this proposition, we obtain the following relative cone theorem. 

\begin{thm}\label{relct}
Let $\pi:X \to S$ be a projective morphism 
from a normal surface $X$ to a variety $S$.
Let $\Delta$ be an effective $\mathbb{R}$-divisor 
such that $K_X+\Delta$ is $\mathbb{R}$-Cartier. 
Let $\Delta=\sum b_iB_i$ be the prime decomposition. 
Let $H$ be an $\mathbb{R}$-Cartier 
$\pi$-ample $\mathbb{R}$-divisor on $X$. 
Then the following assertions hold{\em{:}}

\begin{enumerate}

\item{
$
\neb(X/S)=
\neb(X/S)_{K_X+\Delta \geq 0}+
\sum{\mathbb{R}_{\geq 0}[C_i]}
$.}
\item{
$
\neb(X/S)=
\neb(X/S)_{K_X+\Delta+H \geq 0}+
\sum\limits_{\rm finite}{\mathbb{R}_{\geq 0}[C_i]}
$.}
\item{Each $C_i$ in $(1)$ and $(2)$ is rational or 
$C_i=B_j$ for some $B_j$ with $B_j^2<0$.}
\item{There exists a positive integer $L(X, S, \Delta)$ 
such that each $C_i$ in $(1)$ and $(2)$ 
satisfies $0<-C_i\cdot (K_X+\Delta) \leq L(X, S, \Delta)$.}
\end{enumerate}
\end{thm}

\begin{proof}
If $\dim \pi(X)=0$, then the assertion follows from Theorem~\ref{ct}. 
If $\dim \pi(X) \geq 1$, then 
the assertions (1), (2) and (4) 
immediately follow from Proposition~\ref{relctt}. 
We prove (3). 
Let $C$ be a $(K_X+\Delta)$-negative proper curve which 
generates an extremal ray and $\pi(C)$ is one point. 
We may assume $C\neq B_j$ for all $B_j$ with $B_j^2<0$. 
Take the Stein factorization of $\pi$: 
$$\pi:X\overset{\theta}\to T\to S.$$ 

Let us take the Nagata compactification of $T$ and 
its normalization $\overline{T}$. 
Moreover, take the normalization $\overline{X}$ 
of a compactification of $X \to \overline{T}$. 
We obtain the following commutative diagram. 
$$
\begin{CD}
X @>{\rm open}>{\rm immersion}> \overline{X}\\
@V\theta VV @VV\overline{\theta}V\\
T @>{\rm open}>{\rm immersion}> \overline{T}\\
\end{CD}
$$
In $\overline{X}$, we can apply 
(BBII) in the sense of Definition~\ref{deflogbb} to $C$. 
Then we obtain 
$$p^nC\equiv_{\rm Mum} \alpha C'+Z$$
for a positive integer $n$, 
a non-negative integer $\alpha$, 
a curve $C'$ and a sum of rational curves $Z$. 
We consider the two cases: $\dim T=1$ and $\dim T=2$. 

Assume $\dim T=1$. 
Take an ample divisor $A$ on $\overline{T}$. 
Since $C\cdot \overline{\theta}^*A=0$, 
the prime components of $Z$ must be $\overline{\theta}$-vertical. 
In advance, let $c_0\in C$ be a point, 
in the notation of Definition~\ref{deflogbb}, such that 
$c_0$ is not contained in any curve $C''\neq C$
which is contained in the fiber containing $C$. 
Then, there exists a prime component $Z_j$ of $Z$ with $c_0\in Z_j$. 
Here $Z_j$ must be $C$. 
In particular, $C$ is rational and 
this is what we want to show. 

Assume $\dim T=2$. 
Then, $\overline{T}$ is a proper normal surface. 
Since $\overline{\theta}_*(\alpha C'+Z)\equiv_{\rm Mum}0$, 
each prime component of $Z$ is $\overline{\theta}$-exceptional. 
The remaining proof is the same as the case of $\dim T=1$.
\end{proof}

We give an upper bound $L(X, S, \Delta)$ 
in the case where $\Delta$ is an $\mathbb{R}$-boundary. 

\begin{prop}\label{ext-rat-length2}
Let $\pi:X \to S$ be a projective morphism 
from a normal surface $X$ to a variety $S$.
Let $\Delta$ be an $\mathbb{R}$-boundary such that $K_X+\Delta$ is $\mathbb{R}$-Cartier. 
If $R$ is a $(K_X+\Delta)$-negative extremal ray of $\overline{NE}(X/S)$, 
then $R=\mathbb R_{\geq 0}[C]$ where $C$ is a rational curve such that 
$-(K_X+\Delta)\cdot C\leq 3$.
\end{prop}

\begin{proof}
If $\dim \pi(X)=0,$ then the assertion follows from Proposition~\ref{ext-rat-length1}. 
Thus, we assume $\dim \pi(X)\geq 1.$ 
We can write  $R=\mathbb R_{\geq 0}[C]$ for some curve $C$. 
We show $C$ satisfies the desired properties. 
By $\dim \pi(X)\geq 1$ and Proposition~\ref{relctt}, we see $C^2\leq 0.$ 
Then, by Lemma~\ref{negative-length2}, we have 
$$-(K_X+\Delta)\cdot C\leq 2.$$ 
Since 
$$(K_X+C)\cdot C\leq (K_X+\Delta)\cdot C<0,$$
by Lemma~\ref{non-qcar-rational}, 
we see that $C$ is rational. 
\end{proof}

\subsection{Relative contraction theorem}

In this section, 
we consider the relativization of the contraction theorem. 

\begin{thm}\label{relcont}
Let $\pi:X \to S$ be a projective morphism 
from a normal surface $X$ to a variety $S$.
Let $\Delta$ be an $\mathbb{R}$ divisor. 
Moreover one of the following conditions holds{\em{:}}
\begin{enumerate}
\item[(QF)]{$X$ is \Q-factorial and $\Delta$ is an $\mathbb{R}$-boundary.}
\item[(FP)]{$k=\overline{\mathbb{F}}_p$ and $\Delta$ is an effective $\mathbb{R}$-divisor.}
\item[(LC)]{$(X, \Delta)$ is a log canonical surface.}
\end{enumerate}
Let $R=\mathbb{R}_{\geq 0}[C]$ be a $(K_X+\Delta)$-negative extremal ray 
in $\overline{NE}(X/S)$. 
Then there exists a surjective $S$-morphism $\phi_R: X \to Y$ 
to a variety $Y$ projective over $S$ with the following properties{\em{:}} 
{\em{(1)-(5)}}.
\begin{enumerate}
\item{Let $C'$ be a curve on $X$. 
Then $\phi_R(C')$ is one point iff $[C'] \in R$.}
\item{$(\phi_R)_*(\mathcal{O}_X)=\mathcal{O}_Y$.}
\item{If $L$ is an invertible sheaf with $L\cdot C=0$, 
then $nL=(\phi_R)^*L_Y$ for some invertible sheaf $L_Y$ on $Y$ and 
for some positive integer $n$.}
\item{$\rho(Y/S)=\rho(X/S)-1$.}
\item{If $\dim Y=2$ then 
$Y$ is \Q-factorial $($resp. $(Y, (\phi_R)_*(\Delta))$ 
is log canonical$)$ 
in the case of {\em (QF)} $($resp. {\em (LC)}$)$.}
\end{enumerate}
\end{thm}

These three proofs of (QF), (FP) and (LC) are the same essentially. 
Thus we only prove the case when (QF). 

\begin{proof}
Let $\theta:X \to T$ be the Stein factorization of $\pi$. 
We see that $\dim T=0$, $\dim T=1$ or $\dim T=2$. 
But the case $\dim T=0$ follows from Theorem~\ref{cont}. 
Thus we may assume that $\dim T=1$ or $\dim T=2$. 

Now let us take the compactification.
First, take the Nagata compactification of $T$ and 
its normalization $\overline{T}$. 
Second, take the compactification $\overline{X}$ of $X \to \overline{T}$. 
Moreover, if necessary, replace it by its normalization 
and a resolution of the singular locus in $\overline{X}\setminus X$. 
We obtain the following commutative diagram. 
$$
\begin{CD}
X @>{\rm open}>{\rm immersion}> \overline{X}\\
@V\theta VV @VV\overline{\theta}V\\
T @>{\rm open}>{\rm immersion}> \overline{T}\\
\end{CD}
$$
Then, $\overline{X}$ is projective normal \Q-factorial and 
$\overline{T}$ is proper normal. 
Let $\overline{\Delta}$ be the $\mathbb{R}$-boundary such that 
its restriction to $X$ is $\Delta$ and $\overline{\Delta}$ 
has no prime components contained in $\overline{X}\setminus X$.

Assume $C^2<0$. 
This follows from Theorem~\ref{cont} because 
$C$ is a $(K_{\overline{X}}+\overline{\Delta})$-negative extremal curve 
in the cone of the absolute case $\overline{NE}(\overline{X})$. 

Assume $C^2\geq 0$. 
Then, by Proposition~\ref{relctt}, we see $\rho(X/T)=1$ and $\dim T=1$. 
Set $Y:=T$. 
The assertions (1), (2) and (4) are trivial. 
We want to prove (3). 
Note that all fibers of $\theta$ are irreducible but 
the compactification $\overline{\theta}$ may have 
reducible fiber $G=\sum G_i$. 
Then, by 
$$0>(K_{\overline{X}}+{\overline{\Delta}})\cdot G=
(K_{\overline{X}}+{\overline{\Delta}})\cdot \sum G_i,$$
we obtain 
$0>(K_{\overline{X}}+{\overline{\Delta}})\cdot G_i$ 
for some irreducible component $G_i$ of the fiber $G$. 
Thus, by Theorem~\ref{cont}, 
we may assume that all fibers of $\overline{\theta}$ are irreducible. 
Therefore each fiber $F$ of $\overline{\theta}$ is 
$(K_{\overline{X}}+{\overline{\Delta}})$-negative. 
It is sufficient to prove that $F$ generates 
an extremal ray of $\overline{NE}(\overline{X})$. 
By Theorem~\ref{ct}, 
we have 
$$F \equiv D +\sum r_iC_i,$$ 
where $D \in \neb(\overline{X})_{K_{\overline{X}}+{\overline{\Delta}} \geq 0}$, 
$r_i\in \mathbb{R}_{\geq 0}$ and 
each $C_i$ generates a 
$(K_{\overline{X}}+{\overline{\Delta}})$-negative extremal ray. 
Since $F$ is nef, we have $F\cdot D=F\cdot C_i=0$ for all $i$. 
Here recall that all fibers of $\overline{\theta}$ are irreducible. 
This means that $C_i$ is some fiber with the reduced structure. 
Thus we obtain $F \equiv qC_i$ for some positive number $q$ 
and $F$ generates an extremal ray. 
\end{proof}

Then, we obtain the minimal model program in full generality. 

\begin{thm}[Minimal model Program]\label{relmmp}
Let $\pi:X\to S$ be a projective morphism 
from a normal surface $X$ to a variety $S$. 
Let $\Delta$ be an $\mathbb{R}$-divisor on $X$. 
Assume that one of the following conditions holds{\em{:}}
\begin{enumerate}
\item[(QF)]{$X$ is \Q-factorial and $0\leq \Delta\leq 1$.}
\item[(FP)]{$k=\overline{\mathbb{F}}_p$ and $0\leq\Delta$.}
\item[(LC)]{$(X, \Delta)$ is a log canonical surface.}
\end{enumerate}
Then, there exists a sequence of proper birational morphisms 
\begin{eqnarray*}
&(X, \Delta)=:(X_0, \Delta_0) \overset{\phi_0}\to 
(X_1, \Delta_1) \overset{\phi_1}\to \cdots 
\overset{\phi_{s-1}}\to (X_s, \Delta_s)=:(X^{\dagger}, \Delta^{\dagger})\\
&\,\,\,where\,\,\,(\phi_{i-1})_*(\Delta_{i-1})=:\Delta_i
\end{eqnarray*}
with the following properties. 
\begin{enumerate}
\item{Each $X_i$ is a normal surface, which is projective over $S$.}
\item{Each $(X_i, \Delta_i)$ satisfies {\em (QF)}, {\em (FP)} 
or {\em (LC)} according as the above assumption.}
\item{For each $i$, ${\rm Ex}(\phi_i)=:C_i$ is a proper irreducible curve such that 
$$(K_{X_i}+\Delta_i)\cdot C_i<0$$
and that $C_i$ generates an extremal ray of $\overline{NE}(X/S)$.}
\item{Let $\pi^{\dagger}:X^{\dagger}\to S$ be the $S$-scheme structure morphism. 
$(X^{\dagger}, \Delta^{\dagger})$ satisfies one of the following conditions.
\begin{enumerate}
\item{$K_{X^{\dagger}}+\Delta^{\dagger}$ is $\pi^{\dagger}$-nef.}
\item{There is a projective surjective $S$-morphism $\mu:X^{\dagger} \to Z$ to 
a smooth curve $Z$ such that $Z$ is projective over $S$, 
$\mu_*\mathcal O_{X^{\dagger}}=\mathcal O_Z$, 
$-(K_X^{\dagger}+\Delta^{\dagger})$ is $\mu$-ample and $\rho(X^{\dagger}/Z)=1$.}
\item{$X^{\dagger}$ is a projective surface, 
$-(K_{X^{\dagger}}+\Delta^{\dagger})$ is ample and 
$\rho(X^{\dagger})=1$.}
\end{enumerate}}
\end{enumerate}
In case $({\rm a})$, we say $(X^{\dagger}, \Delta^{\dagger})$ is a minimal model of $(X, \Delta)$ over $S$.\\
In case $({\rm b})$ and $({\rm c})$, we say $(X^{\dagger}, \Delta^{\dagger})$ is a Mori fiber space over $S$. 
\end{thm}

\subsection{Relative abundance theorem}

In this section, 
we consider the relativization of the abundance theorem. 
To descend the problem from the absolute case to the relative case, 
let us consider the following lemma.

\begin{lem}\label{relabundancelemma}
Let $\pi:X \to S$ be a morphism 
from a projective normal $\mathbb{Q}$-factorial surface $X$ 
to a projective variety $S$. 
Let $\Delta$ be an $\mathbb{R}$-boundary on $X$. 
If $K_X+\Delta$ is $\pi$-nef, 
then there exists an ample line bundle $F$ on $S$ such that 
$\Delta+\pi^*(F) \sim_{\mathbb{R}}\Delta'$ 
for some $\mathbb{R}$-boundary $\Delta'$
and $K_X+\Delta'$ is nef.
\end{lem}

\begin{proof}
Take the Stein factorization of $\pi$ 
$$\pi:X \overset{\theta}\to T \overset{\sigma}\to S.$$ 
Take an arbitrary ample line bundle $H$ on $S$. 
Since $\sigma$ is a finite morphism, 
$\sigma^*(H)$ is also ample. 
We may assume that $\sigma^*(H)$ is very ample 
by replacing $H$ with its multiple. 
Note that $\sigma^*(4H)$ is very ample. 
We want to prove that 
$F:=4H$ satisfies the assertion. 
If $\dim T=0$, then the assertion is obvious. 
Thus we can consider the following two cases: 
(1)$\dim T=1$ and (2)$\dim T=2$. \\
(1)Assume $\dim T=1$. 
In this case, $T$ is a smooth projective curve and 
general fibers of $\theta$ are integral 
by Proposition~\ref{genericintegral}. 
Thus, we can take a hyperplane section 
$$P_1+\cdots+P_n=G\in |\sigma^*(4H)|$$ such that 
$P_i\neq P_j$ for all $i\neq j$, 
$\theta^{-1}(P_i)$ is integral for each $i$ and 
$\theta^{-1}(P_i)$ is not a component of $\Delta$ for each $i$. 
Therefore, for an $\mathbb{R}$-boundary $\Delta'$ defined by 
$$\Delta':=\Delta+\theta^*(G),$$ 
$K_X+\Delta'$ is nef by Theorem~\ref{ct} and Proposition~\ref{ext-rat-length1}. \\
(2)Assume $\dim T=2$. 
In this case, $T$ is a normal projective surface and 
$\theta$ is birational. 
By Bertini's theorem, 
we can take an irreducible smooth hyperplane section 
$G\in |\sigma^*(4H)|$ such that 
$\Supp G\cap \theta({\rm Ex}(\theta))=\emptyset$ and 
$G$ is not a component of $\theta_*(\Delta)$. 
Then, $\Delta':=\Delta+\theta^*(G)$ is an $\mathbb{R}$-boundary and 
$K_X+\Delta'$ is nef by Theorem~\ref{ct} and Proposition~\ref{ext-rat-length1}.
\end{proof}

We can prove the relative abundance theorem for 
\Q-factorial surfaces with $\mathbb{R}$-boundary. 

\begin{thm}\label{qfacrelabundance}
Let $\pi:X \to S$ be a projective morphism 
from a normal \Q-factorial surface $X$ 
to a variety $S$. 
Let $\Delta$ be an $\mathbb{R}$-boundary. 
If $K_X+\Delta$ is $\pi$-nef, then 
$K_X+\Delta$ is $\pi$-semi-ample. 
\end{thm}

\begin{proof}
We may assume that $S$ is affine. 
Moreover, by taking Nagata's compactification 
we may assume that $S$ is projective and 
$X$ is projective \Q-factorial. 
Note that the hypothesis of $\pi$-nefness 
may break up by taking the compactification. 
But, by running $(K_X+\Delta)$-minimal model program over $S$, 
we may assume this hypothesis. 
Thus we can apply Lemma~\ref{relabundancelemma}. 
$F$ and $\Delta'$ are the same notations as Lemma~\ref{relabundancelemma}. 
Since $K_X+\Delta'$ is nef, $K_X+\Delta'$ is semi-ample by 
the abundance theorem of the absolute case. 
By $K_X+\Delta' \sim_{\mathbb{R}}K_X+\Delta+\pi^*(F)$, 
$K_X+\Delta$ is $\pi$-semi-ample.
\end{proof}

We obtain the following theorem by applying the same argument. 

\begin{thm}
Let $\pi:X \to S$ be a projective morphism 
from a normal surface $X$ to a variety $S$, 
which are defined over $\overline{\mathbb{F}}_p$. 
Let $\Delta$ be an effective $\mathbb{R}$-divisor. 
If $K_X+\Delta$ is $\pi$-nef, then 
$K_X+\Delta$ is $\pi$-semi-ample. 
\end{thm}

\begin{proof}
We can apply the same proof as Theorem~\ref{qfacrelabundance}. 
\end{proof}

The log canonical case immediately follows from the \Q-factorial case.

\begin{thm}
Let $\pi:X \to S$ be a projective morphism 
from a log canonical surface $(X, \Delta)$
to a variety $S$. 
If $K_X+\Delta$ is $\pi$-nef, then 
$K_X+\Delta$ is $\pi$-semi-ample. 
\end{thm}

\begin{proof}
Take the minimal resolution and apply Theorem~\ref{qfacrelabundance}. 
\end{proof}

We summarize the results obtained in this section. 

\begin{cor}\label{relabundance}
Let $\pi:X\to S$ be a projective morphism 
from a normal surface $X$ to a variety $S$. 
Let $\Delta$ be an $\mathbb{R}$-divisor on $X$. 
Assume that one of the following conditions holds{\em{:}}
\begin{enumerate}
\item[(QF)]{$X$ is \Q-factorial and $0\leq \Delta\leq 1$.}
\item[(FP)]{$k=\overline{\mathbb{F}}_p$ and $0\leq\Delta$.}
\item[(LC)]{$(X, \Delta)$ is a log canonical surface.}
\end{enumerate}
If $K_X+\Delta$ is $\pi$-nef, then $K_X+\Delta$ is $\pi$-semi-ample. 
\end{cor}

\appendix
\def\thesection{A}
\section{Basepoint free theorem}
In this section, we consider the basepoint free theorem. 
First, we prove the following non-vanishing theorem.

\begin{thm}\label{nonvani}
Let $X$ be a projective normal \Q-factorial surface 
and let $\Delta$ be a $\mathbb{Q}$-boundary. 
Let $D$ be a nef Cartier divisor. 
Assume that $D-(K_X+\Delta)$ is nef and big and 
that $(D-(K_X+\Delta))\cdot C>0$ 
for every curve $C \subset \Supp\,\llcorner\Delta\lrcorner$. 
Then $\kappa(X, D) \geq 0$.
\end{thm}

\begin{proof}
If $k=\overline{\mathbb{F}}_p$, 
then the assertion follows from Theorem~\ref{fpbpf}. 
Thus we may assume $k\neq\overline{\mathbb{F}}_p.$

Assume $\kappa(X, D)=-\infty$ and we derive a contradiction.
Let $f: X' \to X$ be the minimal resolution, 
$K_{X'}+\Delta'=f^*(K_X+\Delta)$ and $D'=f^*D$.

\setcounter{step}{0}
\begin{step}
We may assume that $\kappa(X', K_{X'})=-\infty$.

Indeed, we have $\kappa(X', K_{X'}) \leq \kappa(X', K_{X'}+\Delta')
=\kappa(X, K_X+\Delta)=-\infty. $
Note that, if $\kappa(X, K_X+\Delta) \geq 0$, then 
we have 
$\kappa(X, D)=\kappa(X, D-(K_X+\Delta)+(K_X+\Delta)) \geq 0$.
This is what we want to show. 
\end{step}

\begin{step}
In this step, we show $h^2(X', D')=0$.

By Serre duality, we have 
$$h^2(X', D')=h^0(X', K_{X'}-D')\,\,\,\,{\rm and}\,\,$$
$$\kappa(X', K_{X'}-D') \leq 
\kappa(X', K_{X'}+\Delta'-D') = 
\kappa(X, K_X+\Delta-D)=-\infty,$$
because $-(K_X+\Delta-D)$ is nef and big. 
\end{step}

\begin{step}
In this step, we prove that $X'$ is an irrational ruled surface. 

It is sufficient to prove $\chi(\mathcal{O}_{X'}) \leq 0$.
Since $h^0(X', D')=h^2(X', D')=0$, by the Riemann--Roch theorem, 
we obtain 
$$-h^1(X', D')=\chi(\mathcal{O}_{X'})+\frac{1}{2}D'\cdot (D'-K_{X'}). $$ 
Since 
$$D'\cdot (D'-K_{X'})=D\cdot (D-K_X)\,\,\,\,{\rm and}\,\,$$
$$\kappa(X, D-K_X) \geq \kappa(X, D-(K_X+\Delta)) =2, $$
we have $D'\cdot (D'-K_{X'}) \geq 0 $ by the nefness of $D$. 
Therefore we get $0 \geq -h^1(X', D') \geq \chi(\mathcal{O}_{X'})$. 
\end{step}

Let $\pi:X' \to Z$ be its ruling. 
By Theorem~\ref{irratqfac}, $\pi$ factors through $X$.

\begin{step}
We reduce the proof to the case where there is no curve $C$ in $X$ such that 
$D\cdot C=0$ and $C^2<0$. 

Let $C$ be such a curve. 
We have $(K_X+C)\cdot C<0$ by the assumption. 
(Indeed, if $C \subset \Supp\,\llcorner\Delta\lrcorner$, then 
$$-(K_X+C)\cdot C\geq -(K_X+\Delta)\cdot C=(D-(K_X+\Delta))\cdot C>0.$$
If $C \not\subset \Supp\,\llcorner\Delta\lrcorner$, then 
$$-(K_X+C)\cdot C> -(K_X+\Delta)\cdot C=(D-(K_X+\Delta))\cdot C\geq 0.)$$ 
This shows that $C=\mathbb{P}^1$ and $C$ is contractable. 
Moreover this induces a contraction map $g:X \to Y$ 
to a \Q-factorial surface $Y$ and 
the irrationality of $X$ shows that $\pi$ factors through $Y$. 
Let $g_*D=D_Y$ and $g_*(\Delta)=\Delta_Y$. 
Then, we have $K_X+\Delta=g^*(K_Y+\Delta_Y)+aC$ 
for some non-negative rational number $a$. 
Therefore, it is easy to see that 
$Y$ has all the assumptions of $X$. 
\end{step}

\begin{step}
We reduce the proof to the case where $K_X+\Delta$ is not nef. 
In particular, 
there is at least one $(K_X+\Delta)$-negative extremal ray. 

If $K_X+\Delta$ is nef, then 
$D=D-(K_X+\Delta)+(K_X+\Delta)$ is nef and big, 
and this is what we want to show. 
Thus, we may assume that $K_X+\Delta$ is not nef. 
\end{step}

\begin{step}
We reduce the proof to the case where $D \equiv 0$.

The nefness of $D$ and $\kappa(X, D)=-\infty$ show $D^2=0$. 
Since $D$ and $D-(K_X+\Delta)$ is nef, 
we have $(D-(K_X+\Delta))\cdot D=-(K_X+\Delta)\cdot D\geq 0$. 
We consider the two cases: 
$-(K_X+\Delta)\cdot D=0$ or $-(K_X+\Delta)\cdot D>0$. 
If $-(K_X+\Delta)\cdot D=0$, then 
we obtain $D \equiv 0$ 
by the bigness of $D-(K_X+\Delta)$. 
This is what we want to show. 
If $-(K_X+\Delta)\cdot D>0$, then we have $K_X\cdot D<0$. 
Two conditions $K_X\cdot D<0$ and $D^2=0$ mean 
$\kappa(X, D)=1$ by resolution and the Riemann--Roch theorem. 
This case is excluded. 
\end{step}

\begin{step}
By Step~4 and Step~6, there exists no curve $C$ with $C^2<0$. 
By Step~5 and the classification of extremal rays, 
we have $\rho(X) \leq 2$. 
Since there is a surjection $X\to Z$ 
to a curve $Z$, we have $\rho(X)\neq 1.$ 
Thus, we obtain $\rho(X)=2$. 
Here, $-(K_X+\Delta)$ is ample because 
$-(K_X+\Delta)$ is nef and big and Step~4. 
Moreover, by Step~4, 
there are two extremal rays inducing 
the structure of the Mori fiber space to a curve. 
By Proposition~\ref{ext-rat-length1}, 
every extremal ray is generated by a rational curve. 
This contradicts the irrationality of $Z$.
\end{step}
This completes the proof. 
\end{proof}

Using the non-vanishing theorem, we obtain the following basepoint free theorem. 

\begin{thm}\label{bpf}
Let $X$ be a projective normal \Q-factorial surface and 
let $\Delta$ be a $\mathbb{Q}$-boundary. 
Let $D$ be a nef Cartier divisor. 
Assume that $D-(K_X+\Delta)$ is nef and big and 
that $(D-(K_X+\Delta))\cdot C>0$ 
for every curve $C \subset \Supp\,\llcorner\Delta\lrcorner$. 
Then $D$ is semi-ample. 
\end{thm}

\begin{proof}
By Theorem \ref{nonvani}, 
we may assume $\kappa(X, D) \geq 0$. 
But by Proposition~\ref{kappa1}, 
we may assume $\kappa(X, D)=0\,\,or\,\,2$. 
By the same argument as Step~4 in the proof of Theorem \ref{nonvani}, 
we may assume that there is no curve $C$ in $X$ 
with $D\cdot C=0$ and $C^2<0$. 
Thus, if $\kappa(X, D)=2$, then $D$ is ample. 
This is what we want to show. 
Hence the remaining case is $\kappa(X, D)=0$. 
We have linear equivalence to effective divisor $nD \sim \sum d_iD_i$. 
Assume $\sum d_iD_i \neq 0$ and let us get a contradiction. 
Since $D^2=0$ and the nefness of $D$, 
we have $D\cdot D_i=0$ for all $i$. 
Moreover we get $D_i^2 \geq 0$ by the above reduction. 
Then, we obtain $D_i^2=D_i\cdot D=0$.
Since $D-(K_X+\Delta)$ is nef and big, we have 
$$(D-(K_X+\Delta))\cdot D_i=-(K_X+\Delta)\cdot D_i > 0. $$ 
This means $K_X\cdot D_i<0$. 
$D_i^2=0$ and $K_X\cdot D_i<0$ show $\kappa(X, D_i)=1$ 
by taking a resolution and applying the Riemann--Roch theorem. 
This contradicts $\kappa(X, D)=0$.
\end{proof}

The following example teaches us that 
the basepoint free theorem does not hold only under the boundary condition.

\begin{ex}
If $k\neq\fff$, then 
there exist a smooth projective surface $X$ over $k$, 
an elliptic curve $C$ in $X$ and a divisor $D$ 
such that $K_X+C=0$ and the divisor $D=D-(K_X+C)$ is nef and big but not semi-ample.
\end{ex}

\begin{proof}[Construction]
Let $X_0:=\mathbb{P}^2$ and 
let $C_0$ be an elliptic curve in $X_0$. 
Let $P_1,\cdots ,P_{10}$ be ten points which are linearly independent. 
Blowup these 10 points. 
We obtain the surface $X$ and 
let $C$ be the proper transform of $C_0$. 
Then $K_X+C=0$ and 
$C$ is not contractable by Answer~\ref{Hironaka} and its construction. 
On the other hand, take an ample divisor $H$ and 
let $D$ be the divisor $D:=H+qC$ with $(H+qC)\cdot C=0$. 
It is easy to check that $D$ is nef and big. 
Because $C$ is not contractable, $D$ is not semi-ample. 
\end{proof}

We can also prove a basepoint free theorem 
under the following assumption. 

\begin{thm}
Let $X$ be a projective normal \Q-factorial surface and 
let $\Delta$ be a $\mathbb{Q}$-boundary. 
Let $D$ be a nef Cartier divisor. 
Assume that $D-(K_X+\Delta)$ is semi-ample. 
Then $D$ is semi-ample. 
\end{thm}

\begin{proof}
Set $\kappa:=\kappa(X, D-(K_X+\Delta))$. 
There are three cases: (0)$\kappa=0$, (1)$\kappa=1$ and (2)$\kappa=2$. \\
(0)Assume $\kappa=0$. 
By the semi-ampleness, 
we obtain $D-(K_X+\Delta)\sim_{\mathbb{Q}} 0$. 
Thus we can apply the abundance theorem to $D$. 
Then we obtain the desired result. \\
(1)Assume $\kappa=1$. 
By the semi-ampleness, the complete linear system $|n(D-(K_X+\Delta))|$ 
induces a morphism $\sigma:X\to B$ to a smooth projective curve. 
By Proposition~\ref{genericintegral}, 
we can find a boundary 
$$\Delta'\sim_{\mathbb{Q}} D-(K_X+\Delta)$$
such that $\Delta+\Delta'$ is a \Q-boundary. 
Thus we can apply the abundance theorem to $K_X+\Delta+\Delta'$. \\
(2)Assume $\kappa=2$. 
The complete linear system $|n(D-(K_X+\Delta))|$ induces 
a birational morphism $f:X \to Y$ to a normal projective surface. 
Since $n(D-(K_X+\Delta))=f^*(H_Y)$ 
where $H_Y$ is an very ample line bundle on $Y$. 
By Bertini's theorem, we can find a member $G\in|H_Y|$ such that 
$$\Delta+\frac{1}{n}f^*(G)$$
is a boundary. 
Thus we can apply the abundance theorem. 
\end{proof}

\appendix
\def\thesection{B}
\section{Rational singularities}
In this section, 
we consider the relation 
between the minimal model program and the rational singularities. 

\begin{dfn}
Let $X$ be a normal surface and 
let $f:Y \to X$ be a resolution of singularities. 
We say $X$ has at worst rational singularities 
if $R^1f_*\mathcal O_Y=0$. 
This property is independent of the choice of resolutions of singularities. 
\end{dfn}

If $X$ is a normal surface whose singularities are at worst rational, 
then $X$ is \Q-factorial by \cite[Proposition~17.1]{Lipman}. 
Let us give an alternative proof of this result. 

\begin{prop}\label{rational-qfac}
Let $X$ be a normal surface. 
If $X$ has at worst rational singularities, then 
$X$ is \Q-factorial. 
\end{prop}

\begin{proof}
Note that, if $g:Z\to X$ is a proper birational morphism and 
$E$ is a $g$-exceptional curve, then $E\simeq \mathbb P^1.$ 

We may assume $X$ is affine. 
Thus, we may assume $X$ is projective. 
Let $f:Y\to X$ be the minimal resolution. 
Let $E$ be an $f$-exceptional curve. 
By Proposition~\ref{p1cont}, we can contract $E$ and 
we obtain 
$$f:Y\to Y'\overset{f'}\to X.$$
By Proposition~\ref{(3)} and Proposition~\ref{(5)}, 
$X'$ is \Q-factorial. 
Assume $f'$ is not an isomorphism. 
Then, we can take an $f'$-exceptional curve $E'$. 
By the same argument, we can contract $E'$ to a \Q-factorial surface. 
Repeat the same procedure. 
Then, we see $X$ is \Q-factorial. 
\end{proof}

The Kodaira vanishing theorem do not hold 
in positive characteristic. 
But we obtain the following relative vanishing theorem.

\begin{thm}\label{relvanishing}
Let $f:X \to Y$ be a proper birational morphism 
from a smooth surface $X$ to a normal surface $Y$. 
Let $L$ be a line bundle on $X$ such that 
$$L \equiv_f K_X+E+N$$
where $E$ is an effective $f$-exceptional $\mathbb{R}$-boundary and 
$N$ is an $f$-nef $\mathbb{R}$-divisor. 
If $E_i\cdot N>0$ for every curve $E_i$ 
with $E_i\subset \llcorner E\lrcorner$, then $R^1f_*(L)=0$.
\end{thm}

\begin{proof}
See \cite[Section~2.2]{KK}.
\end{proof}

In this paper, 
we often use the contraction of $\mathbb{P}^1$. 
For example, 
the minimal model program of Theorem~\ref{qfacmmp} 
is the composition of the contractions of 
$C \simeq \mathbb{P}^1$ with $(K_X+C)\cdot C< 0$. 
The following theorem shows that 
$R^1$ of such contractions vanish.

\begin{thm}\label{p1vanishing}
Let $g:Y \to Z$ be a proper birational morphism 
between normal surfaces 
such that $C:={\rm Ex}(g)$ is an irreducible curve. 
If $(K_Y+C)\cdot C< 0$, 
then $R^1g_*(\mathcal{O}_Y)=0$.
\end{thm}

\begin{proof}
Let $f:X \to Y$ be the minimal resolution of $Y$ and 
let $C_X$ be the proper transform of $C$. 
Set $K_X+C_X+\Delta_X=f^*(K_Y+C)$. 
Then we have 
\begin{eqnarray*}
-\llcorner \Delta_X\lrcorner&=&K_X+(\{\Delta_X\}+C_X)-f^*(K_Y+C).
\end{eqnarray*}
We apply Theorem~\ref{relvanishing} and we obtain 
\begin{eqnarray*}
R^1(g\circ f)_*\mathcal O_X(-\llcorner \Delta_X\lrcorner) = 0
\end{eqnarray*}
by $-f^*(K_Y+C)\cdot C_X>0.$ \\
If $\llcorner \Delta_X\lrcorner=0$, then we obtain 
$$R^1g_*(\mathcal{O}_Y)=R^1g_*(f_*\mathcal{O}_X) \subset
R^1(g\circ f)_*(\mathcal{O}_X)=0 $$
by the Grothendieck--Leray spectral sequence. 
Thus we may assume $\llcorner \Delta_X\lrcorner\neq 0$. 
Since 
$$0 \to 
\mathcal{O}_X(-\llcorner \Delta_X\lrcorner) \to 
\mathcal{O}_X \to 
\mathcal{O}_{\llcorner \Delta_X\lrcorner} \to 0, $$
we obtain 
\begin{eqnarray*}
&0 \to 
f_*\mathcal{O}_X(-\llcorner \Delta_X\lrcorner) \to 
\mathcal{O}_Y \to 
\mathcal{C} \to 0\\
&\mathcal{C} \subset f_*\mathcal{O}_{\llcorner \Delta_X\lrcorner}
\end{eqnarray*}
where $\mathcal{C}$ is the cokernel of 
$f_*\mathcal{O}_X(-\llcorner \Delta_X\lrcorner) \to 
\mathcal{O}_Y$. 
Since $f_*\mathcal{O}_{\llcorner \Delta_X\lrcorner}$ is a skyscraper sheaf, 
so is $\mathcal{C}$. 
Thus we obtain 
$$
R^1g_*(f_*\mathcal{O}_X(-\llcorner \Delta_X\lrcorner)) \to 
R^1g_*(\mathcal{O}_Y) \to R^1g_*(\mathcal{C})=0.
$$
By the Grothendieck--Leray spectral sequence, 
we obtain 
$$R^1g_*(f_*\mathcal{O}_X(-\llcorner \Delta_X\lrcorner))
\subset R^1(g\circ f)_*\mathcal{O}_X(-\llcorner \Delta_X\lrcorner)=0.$$ 
Therefore, we have $R^1g_*(\mathcal{O}_Y)=0.$ 
\end{proof}

As corollaries, 
we obtain the results on minimal models and canonical models 
for surfaces with rational singularities.

\begin{cor}\label{mmratsing}
Let $\pi:X \to S$ be a projective morphism from 
a normal surface $X$ to a variety $S$. 
Let $\Delta$ be an $\mathbb{R}$-boundary. 
Assume $X$ has at worst rational singularities. 
Then, the following assertions holds. 
\begin{enumerate}
\item{$X$ is \Q-factorial. In particular, by Theorem~\ref{relmmp}, 
we can run a $(K_X+\Delta)$-minimal model program over $S$ 
\begin{eqnarray*}
&(X, \Delta)=:(X_0, \Delta_0) \overset{\phi_0}\to 
(X_1, \Delta_1) \overset{\phi_1}\to \cdots 
\overset{\phi_{s-1}}\to (X_s, \Delta_s)\\
&\,\,\,where\,\,\,(\phi_{i-1})_*(\Delta_{i-1})=:\Delta_i.
\end{eqnarray*}}
\item{Each $X_i$ has at worst rational singularities. }
\end{enumerate}
\end{cor}

\begin{proof}
(1) follows from Proposition~\ref{rational-qfac}. 
Each extremal contraction in a minimal model program of $(X, \Delta)$ 
satisfies the condition of Theorem~\ref{p1vanishing}. 
This implies (2). 
\end{proof}

\begin{cor}\label{cmratsing}
Let $\pi:X \to S$ be a projective morphism from 
a normal surface $X$ to a variety $S$. 
Let $\Delta$ be an $\mathbb{R}$-divisor such that 
$0\leq \Delta<1$. 
If $X$ has at worst rational singularities and $K_X+\Delta$ is $\pi$-big, 
then the canonical model of $(X, \Delta)$ over $S$ 
has at worst rational singularities.
\end{cor}

\begin{proof}
By Corollary~\ref{mmratsing} we may assume that 
$K_X+\Delta$ is $\pi$-nef and $\pi$-big. 
If $(K_X+\Delta)\cdot C=0$ for some curve $C$ such that $\pi(C)$ is ${\rm one\,\,point}$, 
then $(K_X+C)\cdot C<0$ because $0\leq \Delta<1$. 
Therefore we can contract this curve $C$ and 
$C$ satisfies the condition of Theorem~\ref{p1vanishing}. 
Repeat this procedure and 
we obtain the required assertion. 
\end{proof}



\begin{thebibliography}{Koll\'ar-Kov\'acs}

\bibitem[Artin]{Artin}
{M.~Artin},
Some numerical criteria for contractability of curves in algebraic surfaces, 
{Amer. J. Math. {\textbf{84}} (1962), 485--496}.


\bibitem[B\u{a}descu1]{Badescu}
{L.~B\u{a}descu},
{\em Algebraic Surfaces}, 
{Universitext, Springer-Verlag, New York, 2001}.

\bibitem[B\u{a}descu2]{Badescu2}
{L.~B\u{a}descu},
{On some contractibility criteria of curves on surfaces}, 
{Sitzungsberichte der Berliner Mathematischen Gesellschaft (1.1.1997-31.12.2000), 2001, 41--52.}

\bibitem[Birkar]{Birkar}
{C.~Birkar},
On existence of log minimal models II,
{to appear in J. Reine Angew. Math.}.


\bibitem[BMII]{BMII}
{E.~Bombieri, D.~Mumford},
Enriques' classification of surfaces in char $p$: II,
{in {\em Complex analysis and algebraic geometry}, 
Cambridge  Univ. Press, {\textbf{35}}, 23--42 (1977)}.


\bibitem[BMIII]{BMIII}
{E.~Bombieri, D.~Mumford},
Enriques' classification of surfaces in char $p$: III,
{Invent. Math. {\textbf{35}}, 197--232 (1976)}.


\bibitem[CMM]{CMM}
{P.~Cascini, J.~M\textsuperscript{c}Kernan, M.~Mustata},
The augmented base locus in positive characteristic, 
{to appear in Proc. of the Edinburgh Math. Soc., volume dedicated to V.V. Shokurov (2012).}




\bibitem[Frey-Jarden]{FJ}
{G.~Frey, M.~Jarden},
Approximation theory and the rank of abelian varieties 
over large algebraic fields,
{Proceedings of the London Mathematical Society, {\textbf{28}} (1974), 112--128}.


\bibitem[Fujino1]{Fujino1}
{O.~Fujino},
Fundamental theorems for the log minimal model program, 
{Publ. Res. Inst. Math. Sci. {\textbf{47}} (2011), no. 3, 727--789.}


\bibitem[Fujino2]{Fujino}
{O.~Fujino},
Minimal model theory for log surfaces, 
{Publ. Res. Inst. Math. Sci. {\textbf{48}} (2012), no. 2, 339--371.}


\bibitem[FT]{FT}
{O.~Fujino, H.~Tanaka},
On log surfaces, 
{Proc. Japan Acad. Ser. A Math. Sci. {\textbf{88}} (2012), no. 8, 109--114.}.




\bibitem[Fujita]{Fujita}
{T.~Fujita},
Fractionally logarithmic canonical rings of algebraic surfaces, 
{J. Fac. Sci. Univ. Tokyo Sect. IA Math. {\textbf{30}} (1984), no 3 685--696}.








\bibitem[Hartshorne]{Hartshorne}
{R.~Hartshorne},
{\em Algebraic Geometry.},
{Grad. Texts in Math., no {\textbf{52}}, Springer-Verlag, NewYork, 1977}.





\bibitem[Kawamata]{Kawamata}
{Y.~Kawamata},
Semistable minimal models of threefolds in positive or mixed characteristic,
{J. Alg. Geom. {\textbf{3}} (1994), 463--491}.


\bibitem[KMM]{KMM}
{Y.~Kawamata, K.~Matsuda, K.~Matsuki},
Introduction to the Minimal Model Program,
{volume {\textbf{10}} of adv. Stud. Pure Math., 
283--360. Kinokuniya--North--Holland, 1987}.


\bibitem[Keel1]{Keel}
{S.~Keel},
Basepoint freeness for nef and big linebundles in positive characteristic, 
{Ann. Math, {\textbf{149}} (1999), 253--286}.


\bibitem[Keel2]{Keel2}
{S.~Keel},
Polarized pushouts over finite fields, 
{Comm. Alg. {\textbf{31}} (2003), 3955--3982.}.



\bibitem[Koll\'ar1]{Kollar}
{J.~Koll\'ar},
Extremal rays on smooth threefolds, 
{Ann. Sci. Ec. Norm. Sup., {\textbf{24}} (1991), 339--361}.

\bibitem[Koll\'ar2]{Kollar2}
{J.~Koll\'ar},
{\em Rational curves on algebraic varieties}, 
{Ergebnisse der Math., Springer, Vol. {\textbf{32}}, 1996}.





\bibitem[Koll\'ar-Kov\'acs]{KK}
{J.~Koll\'ar,~S.~Kov\'acs},
Birational geometry of log surfaces, 
{preprint}.


\bibitem[Koll\'ar-Mori]{KM}
{J.~Koll\'ar,~S.~Mori},
{\em Birational geometry of algebraic varieties},
{Cambrigde Tracts in Mathematics, Vol. {\textbf{134}}, 1998}.


\bibitem[Lipman]{Lipman}
{J. Lipman},
{Rational singularities with applications to algebraic surafaces and 
unique factorization},
{Publ. Math. IHES, {\textbf{36}} (1969), 195--297.}


\bibitem[Ma\c{s}ek]{Masek}
{V. Ma\c{s}ek},
Kodaira--Iitaka and numerical dimensions of algebraic surfaces 
over the algebraic closure of a finite field, 
{Rev. Roumaine Math. Pures Appl. {\textbf{38}} (1993) 7--8, 679--685}.


\bibitem[Matsumura]{Matsumura}
{H.~Matsumera},
{\em Commutative Algebra.},
{W. A. Benjamin Co., New York (1970)}.


\bibitem[Miyanishi]{Miyanishi}
{M.~Miyanishi},
{\em Non-complete algebraic surfaces},
{Lecture Notes Mathematics, {\textbf{857}}. Springer-Verlag, Berlin-New York, (1981)}.


\bibitem[Mori1]{Mori1}
{S.~Mori},
Projective manifolds with ample tangent bundles, 
{Ann. Math, {\textbf{110}} (1979), 593--606}.


\bibitem[Mori2]{Mori2}
{S.~Mori},
Threefolds whose canonical bundles are not numerically effective,
{Ann. Math. {\textbf{116}} (1982), 133--176}.


\bibitem[Mumford1]{Mumford1}
{D. Mumford},
The topology of normal singularities of an algebraic surface and a criterion for simplicity, 
{IHES Publ. Math. {\textbf{9}} (1961), 5–22.}.


\bibitem[Mumford2]{Mum}
{D. Mumford},
Enriques' classification of surfaces in char $p$: I, 
in Global Analysis (Papers in honor of K. Kodaira), 
{Princeton Univ. Press (1969) 325--339}.


\bibitem[Raynaud]{Raynaud}
{M.~Raynaud},
Contre-exemple au "vanishing theorem" 
en caract\'eristique $p>0$, 
{C. P. Ramanujam --- A tribute, 
Studies in Math., {\textbf{8}} (1978), 273--278}.






\bibitem[Sakai]{Sakai}
{F. Sakai},
Classification of normal surfaces, 
{Algebraic geometry, Bowdoin, 1985 (Brunswick, Maine, 1985), 
451--465, Proc. Sympos. Pure Math., {\textbf{46}}; Part 1, Amer. Math. Soc., Providence, RI, 1987}.


\bibitem[T]{T}
{H. Tanaka},
The X-method for klt surfaces in positive characteristic, 
{preprint (2012)}.


\bibitem[Totaro]{Totaro}
{B. Totaro},
Moving codimension-one subvarieties over finite fields, 
{Amer. J. Math. {\textbf{131}}(2009), 1815--1833}.



\end{thebibliography}
\end{document}